\documentclass[onepage,12pt]{article}

\usepackage[utf8]{inputenc}
\usepackage{amsmath}
\usepackage{amsthm}
\usepackage{amssymb}
\usepackage{times}
\usepackage{bm}
\usepackage[numbers]{natbib}
\usepackage{amsfonts}
\usepackage{amssymb}
\usepackage{authblk}
\usepackage{mathtools}
\usepackage{dsfont}
\usepackage{xcolor}
\usepackage{array}
\usepackage{hyperref}
\usepackage{booktabs}
\usepackage{caption}
\usepackage{tikz}
\usepackage{chngpage}

\usetikzlibrary{positioning}
\usepackage{graphicx}
\usepackage[left=1.45in, right=1in, top=01in, bottom=1in, includefoot, textwidth = 5.77in, textheight=9.4in]{geometry}

\DeclarePairedDelimiterX{\norm}[1]{\lVert}{\rVert}{#1}
\DeclarePairedDelimiterX{\abs}[1]{\lvert}{\rvert}{#1}

\newcommand\restr[2]{{
  \left.\kern-\nulldelimiterspace 
  #1 
  \right|_{#2} 
  }}

  \newcommand{\RomanNumeralCaps}[1]
    {\MakeUppercase{\romannumeral #1}}

\newcommand{\bbP}{\mathbb{P}}

\newcommand{\bbR}{\mathbb{R}}
\newcommand{\bbC}{\mathbb{C}}
\newcommand{\bbE}{\mathbb{E}}
\newcommand{\bbN}{\mathbb{N}}
\newcommand{\calH}{\mathcal{H}}
\newcommand{\calD}{\mathcal{D}}
\newcommand{\calB}{\mathcal{B}}
\newcommand{\calX}{\mathcal{X}}
\newcommand{\calY}{\mathcal{Y}}
\newcommand{\calF}{\mathcal{F}}
\newcommand{\calP}{\mathcal{P}}
\newcommand{\calN}{\mathcal{N}}
\newcommand{\calR}{\mathcal{R}}
\newcommand{\calI}{\mathcal{I}}
\newcommand{\calGP}{\mathcal{GP}}

\newcommand{\MMD}{\text{MMD}}
\newcommand{\eMMD}{\emph{MMD}}

\newcommand{\Tr}{\text{Tr}}

\newtheorem{theorem}{Theorem}
\newtheorem{lemma}{Lemma}
\newtheorem{definition}{Definition}
\newtheorem{corollary}{Corollary}
\newtheorem{proposition}{Proposition}

\newtheorem{remark}{Remark}

\title{A Kernel Two-Sample Test for Functional Data}
\author[1]{George Wynne}
\author[1,2]{Andrew B. Duncan}
\affil[1]{Imperial College London, Department of Mathematics}
\affil[2]{The Alan Turing Institute}

\date{}

\begin{document}

\makeatletter\let\Title\@title\makeatother
\maketitle

\abstract{We propose a nonparametric two-sample test procedure based on Maximum Mean Discrepancy (MMD) for testing the hypothesis that two samples of functions have the same underlying distribution, using kernels defined on function spaces. This construction is motivated by a scaling analysis of the efficiency of MMD-based tests for datasets of increasing dimension. Theoretical properties of kernels on function spaces and their associated MMD  are established and employed to ascertain the efficacy of the newly proposed test, as well as to assess the effects of using functional reconstructions based on discretised function samples.  The theoretical results are demonstrated over a range of synthetic and real world datasets.}


\section{Introduction} \label{sec:introduction}
Nonparametric two-sample tests for equality of distributions are widely studied in statistics, driven by applications in goodness-of-fit tests, anomaly and change-point detection and clustering. Classical examples of such tests include the Kolmogorov-Smirnov test \citep{kolmogorov1933sulla,smirnov1948table,schmid1958kolmogorov} and Wald-Wolfowitz runs test \citep{wald1940test} with subsequent multivariate extensions \citep{friedman1979multivariate}.

Due to advances in the ability to collect large amounts of real time or spatially distributed data there is a need to develop statistical methods appropriate for functional data, where each data sample is a discretised function. Such data has been studied for decades in the Functional Data Analysis (FDA) literature \citep{horvath2012inference,Hsing2015} particularly in the context of analysing populations of time series, or in statistical shape analysis \citep{mardia1989statistical}.  More recently, due to this modern abundance of functional data, increased study has been made in the machine learning literature for algorithms suited to such data \citep{Berrendero2020,Chevyrev2018,Kadri2016,CARMELI2010,Zhang2012}.

In this paper we consider the case where the two probability distributions being compared are supported over a real, separable Hilbert space, for example $L^2(\mathcal{D})$ with $\mathcal{D} \subset \mathbb{R}^d$ and we have discretised observations of the function samples. If the samples consist of evaluations of the functions over a common mesh of points, then well-known methods for nonparametric two-sample testing for vector data can be used directly. Aside from the practical issue that observations are often made on irregular meshes for each different sample there is also the issue of degrading performance of classical tests as mesh size increases, meaning the observed vectors are high dimensional.  As is typical with nonparametric two-sample tests, the testing power will degenerate rapidly with increasing data dimension.  We therefore seek to better understand how to develop testing methods which are not strongly affected by the mesh resolution, exploiting intrinsic statistical properties of the underlying functional probability distributions.

In the past two decades kernels have seen a surge of use in statistical applications \citep{Muandet2017,Gretton2012,Sutherland2016,Borgwardt2006}. In particular, kernel based two-sample testing \citep{Gretton2012,Gretton2007} has become increasingly popular.  These approaches are based on a distance on the space of probability measures known as \emph{Maximum Mean Discrepancy}. Given two probability distributions $P$ and $Q$, a kernel $k$ is employed to construct a mapping known as the \emph{mean embedding}, of the two distributions into an infinite dimensional Reproducing Kernel Hilbert Space (RKHS).  The MMD between $P$ and $Q$, denoted $\mbox{MMD}_k(P,Q)$ is given by the RKHS norm of the difference between the two embeddings, and defines a pseudo-metric on the space of probability measures. This becomes a metric if $k$ is  \emph{characteristic}, see Section \ref{sec:RKHS}.  By the kernel trick, MMD simplifies to a closed form, up to expectations, with respect to $P$ and $Q$, which can be estimated unbiasedly using Monte Carlo simulations.

A major advantage of kernel two-sample tests is that they can be constructed on any input space which admits a well-defined kernel, including Riemannian manifolds \citep{Pelletier2005}, as well as discrete structures such as graphs \citep{ShaweTaylor2004} and strings \citep{Grtner2003}. The flexibility in the choice of kernel is one of the strengths of MMD-based testing, where \emph{a priori} knowledge of the structure of the underlying distributions can be encoded within the kernel to improve the sensitivity or specificity of the corresponding test.  The particular choice of kernel strongly influences the efficiency of the test, however a general recipe for constructing a good kernel is still an open problem.  On Euclidean spaces, radial basis function (RBF) kernels are often used, i.e. kernels of the form $k(x,y) = \phi(\gamma^{-1} \lVert x - y \rVert_{2})$, where $\lVert \cdot \rVert_{2}$ is the Euclidean norm,  $\phi:\mathbb{R}_{+}\rightarrow\mathbb{R}_{+}$ is a function and $\gamma > 0$ is the \emph{bandwidth}.  Numerous kernels used in practice belong to this class of kernels, including the Gaussian kernel $\phi(r) = e^{-r^2/2}$, the Laplace kernel $\phi(r) = e^{-r}$ and others including  the rational quadratic kernel, the Matern kernel and the multiquadric kernel.  The problem of selecting the bandwidth parameter to maximise test efficiency over a particular input space has been widely studied.  One commonly used strategy is the \emph{median heuristic} where the bandwidth is chosen to be the median of the inter-sample distance.  Despite its popularity, there is only limited understanding of the median heuristic, with some notable exceptions.  In \citet{ramdas2015decreasing,Ramdas2014_mean_shift} the authors investigate the diminishing power of the kernel two-sample test using a Gauss kernel for distributions of white Gaussian random vectors with increasing dimension,  demonstrating that under appropriate alternatives, the power of the test will decay with a rate  dependent on the relative scaling of $\gamma$ with respect to dimension.  Related to kernel based tests are energy distance tests \citep{szekely2003statistics,szekely2004testing}, the relationship was made clear in \citet{Sejdinovic2013}.

There has been relatively little work on understanding the theoretical properties of kernels on function spaces. A Gauss type kernel on  $L^{2}([0,1])$ was briefly considered in \citet[Example 3]{Christmann2010}.  Recently, in \citet{Chevyrev2018} a kernel was defined on the Banach space of paths on $[0,1]$ of unbounded 1-variation, using a novel approach based on path signatures, demonstrating that this is a characteristic kernel over the space of such paths.  The associated MMD has been employed as a loss function to train models generating stochastic processes \citep{Kidger2019}.  Furthermore, in \citet{nelsen2020random} the authors propose an extension of the random Fourier feature kernel of \citet{NIPS2007_3182} to the setting of an infinite dimensional Banach space, with the objective of regression between Banach spaces.  This paper will build on aspects of these works, but with a specific emphasis on two-sample testing for functional data.

Two-sample testing in function spaces has received much attention in FDA and is studied in a variety of contexts.  Broadly speaking there are two classes of methods.  The first approach seeks to initially reduce the problem to a finite dimensional problem through a projection onto a finite orthonormal basis within the function space,  typically using principal components, and then makes use of standard multivariate two-sample tests \citep{benko2009common,lopes2011more}.  The second approach poses a two-sample test directly on function space \citep{aue2018detecting,bucchia2017change,horvath2014testing,Pomann2016,cabana2017permutation}.  Many of these works construct the test on the Hilbert space $L^2(\mathcal{D})$ using the $L^2(\calD)$ norm as the testing statistic.   A priori, it is not obvious why this norm will be well suited to the testing problem, in general.  Investigation into the impact of the choice of distance in distanced based tests for functional data has been studied in the literature \citep{chen2014optimally,chakraborty2019new,Zhu2019} and a distance other than $L^{2}(\calD)$ for the functional data was advocated. This motivates the investigation into kernels which involve distances other than $L^{2}(\calD)$ in their formulation. In many works, the two-sample tests are designed to handle a specific class of discrepancy, such as a shift in mean, such as \citet{Horvth2012} and \citet{zhang2010two}, or a shift in covariance structure \citep{Panaretos2010,ferraty2003curves,FREMDT2012}.

This paper has two main aims. First, to naturally generalise the finite dimensional theory of kernels to real, separable Hilbert spaces to establish kernels that are characteristic, identify their RKHS and establish topological properties of the associated MMD. In particular the proof of characterticness builds upon the spectral methods introduced in \citet{Sriperumbudur2010} and the weak convergence results build upon \citet{Simon-Gabriel2018}. Second, we apply such kernels to the two-sample testing problem and analyse the power of the tests as well as the statistical impact of performing the tests using data reconstructed from discrete functional observations. 

The specific contributions are as follows.

\begin{enumerate}
    \item For Gaussian processes, we identify a scaling of the Gauss kernel bandwidth with mesh-size which results in testing power which is asymptotically independent of mesh-size, under mean-shift alternatives.  In the scaling of vanishing mesh-size we demonstrate that the associated kernel converges to a kernel over functions.  
	\item Motivated by this, we construct a family of kernels defined on real, separable Hilbert spaces and identify sufficient conditions for the kernels to be characteristic, when MMD metrises the weak topology and provide an explicit construction of the reproducing kernel Hilbert space for a Gauss type kernel.
	\item Using these kernels we investigate the statistical effect of using reconstructed functional data in the two-sample test.
	\item We numerically validate our theory and compare the kernel based test with established two-sample tests from the functional data analysis literature. 
\end{enumerate}

The remainder of the paper is as follows. Section \ref{sec:inf_measures} covers preliminaries of modelling random functional data such as the Karhunen-Lo\`eve expansion and Gaussian measures. Section \ref{sec:RKHS} recalls some important properties of kernels and their associated reproducing kernel Hilbert spaces, defines maximum mean discrepancy and the kernel two-sample test. Section \ref{sec:scaling} outlines the scaling of test power that occurs when an increasingly finer observation mesh is used for functional data. Section \ref{sec:kernel} defines a broad class of kernels and offers an integral feature map interpretation as well as outlining when the kernels are characteristic, meaning the two-sample test is valid. Section \ref{sec:inf_dim_MMD} highlights the statistical impact of fitting curves to discretised functions before performing the test. A relationship between MMD and weak convergence is highlighted and closed form expressions for the MMD and mean-embeddings when the distributions are Gaussian processes are given. Section \ref{sec:implementation} provides multiple examples of choices for the kernel hyper parameters and principled methods of constructing them. Section \ref{sec:numerics} contains multiple numerical experiments validating the theory in the paper, a simulation is performed to validate the scaling arguments of Section \ref{sec:scaling} and synthetic and real data sets are used to compare the performance of the kernel based test against existing functional two-sample tests.  Concluding remarks and thoughts about future work are provided in Section \ref{sec:conclusion}.

\section{Hilbert Space Modelling of Functional Data} \label{sec:inf_measures}

In this paper we shall follow the Hilbert space approach to functional data analysis and use this section to outline the required preliminaries \citep{cuevas2014partial,Hsing2015}. Before discussing random functions we establish notation for families of operators that will be used extensively.  Let $\calX$ be a real, separable Hilbert space with inner product $\langle\cdot,\cdot\rangle_{\calX}$ then $L(\calX)$ denotes the set of bounded linear maps from $\calX$ to itself, $L^{+}(\calX)$ denotes the subset of $L(\calX)$ of operators that are self-adjoint (also known as symmetric) and non-negative, meaning $\langle Tx,y\rangle_{\calX}\geq 0\:\forall   x,y\in\calX$. The subset of $L^{+}(\calX)$ of trace class operators is denoted $L^{+}_{1}(\calX)$ and by the spectral theorem \citep[Theorem A.5.13]{Steinwart2008} such operators can be diagonalised. This means for every $T\in L^{+}_{1}(\calX)$ there exists an orthonormal basis of eigenfunctions $\{e_{n}\}_{n=1}^{\infty}$ in $\mathcal{X}$ such that $Tx = \sum_{n=1}^{\infty}\lambda_n \langle x, e_n \rangle_{\calX} e_n$, where $\{\lambda_n\}_{n=1}^{\infty}$ are non-negative eigenvalues and the trace satisfies $\Tr(T) = \sum_{n=1}^{\infty}\lambda_n < \infty$. When the eigenvalues are square summable the operator is called Hilbert-Schmidt and the Hilbert-Schmidt norm is $\norm{T}_{HS}^{2} = \sum_{n=1}^{\infty}\lambda_{n}^{2}$. 

We now outline the Karhunen-Lo\`eve expansion of stochastic processes. Let $x(\cdot)$ be a stochastic process in $\calX = L^{2}([0,1])$, note the following will hold for a stochastic process taking values in any real, separable Hilbert space but we focus on $L^{2}([0,1])$ since it is the most common setting for functional data. Suppose that the pointwise covariance function $\bbE[x(s)x(t)] = k(s,t)$ is continuous, then the mean function $m(t) = \bbE[X(t)]$ is also in $\calX$. Define the covariance operator  $C_{k}\colon \calX\rightarrow \calX$ associated with $X$ by $C_{k}y(t) = \int_{0}^{1}k(s,t)y(s)ds$. Then $C_{k}\in L^{+}_{1}(\calX)$ and denote the spectral decomposition $C_{k}y= \sum_{n=1}^{\infty}\lambda_n \langle y, e_n \rangle_{\calX} e_n$. The Karhunen-Lo\`eve (KL) expansion \citep[Theorem 11.4]{Sullivan2015} provides a characterisation of the law of the process $x(\cdot)$ in terms of an infinite-series expansion.  More specifically, we can write $x(\cdot) \sim m + \sum_{n=1}^{\infty} \lambda_n^{1/2}\eta_n e_n(\cdot)$,
where $\lbrace \eta_n \rbrace_{n=1}^\infty$ are unit-variance uncorrelated random variables.  Additionally, Mercer's theorem  \citep{Steinwart2012} provides an expansion of the covariance as \sloppy$k(s, t) = \sum_{n=1}^\infty \lambda_n e_n(s)e_n(t)$ where the convergence is uniform.

An important case of random functions are Gaussian processes \citep{Rasmussen2006}. Given a kernel $k$, see Section \ref{sec:RKHS}, and a function $m$ we say $x$ is a Gaussian process with mean function $m$ and covariance function $k$ if for every finite collection of points $\{s_{n}\}_{n=1}^{N}$ the random vector $(x(s_{1}),\ldots,x(s_{N}))$ is a multivariate Gaussian random variable with mean vector $(m(s_{1}),\ldots,m(s_{N}))$ and covariance matrix $k(s_{n},s_{m})_{n,m=1}^{N}$. The mean function and covariance function completely determines the Gaussian process.  We write $x\sim\calGP(m,k)$ to denote the Gaussian process with mean function $m$ and covariance function $k$. If $x\sim\calGP(0,k)$ then in the Karhunen-Lo\`eve representation $\eta_{n}\sim\calN(0, 1 )$ and the $\eta_{n}$ are all independent. 

Gaussian processes that take values in $\calX$ can be associated with Gaussian measures on $\calX$. Gaussian measures are natural generalisations of Gaussian distributions on $\bbR^{d}$ to infinite dimensional spaces, which are defined by a mean element and covariance operator rather than a mean vector and covariance matrix, for an introduction see \citet[Chapter 1]{DaPrato2006}. Specifically $x\sim\calGP(m,k)$ can be associated with the Gaussian measure $N_{m, C_{k}}$ with mean $m$ and covariance operator $C_{k}$, the covariance operator associated with $k$ as outlined above. Similarly given any $m \in \mathcal{X}$ and $C \in L_1^+(\mathcal{X})$ then there exists a Gaussian measure $N_{m,C}$ with mean $m$ and covariance operator $C$ \citep[Theorem 1.12]{DaPrato2006}. In fact, the Gaussian measure $N_{m,C}$ is characterised as the unique probability measure on $\calX$ with Fourier transform $\widehat{N}_{m,C}(y) = \exp(i\langle m,y\rangle_{\calX}-\frac{1}{2}\langle Cy, y\rangle_{\calX})$. Finally, if $C$ is injective then a Gaussian measure with covariance operator $C$ is called non-degenerate and has full support on $\calX$ \citep[Proposition 1.25]{DaPrato2006}.

\section{Reproducing Kernel Hilbert Spaces and Maximum Mean Discrepancy} \label{sec:RKHS}

This section will outline what a kernel and a reproducing kernel Hilbert space is with examples and associated references. Subsection \ref{subsec:RKHS_def} defines kernels and RKHS, Subsection \ref{subsec:reps_of_MMD} defines MMD and the corresponding estimators and Subsection \ref{subsec:test_procedure} outlines the testing procedure.

\subsection{Kernels and Reproducing Kernel Hilbert Spaces}\label{subsec:RKHS_def}

Given a nonempty set $\calX$ a kernel is a function $k\colon\calX\times\calX\rightarrow\bbR$ which is symmetric, meaning $k(x,y)= k(y,x)$, for all $x,y\in\calX$, and positive definite, that is, the matrix $\lbrace k(x_n, x_m);\, n, m \in \lbrace 1, \ldots, N \rbrace \rbrace$ is positive semi-definite, for all $\lbrace x_n \rbrace_{n=1}^N \subset \calX$ and for $N \in \mathbb{N}$.  
For each kernel $k$ there is an associated Hilbert space of functions over $\calX$ known as the reproducing kernel Hilbert space (RKHS) denoted $\calH_{k}(\calX)$ \citep{Berlinet2004,Steinwart2008,Fasshauer2014}. RKHSs have found numerous applications in function approximation and inference for decades since their original application to spline interpolation \citep{wahba1990spline}.  Multiple detailed surveys exist in the literature \citep{saitoh2016theory,Paulsen2016}.  The RKHS associated with $k$ satisfies the following two properties i). $k(\cdot,x)\in \calH(\calX)$ for all $x\in\calX$ ii). $\langle f,k(\cdot,x)\rangle_{\calH(\calX)} = f(x)$ for all $x \in \calX$ and $f \in \calH(\calX)$. The latter is known as the reproducing property. The RKHS is constructed from the kernel in a natural way. The linear span of a kernel $k$ with one input fixed $\calH_{0}(\calX) = \left\{\sum_{n=1}^{N}a_{n}k(\cdot,x_{n})\colon N\in\bbN, \{a_{n}\}_{n=1}^{N}\subset\bbR, \{x_{n}\}_{n=1}^{N}\subset\calX\right\}$ is a pre-Hilbert space equipped with the following inner product $\langle f,g\rangle_{\calH_{0}(\calX)} = \sum_{n=1}^{N}\sum_{m=1}^{M}a_{n}b_{m}k(x_{n},y_{m})$
where $f = \sum_{n=1}^{N}a_{n}k(\cdot,x_{n})$ and $g = \sum_{m=1}^{M}b_{m}k(\cdot,y_{m})$. The RKHS $\calH_{k}(\calX)$ of $k$ is then obtained from $\calH_{0}(\calX)$ through completion. More specifically $\calH_{k}(\calX)$ is the set of functions which are pointwise limits of Cauchy sequences in $\calH_{0}(\calX)$ \citep[Theorem 3]{Berlinet2004}. The relationship between kernels and RKHS is one-to one, for every kernel the RKHS is unique and for every Hilbert space of functions such that there exists a function $k$ satisfying the two properties above it may be concluded that the $k$ is unique and a kernel. This result is known as the Aronszajn theorem \citep[Theorem 3]{Berlinet2004}. 

A kernel $k$ on $\mathcal{X} \subseteq \mathbb{R}^d$ is said to be \emph{translation invariant} if it can be written as $k(x,y) = \phi(x-y)$ for some $\phi$. Bochner's theorem, Theorem \ref{thm:bochner} in the Appendix, tells us that if $k$ is continuous and translation invariant then there exists a Borel meaure on $\calX$ such that $\hat{\mu}_{k}(x-y) = k(x,y)$ and we call $\mu_{k}$  the spectral measure of $k$. The spectral measure is an important tool in the analysis of kernel methods and shall become important later when discussing the two-sample problem. 

\subsection{Maximum Mean Discrepancy}\label{subsec:reps_of_MMD}

Given a kernel $k$ and associated RKHS $\calH_k(\calX)$ let $\calP$ be the set of Borel probability measures on $\calX$ and assuming $k$ is measurable define $\calP_k \subset \cal{P}$ as the set of all $P \in \calP_k$ such that \sloppy$\int k(x,x)^{\frac{1}{2}} dP(x) < \infty$. Note that $\calP_{k} = \calP$ if and only if $k$ is bounded \citep[Proposition 2]{Sriperumbudur2010} which is very common in practice and shall be the case for all kernels considered in this paper.  For $P, Q \in \calP_k$ we define the \emph{Maximum Mean Discrepancy} denoted $\text{MMD}_{k}(P,Q)$ as follows $\text{MMD}_{k}(P,Q) = \sup_{\norm{f}_{\calH_{k}(\calX)}\leq 1}\left\lvert \int fdP - \int fdQ\right\rvert$. This is an \emph{integral probability metric} \citep{muller1997integral,Sriperumbudur2010} and without further assumptions defines a pseudo-metric on $\calP_k$, which permits the possibility that $\text{MMD}_{k}(P,Q) = 0$ but $P \neq Q$. 

We introduce the \emph{mean embedding} $\Phi_k P$ of $P \in \calP_k$ into $\calH_k(\calX)$ defined by $\Phi_k P = \int k(\cdot, x)dP(x)$.  This can be viewed as the mean in $\calH_{k}(\calX)$ of the function $k(x,\cdot)$ with respect to $P$ in the sense of a Bochner integral \citep[Section 2.6]{Hsing2015}. Following \citet[Section 2]{Sriperumbudur2010} this allows us to write
\begin{align}
	\text{MMD}_{k}(P,Q)^{2} & = \left(\sup_{\norm{f}_{\calH_{k}(\calX)}\leq 1}\left\lvert\int f dP-\int fdQ\right\rvert\right)^{2}  \nonumber\\
	& = \left(\sup_{\norm{f}_{\calH_{k}(\calX)}\leq 1}\lvert\langle \Phi_{k}P-\Phi_{k}Q,f\rangle\rvert\right)^{2} \nonumber\\
	& = \norm{\Phi_{k}P-\Phi_{k}Q}_{\calH_{k}(\calX)}^{2}. \label{eq:MMD_HS}
\end{align}
The crucial observation which motivates the use of MMD as an effective measure of discrepancy is that the supremum can be eliminated using the reproducing property  of the inner product \citep[Section 2]{Sriperumbudur2010}.  This yields the following closed form representation
\begin{align}
	\text{MMD}_{k}(P,Q)^{2} &  = \int \int k(x,x')dP(x)dP(x') + \int\int k(y,y')dQ(y)dQ(y')\nonumber \\
	&\qquad-2\int\int k(x,y)dP(x)dQ(y). \label{eq:MMD_integral_version}
\end{align}

It is clear that $\MMD_{k}$ is a metric over $\calP_{k}$ if and only if the map $\Phi_{k}\colon\calP_{k}\rightarrow\calH_{k}(\calX)$ is injective. Given a subset $\mathfrak{P}\subseteq\calP_{k}$, a kernel is \emph{characteristic to $\mathfrak{P}$} if the map $\Phi_{k}$ is injective over $\mathfrak{P}$. In the case that $\mathfrak{P} = \calP$ we just say that $k$ is characteristic. Various works have provided sufficient conditions for a kernel over finite dimensional spaces to be characteristic \citep{Sriperumbudur2010,Sriperumbudur2011,Simon-Gabriel2018}.

Given independent samples $X_n = \lbrace x_i \rbrace_{i=1}^n$ from $P$ and $Y_m = \lbrace y_i \rbrace_{i=1}^m$ from $Q$ we wish to estimate $\text{MMD}_k(P, Q)^{2}$.  A number of estimators have been proposed.  For clarity of presentation we shall assume that $m = n$, but stress that all of the following can be generalised to situations where the two data-sets are unbalanced.  Given samples $X_n$ and $Y_n$, the following U-statistic is an unbiased estimator of $\text{MMD}^2_k(P, Q)^{2}$
\begin{align}
	\widehat{\text{MMD}}_{k}(X_{n},Y_{n})^{2}\coloneqq \frac{1}{n(n-1)}\sum_{i\neq j}^{n}h(z_{i},z_{j}),\label{eq:MMD_h_representation}
\end{align}
where $z_{i} = (x_{i},y_{i})$ and $h(z_{i},z_{j}) = k(x_{i},x_{j}) + k(y_{i},y_{j}) - k(x_{i},y_{j})-k(x_{j},y_{i})$.  This estimator can be evaluated in $O(n^2)$ time.  An unbiased linear time estimator proposed in \citet{Wittawat2017} is given by 
\begin{align}
	\widehat{\text{MMD}}_{k,\text{lin}}(X_{n},Y_{n})^{2}\coloneqq \frac{2}{n}\sum_{i=1}^{n/2}h(z_{2i-1},z_{2i}), \label{eq:linear_time_est}
\end{align}
where it is assumed that $n$ is even.  While the cost for computing $\widehat{\text{MMD}}_{k,\text{lin}}(X_{n},Y_{n})^{2}$ is only $O(n)$ this comes at the cost of reduced efficiency, i.e. $\mbox{Var}(\widehat{\text{MMD}}_{k}(X_{n},Y_{n})^{2}) <  \mbox{Var}(\widehat{\text{MMD}}_{k,\text{lin}}(X_{n},Y_{n})^{2})$, see for example \citet{Sutherland2019}.  Various probabilistic bounds have been derived on the error between the estimator and $\MMD_{k}(P,Q)^{2}$  \citep[Theorem 10, Theorem 15]{Gretton2012}.

\subsection{The Kernel Two-Sample Test}\label{subsec:test_procedure}

Given independent samples $X_n = \lbrace x_i \rbrace_{i=1}^n$ from $P$ and $Y_n = \lbrace y_i \rbrace_{i=1}^n$ from $Q$ we seek to test the hypothesis $H_0\colon P = Q$ against the alternative hypothesis $H_1\colon P \neq Q$ without making any distributional assumptions.   The \emph{kernel two-sample test} of \citet{Gretton2012}  employs an estimator of MMD as the test statistic.  Indeed, fixing a characteristic kernel $k$, we reject $H_0$ if $\widehat{\MMD}_{k}(X_{n},Y_{n})^{2} > c_{\alpha}$, where $c_{\alpha}$ is a threshold selected to ensure a false-positive rate of $\alpha$.   While we do not have a closed-form expression for $c_{\alpha}$, it can be estimated using a permutation bootstrap.  More specifically, we randomly shuffle $X_n \cup Y_n$, split it into two data sets $X_{n}'$ and $Y_{n}'$, from which  $\widehat{\MMD}_{k}(X_{n}',Y_{n}')^{2}$ is calculated. This is repeated numerous times so that an estimator of the threshold $\hat{c}_{\alpha}$ is then obtained as the $(1-\alpha)$-th quantile of the resulting empirical distribution.  The same test procedure may be performed using the linear time MMD estimator as the test statistic.

The efficiency of a test is characterised by its false-positive rate $\alpha$ and its its false-negative rate $\beta$.  The power of a test is a measure of its ability to correctly reject the null hypothesis.  More specifically,  fixing $\alpha$, and obtaining an estimator $\hat{c}_{\alpha}$ of the threshold, we define the power of the test at $\alpha$ to be $\mathbb{P}(n\widehat{\MMD}_k(X_n,Y_n)^2\geq \hat{c}_{\alpha})$. Invoking the central limit theorem for U-statistics \citep{Serfling1980} we can quantify the decrease in variance of the unbiased MMD estimators, asymptotically as $n\rightarrow \infty$.

\begin{theorem}{\citep[Corollary 16]{Gretton2012}}\label{thm:clt_mmd_test}
Suppose that $\mathbb{E}_{x\sim P,y\sim Q}[h^2(x,y)] < \infty$.  Then under the alternative hypothesis $P\neq Q$, the estimator $\widehat{\eMMD}_{k}(X_{n},Y_{n})^{2}$ converges in distribution to a Gaussian as follows
$$
    \sqrt{n}\left(\widehat{\eMMD}_{k}(X_{n},Y_{n})^{2} - \eMMD_k^2(P, Q)\right) \xrightarrow{D} \mathcal{N}(0, 4\xi_{1}),\quad n\rightarrow \infty,
$$
where $\xi_{1} = \emph{Var}_{z}\left[\bbE_{z'}[h(z,z')]\right]$.  An analogous result holds for the linear-time estimator, with $\xi_{2} = \emph{Var}_{z,z'}\left[h(z,z')\right]$ instead of $\xi_{1}$.
\end{theorem}

In particular, under the conditions of Theorem \ref{thm:clt_mmd_test}, for large $n$, the power of the test will satisfy the following asymptotic result
\begin{align}
\mathbb{P}\left(n\widehat{\MMD}_{k}(X_{n},Y_{n})^{2} > \widehat{c}_{\alpha} \right) \approx \Phi\left(\sqrt{n}\frac{\MMD_k(P, Q)^2}{2\sqrt{\xi_{1}}} - \frac{c_\alpha}{2\sqrt{n \xi_{1}}} \right),\label{eq:first_power_surr}
\end{align}
where $\Phi$ is the CDF for a standard Gaussian distribution and $\xi_{1} = \text{Var}_{z}\left[\bbE_{z'}[h(z,z')]\right]$.  The analogous result for the linear-time estimator holds with $\xi_{2} = \text{Var}_{z,z'}\left[h(z,z')\right]$ instead of $\xi_{1}$ \citep{Ramdas2014_mean_shift,Liu2020}. This suggests that the test power can be maximised by maximising \sloppy${\MMD_k(P, Q)^2}/{\sqrt{\xi_{1}}}$ which can be seen as a signal-to-noise-ratio \citep{Liu2020}.  It is evident from previous works that the properties of the kernel will have a very significant impact on the power of the test, and methods have been proposed for increasing test power by optimising the kernel parameters using the signal-to-noise-ratio as an objective \citep{Sutherland2016,Ramdas2014_mean_shift,Liu2020}.

\section{Resolution Independent Tests for Gaussian Processes}\label{sec:scaling}

To motivate the construction of kernel two-sample tests for random functions, in this section we will consider the case where the samples $X_n$ and $Y_n$ are independent realisations of two Gaussian processes, observed along a regular mesh $\Xi_{N} = \lbrace t_1, \ldots, t_N\rbrace$ of $N$ points in $\mathcal{D}$ where $\calD\subset\bbR^{d}$ is some compact set. Therefore $N$ will be the dimension of the observed vectors. To develop ideas, we shall focus on a  mean-shift alternative, where the underlying Gaussian processes are given by $\mathcal{GP}(0, k_{0})$ and $\mathcal{GP}(m, k_{0})$ respectively, where $k_{0}$ is a covariance function, and $m\in L^{2}(\calD)$ is the mean function. We use the subscript on $k_{0}$ to distinguish it from the kernel $k$ we use to perform the test. We will use the linear time test due to easier calculations. This reduces to a multivariate two-sample hypothesis test problem on $\mathbb{R}^N$, with samples  $X_n = \lbrace x_i \rbrace_{i=1}^n$ from $P=\mathcal{N}(0, \Sigma)$ and $Y_n = \lbrace y_i \rbrace_{i=1}^n$ from $Q = \mathcal{N}(m_{N}, \Sigma)$, where $\Sigma_{i,j} = k_{0}(t_i, t_j)$ for $i,j=1,\ldots,N$ and $m_{N} = (m(t_1), \ldots, m(t_N))^\top$.

We consider applying a two-sample kernel test as detailed in Section \ref{sec:RKHS}, with a Gaussian kernel $k(x,y) = \exp(-\frac{1}{2}\gamma_{N}^{-2}\lVert x - y \rVert_{2}^2)$ on $\bbR^{N}$ where $\gamma_{N}$ may depend on $N$. The large $N$ limit was studied in \citet{ramdas2015decreasing} but not in the context of functional data. This motivates the question whether there is a scaling of $\gamma_{N}$ with respect to $N$ which, employing the structure of the underlying random functions, guarantees that the statistical power remains independent of the mesh size $N$.  To better understand the influence of bandwidth on power, we use the signal-to-noise ratio as a convenient proxy, and study its behaviour in the large $N$ limit. We say the mesh $\Xi_{N}$ satisfies the Riemann scaling property if  $\frac{1}{N}\norm{m_{N}}_{2}^{2} = \frac{1}{N}\sum_{i=1}^{N}m(t_{i})^{2}\rightarrow\int_{\calD}m(t)^{2}dt = \norm{m}_{L^{2}(\calD)}^{2}$ as $N\rightarrow\infty$ for all $m\in L^{2}(\calD)$, this will be used in the next result to characterise the signal-to-noise ratio from the previous subsection. 

\begin{proposition}\label{prop:scaling}
    Let $P, Q$ be as above with $
    \Xi_{N}$ satisfying the Riemann scaling property and $\gamma_{N}=\Omega(N^{\alpha})$ with $\alpha > 1/2$ then if $k_{0}(s,t) = \delta_{st}$
    \begin{align}\label{eq:scaling_id}
        \frac{\eMMD_{k}(P,Q)^{2}}{\sqrt{\xi_{2}}} \sim \frac{\sqrt{N}\norm{m}_{L^{2}(\calD)}^{2}}{2\sqrt{1+\norm{m}_{L^{2}(\calD)}^{2}}},
    \end{align}
    and if $k_{0}$ is continuous and bounded  then
     \begin{align}\label{eq:scaling_ratio}
        \frac{\eMMD_{k}(P,Q)^{2}}{\sqrt{\xi_{2}}} \sim \frac{\norm{m}_{L^{2}(\calD)}^{2}}{2\sqrt{\norm{C_{k_{0}}}_{HS}^{2}+\norm{C^{1/2}_{k_{0}}m}_{L^{2}(\calD)}^{2}}},
    \end{align}
    where $\sim$ means asymptotically equal in the sense that the ratio of the left and right hand side converges to one as $N\rightarrow\infty$.
\end{proposition}

The proof of this result is in the Appendix and generalises \citet{ramdas2015decreasing} by considering non-identity $\Sigma$. The way this ratio increases with $N$, the number of observation points, in the white noise case makes sense since each observation is revealing new information about the signal as the noise is independent at each observation. On the other hand the non-identity covariance matrix means the noise is not independent at each observation and thus new information is not obtained at each observation point. Indeed the stronger the correlations, meaning the slower the decay of the eigenvalues of the covariance operator $C_{k_{0}}$, the smaller this ratio shall be since the Hilbert-Schmidt norm in the denominator will be larger. 

It is important to note that the ratio in the right hand sides of \eqref{eq:scaling_ratio} and \eqref{eq:scaling_id} are independent of the choice of $\alpha$ once $\alpha > 1/2$ meaning that once greater than $1/2$ this parameter will be ineffective for obtaining greater testing power. The next subsection discusses how $\alpha = 1/2$ provides a scaling resulting in kernels defined directly over function spaces, facilitating other methods to gain better test power.  

\subsection{Kernel Scaling}
Proposition \ref{prop:scaling} does not include the case $\gamma_{N}=\Theta (N^{1/2})$ however it can be shown that the ratio does not degenerate in this case, see Theorem \ref{thm:MMD_two_GPs} and Theorem \ref{thm:variance_of_estimator_mean_shift}. In fact, the two different scales of the ratio, when $\Sigma$ is the identity matrix or a kernel matrix, still occur. This is numerically verified in Section \ref{sec:numerics}. 

Suppose $\gamma_{N} = \gamma_{0}N^{1/2}$ for some $\gamma_{0}\in\bbR$ and one uses a kernel of the form $k(x,y) = f(\gamma_{N}^{-2}\norm{x-y}_{2}^{2})$ over $\bbR^{N}$ for some continuous $f$. Suppose now though that our inputs shall be $x_{N},y_{N}$, discretisations of functions $x,y\in L^{2}(\calD)$ observed on a mesh $\Xi_{N}$ that satisfies the Riemann scaling property. Then as the mesh gets finer we observe the following scaling
\begin{align*}
    k(x_{N},y_{N}) = f(\gamma_{N}^{-2}\norm{x_{N}-y_{N}}_{2}^{2})\xrightarrow{N\rightarrow\infty} f(\gamma_{0}^{-2}\norm{x-y}_{L^{2}(\calD)}^{2}).
\end{align*}
Therefore the kernel, as the discretisation resolution increases, will converge to a kernel over $L^{2}(\calD)$ where the Euclidean norm is replaced with the $L^{2}(\calD)$ norm. For example the Gauss kernel would become $\exp(-\gamma_{0}^{-2}\norm{x-y}_{L^{2}(\calD)}^{2})$. 

This scaling coincidentally is similar to the scaling of the widely used \emph{median heuristic} defined as 
\begin{align}
\label{eq:median_heuristic}
{\gamma}^{2} = \text{Median}\big\{\norm{a-b}_{2}^{2}\colon a,b\in \{x_{i}\}_{i=1}^{n}\cup\{y_{i}\}_{i=1}^{m}, a\neq b\big\},
\end{align}
where $\{x_{i}\}_{i=1}^{n}$ are the samples from $P$, $\{y_{i}\}_{i=1}^{n}$ samples from $Q$. It was not designed with scaling in mind however in \citet{Ramdas2014_mean_shift} it was noted that it results in a $\gamma^{2} = \Theta(N)$ scaling for the mean shift, identity matrix case. The next lemma makes this more precise by relating the median of the squared distance to its expectation. 

\begin{lemma}\label{lem:GP_median_heur}
	Let $P = \mathcal{N}(\mu_1, \Sigma_1)$  and $Q = \mathcal{N}(\mu_2, \Sigma_2)$ be independent Gaussian distributions on $\mathbb{R}^N$ then $\bbE_{x\sim P,y\sim Q}[\norm{x-y}_{2}^{2}] = \emph{Tr}(\Sigma_{1}+\Sigma_{2}) + \norm{\mu_{1}-\mu_{2}}_{2}^{2}$ and 
	\begin{align*}
	    \left\lvert\frac{\emph{Median}_{x\sim P,y\sim Q}[\norm{x-y}_{2}^{2}]}{\bbE_{x\sim P,y\sim Q}[\norm{x-y}_{2}^{2}]}-1\right\rvert \leq \sqrt{2}\left(1-\frac{\norm{\mu_{1}-\mu_{2}}_{2}^{4}}{(\emph{Tr}(\Sigma_{1}+\Sigma_{2}) + \norm{\mu_{1}-\mu_{2}}_{2}^{2})^{2}}\right)^{\frac{1}{2}},
	\end{align*}
	in particular if $P,Q$ are discretisations of Gaussian processes \sloppy{
	$\calGP(m_{1},k_{1}),\calGP(m_{2},k_{2})$} on a mesh $\Xi_{N}$ of $N$ points satisfying the Riemann scaling property over some compact $\calD\subset\bbR^{d}$  with $m_{1},m_{2}\in L^{2}(\calD)$ and $k_{1},k_{2}$ continuous then $\bbE_{x\sim P,y\sim Q}[\norm{x-y}_{2}^{2}]\sim N(\emph{Tr}(C_{k_{1}}+C_{k_{2}})+\norm{m_{1}-m_{2}}_{L^{2}(\calD)}^{2})$ and as $N\rightarrow\infty$ the right hand side of the above inequality converges to
	\begin{align*}
	   \sqrt{2}\left(1-\frac{\norm{m_{1}-m_{2}}_{L^{2}(\calD)}^{4}}{(\emph{Tr}(C_{k_{1}}+C_{k_{2}}) + \norm{m_{1}-m_{2}}_{L^{2}(\calD)}^{2})^{2}}\right)^{\frac{1}{2}}.
	\end{align*}
\end{lemma}

The above lemma does not show that the median heuristic results in $\gamma_{N} = \gamma_{0}N^{1/2}$ but relates it to the expected squared distance which does scale directly as $\gamma_{0}N^{1/2}$. Therefore investigating the properties of such a scaling is natural. 

Since $L^{2}(\calD)$ is a real, separable Hilbert space when using kernels defined directly over $L^{2}(\calD)$ in later sections we can leverage the theory of probability measures on such Hilbert spaces to deduce results about the testing performance of such kernels. In fact, we shall move past $L^{2}(\calD)$ and obtain results for kernels over arbitrary real, separable Hilbert spaces. Note that a different scaling of $\gamma_{N}$ would not result in such a scaling of the norm to $L^{2}(\calD)$ so such theory cannot be applied.

\section{Kernels and RKHS on Function Spaces} \label{sec:kernel}

For the rest of the paper, unless specified otherwise, for example in Theorem \ref{thm:admissible_char}, the spaces $\calX,\calY$ will be real, separable Hilbert spaces with inner products and norms $\langle\cdot,\cdot\rangle_{\calX},\langle\cdot,\cdot\rangle_{\calY},\norm{\cdot}_{\calX},\norm{\cdot}_{\calY}$. We adopt the notation in Section \ref{sec:inf_measures} for various families of operators. 

\subsection{The Squared-Exponential \texorpdfstring{$T$}{} kernel}
Motivated by the scaling discussions in Section \ref{sec:scaling} we define a kernel that acts directly on a Hilbert space. 

\begin{definition}\label{def:SE_T}
	For $T\colon\calX\rightarrow\calY$ the squared-exponential $T$ kernel (SE-$T$) is defined as 
	\begin{align*}
		k_{T}(x,y) = e^{-\frac{1}{2}\norm{T(x)-T(y)}_{\calY}^{2}}.
	\end{align*}
\end{definition}

We use the name squared-exponential instead of Gauss because the SE-$T$ kernel is not always the Fourier transform of a Gaussian distribution whereas the Gauss kernel on $\bbR^{d}$ is, which is a key distinction and is relevant for our proofs. Lemma \ref{lem:pos_def} in the Appendix assures us this function is a kernel. This definition allows us to adapt results about the Gauss kernel on $\bbR^{d}$ to the SE-$T$ kernel since it is the natural infinite dimensional generalisation. For example the following theorem characterises the RKHS of the SE-$T$ kernel for a certain choice of $T$, as was done in the finite dimensional case in \citet{Minh2009}. Before we state the result we introduce the infinite dimensional generalisation of a multi-index, define $\Gamma$ to be the set of summable sequences indexed by $\bbN$ taking values in $\bbN\cup\{0\}$ and for $\gamma\in\Gamma$ set $\lvert\gamma\rvert = \sum_{n=1}^{\infty}\gamma_{n}$, so $\gamma\in\Gamma$ if and only if $\gamma_{n}=0$ for all but finitely many $n\in\bbN$ meaning $\Gamma$ is a countable set. We set $\Gamma_{n} = \{\gamma\in\Gamma\colon\lvert\gamma\rvert = n\}$ and the notation $\sum_{\lvert\gamma\rvert \geq 0}$ shall mean $\sum_{n=0}^{\infty}\sum_{\gamma\in\Gamma_{n}}$ which is a countable sum. 

\begin{theorem}\label{thm:gauss_rkhs}
	Let $T\in L^{+}(\calX)$ be of the form $Tx = \sum_{n=1}^{\infty}\lambda_{n}^{1/2}\langle x,e_{n}\rangle_{\calX} e_{n}$ with convergence in $\calX$ for some orthonormal basis $\{e_{n}\}_{n=1}^{\infty}$ and bounded positive coefficients $\{\lambda_{n}\}_{n=1}^{\infty}$ then the RKHS of the SE-$T$ kernel is
	\begin{align*}
		\calH_{k_{T}}(\calX)  = \left\{ F(x) = e^{-\frac{1}{2}\norm{Tx}_{\calX}^{2}}\sum_{\lvert\gamma\rvert\geq 0}w_{\gamma}x^{\gamma}\colon \sum_{\lvert\gamma\rvert \geq 0}\frac{\gamma !}{\lambda^{\gamma}}w_{\gamma}^{2} < \infty \right\}, 
	\end{align*}
	where  $x^{\gamma} = \prod_{n=1}^{\infty}x_{n}^{\gamma_{n}}$, $x_{n}=\langle x,e_{n}\rangle_{\calX}$,  $\lambda^{\gamma} = \prod_{n=1}^{\infty}\lambda_{n}^{\gamma_{n}}$ and $\gamma ! = \prod_{n=1}^{\infty}\gamma_{n}!$ and $\calH_{k_{T}}(\calX)$ is equipped with the inner product $\langle F,G\rangle_{\calH_{k_{T}}(\calX)} = \sum_{\lvert\gamma\rvert \geq 0}\frac{\gamma !}{\lambda^{\gamma}}w_{\gamma}v_{\gamma}$ where $F(x) = e^{-\frac{1}{2}\norm{Tx}_{\calX}^{2}}\sum_{\lvert\gamma\rvert \geq 0}w_{\gamma}x^{\gamma}$, $G(x) = e^{-\frac{1}{2}\norm{Tx}_{\calX}^{2}}\sum_{\lvert\gamma\rvert \geq 0}v_{\gamma}x^{\gamma}$.
\end{theorem}

\begin{remark}
In the proof of Theorem \ref{thm:gauss_rkhs} an orthonormal basis of $\calH_{k_{T}}(\calX)$ is given which resembles the infinite dimensional Hermite polynomials which are used throughout infinite dimensional analysis and probability theory, for example see \citet[Chapter 10]{DaPrato2002} and \citet[Chapter 2]{Nourdin2009}. In particular they are used to define Sobolev spaces for functions over a real, separable Hilbert space \citep[Theorem 9.2.12]{DaPrato2002} which raises the interesting and, as far as we are aware, open question of how $\calH_{k_{T}}(\calX)$ relates to such Sobolev spaces for different choices of $T$. 
\end{remark}

For the two-sample test to be valid we need the kernel to be characteristic meaning the mean-embedding is injective over $\calP$, so the test can tell the difference between any two probability measures. To understand the problem better we again leverage results regarding the Gauss kernel on $\bbR^{d}$, in particular the proof in \citet[Theorem 9]{Sriperumbudur2010} that the Gauss kernel on $\bbR^{d}$ is characteristic. This uses the fact that the Gauss kernel on $\bbR^{d}$ is the Fourier transform of a Gaussian distribution on $\bbR^{d}$ whose full support implies the kernel is characteristic. By choosing $T$ such that the SE-$T$ kernel is the Fourier transform of a Gaussian measure on $\calX$ that has full support we can use the same argument.

\begin{theorem}\label{thm:trace_class_char}
	Let $T\in L^{+}_{1}(\calX)$ then the SE-$T$ kernel is characteristic if and only if $T$ is injective.
\end{theorem} 

This is dissatisfyingly limiting since $T\in L^{+}_{1}(\calX)$ is a restrictive assumption, for example it does not include $T=I$ the identity operator. We shall employ a limit argument to reduce the requirements on $T$. To this end we define admissible maps.

\begin{definition}
A map $T\colon\calX\rightarrow\calY$ is called admissible if it is Borel measurable, continuous and injective.
\end{definition}

The next result provides a broad family of kernels which are characteristic. It applies for $\calX$ being more general than a real, separable Hilbert space. A Polish space is a separable, completely metrizable topological space. Multiple examples of admissible $T$ are given in Section \ref{sec:implementation} and are examined numerically in Section \ref{sec:numerics}.

\begin{theorem}\label{thm:admissible_char}
	Let $\calX$ be a Polish space, $\calY$ a real, separable Hilbert space and $T$ an admissible map then the SE-$T$ kernel is characteristic.
\end{theorem}

Theorem \ref{thm:admissible_char} generalises Theorem \ref{thm:trace_class_char}. A critical result used in the proof is the Minlos-Sazanov theorem, detailed as Theorem \ref{thm:minlos-sazanov} in the Appendix, which is an infinite dimensional version of Bochner's theorem. The result allows us to identify spectral properties of the SE-$T$ kernel which are used to deduce characteristicness. 

\subsection{Integral Kernel Formulation}\label{subsec:integral_kernel}
Let $k_{0}\colon\bbR\times\bbR\rightarrow\bbR$ be a kernel, $C\in L^{+}_{1}(\calX)$ and $N_{C}$ the corresponding mean zero Gaussian measure on $\calX$ and define $k_{C,k_{0}}\colon\calX\times\calX\rightarrow\bbR$ as follows
\begin{align*}
	k_{C,k_{0}}(x,y)\coloneqq \int_{\calX}k_{0}\left(\langle x,h\rangle_{\calX},\langle y,h\rangle_{\calX}\right)dN_{C}(h).
\end{align*} 
Consider the particular case where $k_0(s,t) = \langle \Phi(s), \Phi(t)\rangle_{\mathcal{F}}$, where $\Phi:\mathbb{R} \rightarrow \mathcal{F}$ is a continuous feature map, mapping into a Hilbert space $(\mathcal{F}, \langle\cdot,\cdot\rangle_{\mathcal{F}})$, which will typically be $\mathbb{R}^F$ for some $F\in\bbN$. In this case the functions $x\rightarrow \Phi(\langle x, h\rangle_{\calX})$ can be viewed as $\mathcal{F}$--valued random features for each $h\in\calX$ randomly sampled from $N_{C}$, and $k_{C, k_0}$ is very similar to the random feature kernels considered in \citet{nelsen2020random} and \citet{bach2017equivalence}.  Following these previous works, we may completely characterise the RKHS of this kernel, the result involves $L^{2}_{N_{C}}(\calX;\calF)$ which is the space of equivalence classes of functions from $\calX$ to $\calF$ that are square integrable in the $\calF$ norm with respect to $N_{C}$ and $L^{2}(\calX)\coloneqq L^{2}(\calX;\calF)$. 

\begin{proposition}
    Suppose that $\psi(x,h) = \Phi(\langle x, h\rangle_{\calX})$ satisfies $\psi \in L^2_{N_{C}\times N_{C}}(\calX \times \calX; \mathcal{F})$ then
    the RKHS defined by the kernel $k_{C,k_0}$ is given by
    $$
       \calH_{k_{C, k_0}}(\calX) =  \left\lbrace \int \langle v(h), \psi(\cdot,h)\rangle_{\mathcal{F}}dN_C(h)\colon v \in L^2_{N_C}(\calX; \mathcal{F})\right\rbrace \subset L^2_{N_C}(\calX).
    $$
\end{proposition}
The proof of this result is an immediate generalization of the real-valued case given in \citet{nelsen2020random}. Using the spectral representation of translation invariant kernels we can provide conditions for $k_{C, k_0}$ to be a characteristic kernel.

\begin{proposition}\label{prop:integral_kernel}
	If $k_{0}$ is a kernel over $\bbR\times\bbR$ then $k_{C,k_{0}}$ is a kernel over $\calX\times\calX$. If $C$ is injective and $k_{0}$ is also continuous and translation invariant with spectral measure $\mu$ such that there exists an interval $(a,b)\subset\bbR$ with $\mu(U) > 0$ for every open subset $U\subset(a,b)$ then $k_{C,k_{0}}$ is characteristic. 
\end{proposition}

For certain choices of $T$ the SE-$T$ kernel falls into a family of {integral kernels}.  Indeed, if $k_{0}(x,y) = \cos(x-y)$ then $k_{C,k_{0}}$ is the SE-$C^{\frac{1}{2}}$ kernel
\begin{align*}
	k_{C,k_{0}}(x,y)   = \widehat{N}_{C}(x-y)&= e^{-\frac{1}{2}\norm{x-y}_{C}^{2}}  = e^{-\frac{1}{2}\sum_{n=1}^{\infty}\lambda_{n}(x_{n}-y_{n})^{2}},
\end{align*} where $\norm{x-y}_{C}^{2} = \langle C(x-y),x-y\rangle_{\calX}$, $\{\lambda_{n}\}_{n=1}^{\infty}$ are the eigenvalues of $C$ and $x_{n} = \langle x,e_{n}\rangle$ are the coefficients with respect to the eigenfunction basis $\{e_{n}\}_{n=1}^{\infty}$ of $C$.

Secondly, let $\gamma >0$ and assume $C$ is non-degenerate and set $k_{0}$ to be the complex exponential of $\gamma$ multiplied the by white noise mapping associated with $C$, see \citet[Section 1.2.4]{DaPrato2006}, then $k_{C,k_{0}}$ is the SE-$\gamma I$ kernel
\begin{align}
	k_{C,k_{0}}(x,y) = k_{\gamma I}(x,y) =  e^{-\frac{\gamma}{2}\norm{x-y}_{\calX}^{2}},\label{eq:id_kernel}
\end{align}
Note that $k_{\gamma I}$ is not the Fourier transform of any Gaussian measure on $\calX$ \citep[Proposition 1.2.11]{Maniglia2004} which shows how the integral kernel framework is more general than only using the Fourier transform of Gaussian measures to obtain kernels, as was done in Theorem \ref{thm:trace_class_char}.

The integral framework can yield non-SE type kernels. Let $N_{1}$ be the measure associated with the Gaussian distribution $\calN(0,1)$ on $\bbR$, $C$ be non-degenerate and $k_{0}(x,y) = \widehat{N}_{1}(x-y)$ then we have
\begin{equation}\label{eq:IMQ_kernel}
\begin{aligned}
  k_{C,k_{0}}(x,y) &= \int_{\calX}\int_{\bbR}e^{iz\langle h,x-y\rangle}dN_{1}(z)dN_{C}(h) = \left({\norm{x-y}_{C}^{2}+1}\right)^{-\frac{1}{2}}. 
\end{aligned}
\end{equation}

\begin{definition}\label{def:IMQ_T}
	For $T\colon\calX\rightarrow\calY$ the inverse multi-quadric $T$ kernel (IMQ-$T$) is defined as 
		$k_{T}(x,y) = \left({\norm{T(x)-T(y)}_{\calY}^{2}+1}\right)^{-1/2}.$
\end{definition}

By using Proposition \ref{prop:integral_kernel} we immediately obtain that if $T\in L^{+}_{1}(\calX)$ and $T$ is non-degenerate then the IMQ-$T$ kernel is characteristic. But by the same limiting argument as Theorem \ref{thm:admissible_char} and the integral kernel formulation of IMQ-$T$ we obtain a more general result. 

\begin{corollary}\label{cor:IMQ_characteristic}
	Under the same conditions as Theorem \ref{thm:admissible_char} the IMQ-$T$ kernel is characteristic.
\end{corollary}

\section{MMD on Function Spaces} \label{sec:inf_dim_MMD}

In Section \ref{sec:kernel} we derived kernels directly over function spaces that were characteristic, meaning that the MMD induced by them is a metric on $\calP(\calX)$. Therefore a two-sample test based on such kernels may be constructed, as detailed in Section \ref{sec:RKHS}, using the same form of U-statistic estimators and bootstrap technique as the finite dimensional scenario. This section will explore properties of the test. Subsection \ref{subsec:discrete_test} will investigate the effect of performing the test on reconstructions of the random function based on observed data. Subsection \ref{subsec:GP_embedding} will provide explicit calculations for MMD when $P,Q$ are Gaussian processes. Subsection \ref{subsec:weak_convergence} discusses the topology on $\calP(\calX)$ induced by MMD and how it relates to the weak topology. 

\subsection{Influence of Function Reconstruction on MMD Estimator}\label{subsec:discrete_test}

In practice, rather than having access to the full realisation of random functions the data available will be some finite-dimensional representation of the functions, for example through discretisation over a mesh, or as a projection onto a finite dimensional basis of $\calX$. Therefore to compute the kernel a user may need to approximate the true underlying functions from this finite dimensional representation. We wish to ensure that the effectiveness of the tests using reconstructed data.

We formalise the notion of discretisation and reconstruction as follows. Assume that we observe $\lbrace \mathcal{I}{x}_i\rbrace_{i=1}^{n}$ where $\lbrace x \rbrace_{i=1}^n$ are the random samples from $P$ and $\mathcal{I}:\mathcal{X}\rightarrow \mathbb{R}^{N}$ is a discretisation map. For example, $\calI$ could be point evaluation at some some fixed $t_1, t_2, \ldots, t_{N} \in \mathcal{D}$ i.e. $\mathcal{I}x_i = (x_i(t_1), \ldots, x_i(t_{N}))^{\top}$. Noisy point evaluation can also be considered in this framework. Then a reconstruction map  $\mathcal{R}:\mathbb{R}^{N} \rightarrow \mathcal{X}$ is employed so that $\calR\calI X_{n} = \lbrace\mathcal{R}\mathcal{I}{x}_i\rbrace_{i=1}^{n}$ is used to perform the test, analogously for $\calR\calI Y_{n}$. For example, $\calR$ could be a kernel smoother or a spline interpolation operator. In practice one might have a different number of observations for each function, the following results can be adapted to this case straightforwardly. 

\begin{proposition}\label{prop:approx_MMD}
	Assume $k$ is a kernel on $\calX$ satisfying $\abs{k(x,y)-k(u,v)}\leq L\norm{x-y-(u-v)}_{\calX}$ for all $u,v,x,y\in\calX$ for some $L > 0$ and let $P,Q\in\calP$ with  $X_{n}=\{x_{i}\}_{i=1}^{n},Y_{n} = \{y_{i}\}_{i=1}^{n}$ i.i.d. samples from $P,Q$ respectively with reconstructed data $\calR\calI X_{n} = \lbrace\mathcal{R}\mathcal{I}{x}_i\rbrace_{i=1}^{n},\calR\calI Y_{n} = \lbrace\mathcal{R}\mathcal{I}{y}_i\rbrace_{i=1}^{n}$ then
	\begin{align*}
		\bigg\lvert \widehat{\eMMD}_{k}(X_{n},Y_{n})^{2} &- \widehat{\eMMD}_{k}(\calR\calI X_{n},\calR\calI Y_{n})^{2} \bigg\rvert\\
		& \leq \frac{4L}{n}\sum_{i=1}^{n}\norm{\calR\calI x_{i}-x_{i}}_{\calX} + \norm{\calR\calI y_{i}-y_{i}}_{\calX}.
	\end{align*}
\end{proposition}

\begin{corollary}\label{cor:approx_T}
    If $k_{T}$ is the SE-$T$ or IMQ-$T$ kernel then the above bound holds with $\norm{T(\calR\calI x_{i})-T(x_{i})}_{\calY},\norm{T(\calR\calI y_{i})-T(y_{i})}_{\calY}$ instead of $\norm{\calR\calI x_{i}-x_{i}}_{\calX},\norm{\calR\calI y_{i}-y_{i}}_{\calX}$ with $L = \frac{1}{\sqrt{e}}$ and $L=\frac{2}{3\sqrt{3}}$ respectively.
\end{corollary}

An analogous result can be derived for the linear time estimator with the same proof technique. While Proposition \ref{prop:approx_MMD} provides a statement on the approximation of $\widehat{\MMD}_{k}(\calR\calI X_{n},\calR\calI Y_{n})^{2}$ we are primarily concerned with its statistical properties. Asymptotically, the test efficiency is characterised via the Gaussian approximation in Theorem \ref{thm:clt_mmd_test}, specifically through the asymptotic variance in \eqref{eq:first_power_surr}.   The following result provides conditions under which a similar central limit theorem holds for the estimator based on reconstructed data, with the same asymptotic variance. It imposes conditions on the number of discretisation points per function sample $N$, the error of the approximations and the number of function samples $n$.

\begin{theorem}\label{thm:recon_normal}
        Let $k$ satisfy the condition in Proposition \ref{prop:approx_MMD} and let $X_n = \lbrace x_i \rbrace_{i=1}^n$ and $Y_n = \lbrace y_i \rbrace_{i=1}^n$ be i.i.d. samples from $P$ and $Q$ respectively with $P\neq Q$, and associated reconstructions $\calR\calI X_{n}$ and $\calR\calI Y_{n}$ based on $N(n)$ dimensional discretisations $\mathcal{I}{X}_n$ and   $\mathcal{I}{Y}_n$ where $N(n)\rightarrow \infty$ as $n\rightarrow \infty$.   If $n^{\frac{1}{2}} \mathbb{E}_{x\sim P}[\lVert x - \mathcal{R}\mathcal{I} x \rVert_{\calX}] \rightarrow 0$ and $n^{\frac{1}{2}} \mathbb{E}_{y\sim P}[\lVert y - \mathcal{R}\mathcal{I} y \rVert_{\calX}] \rightarrow 0$ as $n\rightarrow \infty$, then for $\xi = 4\emph{Var}_{z}\left[\bbE_{z'}[h(z,z')]\right]$
        	\begin{align*}
		n^{\frac{1}{2}}\big(\widehat{\eMMD}_{k_{T}}(\calR\mathcal{I}{X}_{n},\calR\mathcal{I}{Y}_{n})^{2}-\eMMD_{k_{T}}(P,Q)^{2}\big) \xrightarrow{d} \calN(0,\xi).
	\end{align*}
\end{theorem}

A similar result can be derived for the linear time estimator by using the linear time estimator version of Proposition \ref{prop:approx_MMD}.  The discretisation map, number of discretisations per function sample and the reconstruction map need to combine to satisfy the convergence assumption. For example if a weaker reconstruction map is used then more observations per function sample will be needed to compensate for this. Additionally if the discretisation map offers less information about the underlying function, for example it provides observations that are noise corrupted, then more observations per function sample are needed. 

We now discuss three settings in which these assumptions hold, relevant to different applications.
We shall assume that $k$ satisfies the conditions of Proposition \ref{prop:approx_MMD} and that $T = I$.

\subsubsection{Linear interpolation of regularly sampled data}
    Let $\calX = L^{2}([0,1])$ and $\Xi_{N(n)} = \{t_{i}\}_{i=1}^{N(n)}$ be a mesh of evaluation points where $t_{i+1}-t_i = N(n)^{-1}$ for all $i$ and define $\mathcal{I} x = (x(t_1), \ldots, x(t_{N(n)}))^\top \in \mathbb{R}^{N(n)}$.  Let $\mathcal{R}$ be the piecewise linear interpolant defined as
   $$
    (\mathcal{R}\calI x)(t) = (x(t_{i+1}) - x(t_k))\frac{t-t_i}{t_{i+1}-t_i} + x(t_i), \quad \mbox{ for } t \in [t_i, t_{i+1}).
   $$
   Suppose that realisations $x \sim P$ and $y \sim Q$ are almost surely in $C^2([0,1])$ and in particular satisfy $\mathbb{E}_{x\sim P}[\lVert x'' \rVert_{\calX}^2] < \infty$ and $\mathbb{E}_{y\sim Q}[\lVert y'' \rVert_{\calX}^2] < \infty$.  Then 
   $$
    \bbE_{x\sim P}[\lVert x - \mathcal{R}\mathcal{I} x \rVert_{\calX}] \leq \frac{1}{N(n)^2}  \bbE_{x\sim P}[\lVert x'' \rVert_{\calX}],
   $$
   and analogously for $y\sim Q$. Therefore if $N(n) \sim n^{\alpha}$ with $\alpha > 1/4$ then the conditions of Theorem \ref{thm:recon_normal} are satisfied. 

\subsubsection{Kernel interpolant of quasi-uniformly sampled data}
    Let $\calX = L^{2}(\calD)$ with $\mathcal{D} \subset \mathbb{R}^d$ compact.  As in the previous example, $\mathcal{I}$ will be the evaluation operator over a set of points $\Xi_{N(n)} = \{t_i\}_{i=1}^{N(n)}$ but now we assume the points are placed quasi-uniformly in the scattered data approximation sense, for example regularly placed grid points, see \citet{Wynne2020} and \citet[Chapter 4]{Wendland2005} for other method to obtain quasi-uniform points. 
    
    We set $\calR$ as the kernel interpolant using a kernel $k_{0}$ with RKHS norm equivalent to $W^{\nu}_{2}(\calD)$ with $\nu > d/2$, this is achieved by the common Mat\'ern and Wendland kernels \citep{Wendland2005,Kanagawa2018Review}. For this choice of recovery operator, $(\calR\calI x)(t) = k_{0}(t,\Xi)K_{0}^{-1}\calI(x)$ where $k_{0}(t,\Xi) = (k_{0}(t,t_{1}),\ldots,k_{0}(t,t_{N(n)})$ and $K_{0}$ is the kernel matrix of $k_{0}$ over $\Xi_{N(n)}$.
    
    Suppose the realisations of $P$ and $Q$ lie almost surely in $W^{\tau}_{2}(\calD)$ for some $\tau > d/2$, this assumption is discussed when $P,Q$ are Gaussian processes in \citet[Section 4]{Kanagawa2018Review}, then
\begin{align*}
\bbE_{x\sim P}[\norm{x-\calR\calI x}_{\calX}] \leq CN(n)^{-(\tau\wedge\nu)/d},
\end{align*}
for some constant $C > 0$, with an identical result holding for realisations of $Q$ \citep{Wynne2020,Narcowich2006}.  It follows that choosing $N(n) \sim n$ guarantees that the conditions of Theorem \ref{thm:recon_normal}  hold in this setting. Here we see that to maintain a scaling of $N(n)$ independent of dimension $d$ we need the signal smoothness $\nu,\tau$ to increase with $d$. 

Note that the case where $\calI$ is pointwise evaluation corrupted by noise may be treated in a similar way by using results from Bayesian non-parametrics, for example \citet[Theorem 5]{VanderVaart2011}. In this case $\calR$ would be the posterior mean of a Gaussian process that is conditioned on $\calI(x)$. 

\subsubsection{Projection onto an orthonormal basis}
    Let $\calX$ be an arbitrary real, separable Hilbert space and $\{e_n\}_{n=1}^{\infty}$ be an orthonormal basis.  Suppose that $\mathcal{I}$ is a projection operator onto the first $N(n)$ elements of the basis $\mathcal{I}x = (\langle x, e_1\rangle_{\calX}, \ldots, \langle x, e_{N(n)}\rangle_{\calX})^\top$
    and $\calR$ constructs a function from basis coefficients $\mathcal{R}(\beta_1, \ldots, \beta_{N(n)}) = \sum_{i=1}^{N(n)} \beta_i e_i$ meaning $\calR\calI x = \sum_{i=1}^{N(n)}\langle x,e_{i}\rangle_{\calX} e_{i}$. A typical example on $L^2([0,1])$ would be a Fourier series representation of the samples $\lbrace x_i \rbrace_{i=1}^n$ and $\lbrace y_i \rbrace_{i=1}^n$ from which the functions can be recovered efficiently via an inverse Fast Fourier Transform.   By Parseval's theorem  $\mathbb{E}_{x\sim P}[\lVert x - \mathcal{R}\mathcal{I}x \rVert^2_{\calX}] = \sum_{i=N(n)+1}^\infty |\langle x, e_i \rangle_{\calX}|^2 \rightarrow 0$ as $N(n) \rightarrow \infty$.  For the conditions of Theorem \ref{thm:recon_normal} to hold, we require that $n^{1/2}\mathbb{E}_{x\sim P}\big[\sum_{i=N(n)+1}^\infty |\langle x, e_i \rangle_{\calX}|^2\big]\rightarrow 0$ as $n\rightarrow \infty$ which means $N(n)$ will need to grow in a way to compensate for the auto-correlation of realisations of $P$ and $Q$. 
    
    In this setting the use of the integral kernels described in Section \ref{subsec:integral_kernel} are particularly convenient.   Indeed, let $C = \sum_{i=1}^\infty \lambda_i e_i\otimes e_i \in L^1_+(\calX)$ and consider the integral kernel $k_{C, k_0}$ where $k_0(s,t) = \Phi(s) \Phi(t)$ for a feature map $\Phi$ taking values in $\mathbb{R}$. An evaluation of the kernel $k(x,y)$ can then be approximated using a random Fourier feature approach \citep{NIPS2007_3182} by  
    $$
       k(x,y) \approx \frac{1}{n_S} \sum_{l=1}^{n_S} \Phi\left(\sum_{i=1}^{N(n)} \lambda_i^{\frac{1}{2}} x_i \eta_i^l\right) \Phi\left(\sum_{j=1}^{N(n)} \lambda_j^{\frac{1}{2}} y_j \eta_j^l\right),
    $$
    where $x_i = \langle x, e_i\rangle_{\calX}$, $y_i = \langle y, e_i \rangle_{\calX}$ and $\eta_i^l \sim \mathcal{N}(0,1)$ i.i.d. for $i = 1,\ldots N(n)$ and $l=1,\ldots, n_S$ for some $n_S \in \mathbb{N}$.  The permits opportunities to reduce the computational cost of MMD tests as judicious choices of $\Phi$ will permit accurate approximations of $k(x,y)$ using $n_S$ small.  Similarly, the weights, $\lambda_i$ can be chosen to reduce the dimensionality of the functions $x_i$ and $y_i$.

\subsection{Explicit Calculations for Gaussian Processes}\label{subsec:GP_embedding}

A key property of the SE-$T$ kernel is that the mean-embedding $\Phi_{k_{T}}P$ and \sloppy$\MMD_{k_{T}}(P,Q)^{2}$ have closed form solutions when $P,Q$ are Gaussian measures. Using the natural correspondence between Gaussian measures and Gaussian processes from Section \ref{sec:inf_measures} we may get closed form expressions for Gaussian processes. This addresses the open question regarding the link between Bayesian non-parametrics methods and kernel mean-embeddings that was discussed in \citet[Section 6.2]{Muandet2017}. 

Before stating the next result we need to introduce the concept of determinant for an operator, for $S\in L_{1}(\calX)$ define $\det(I+S) = \prod_{n=1}^{\infty}(1+\lambda_{n})$ where $\{\lambda_{n}\}_{n=1}^{\infty}$ are the eigenvalues of $S$. The equality $\det\big((I+S)(I+R)\big) = \det(I+S)\det(I+R)$ holds and is frequently used. 

\begin{theorem}\label{thm:embedding_of_GP}
	Let $k_{T}$ be the SE-$T$ kernel for some $T\in  L^{+}(\calX)$ and $P = N_{a,S}$ be a non-degenerate Gaussian measure on $\calX$ then
	\begin{align*}
		\Phi_{k_{T}}(N_{a,S})(x) = \det(I + TST)^{-\frac{1}{2}}e^{-\frac{1}{2}\langle (I+TST)^{-1}T(x-a),T(x-a)\rangle_{\calX}}.
	\end{align*}
\end{theorem}

\begin{theorem}\label{thm:MMD_two_GPs}
	Let $k_{T}$ be the SE-$T$ kernel for some $T\in L^{+}(\calX)$ and $P = N_{a,S}, Q = N_{b,R}$ be non-degenerate Gaussian measures on $\calX$ then 
	\begin{align*}
	\eMMD_{k_{T}} (P,Q)^{2} &=\det(I+2TST)^{-\frac{1}{2}} + \det(I+2TRT)^{-\frac{1}{2}} \\
	& -2\det\big((I+TST)(I+(TRT)^{\frac{1}{2}}(I+TST)^{-1}(TRT)^{\frac{1}{2}})\big)^{-\frac{1}{2}}\\
	& \hspace{1cm}\times e^{-\frac{1}{2}\langle (I+T(S+R)T)^{-1}T(a-b),T(a-b)\rangle_{\calX}}.
	\end{align*}
\end{theorem}

These results outline the geometry of Gaussian measures with respect to the distance induced by the SE-$T$ kernel. We see that the means only occur in the formula through their difference and if both mean elements are zero then the distance is measured purely in terms of the spectrum of the covariance operators. 

\begin{corollary}\label{cor:MMD_two_GPs_commute}
	Under the Assumptions of Theorem \ref{thm:MMD_two_GPs} and that $T,S,R$ commute then
	\begin{align*}
		\emph{MMD}_{k_{T}}(P,Q)^{2} & = \det(I + 2TST)^{-\frac{1}{2}} + \det(I + 2TRT)^{-\frac{1}{2}} \\
	& -2 \det(I + T(S+R)T)^{-\frac{1}{2}}e^{-\frac{1}{2}\langle (I + T(S+R)T)^{-1}T(a-b),T(a-b)\rangle_{\calX}}.
	\end{align*}
\end{corollary}

Since the variance terms $\xi_{1},\xi_{2}$ from Section \ref{sec:RKHS} are simply multiple integrals of the SE-$T$ kernel against Gaussian measures we may obtain closed forms for them too. Theorem \ref{thm:variance_of_estimator_mean_shift} is a particular instance of the more general Theorem \ref{thm:variance_of_estimator} in the Appendix. 

\begin{theorem}\label{thm:variance_of_estimator_mean_shift}
	Let $P = N_{S}, Q = N_{m,S}$ be non-degenerate Gaussian measures on $\calX$, $T\in L^{+}_{1}(\calX)$ and assume $T$ and $S$ commute then when using the SE-$T$ kernel
	\begin{align*}
	\xi_{1} &= 2\det(\Sigma_{S})^{-\frac{1}{2}}\big(1 + e^{-\langle (I+3TST)^{-1}Tm,Tm\rangle_{\calX}}-2e^{-\frac{1}{2}\langle (I+2TST)\Sigma_{S}^{-1}Tm,Tm\rangle_{\calX}}\big)\\
	&\quad -2\det(I+2TST)^{-1}\big(1+e^{-\langle (I+2TST)^{-1}Tm,Tm\rangle_{\calX}}-2e^{-\frac{1}{2}\langle (I+2TST)^{-1}Tm,Tm\rangle_{\calX}}\big),\\
	\xi_{2} &= 2\det(I+4TST)^{-\frac{1}{2}}\big(1+e^{-\langle (I+4TST)^{-1}Tm,Tm\rangle_{\calX}} \big)\\
		& \quad- 2\det(I+2TST)^{-1}\big(1+e^{-\langle (I+2TST)^{-1}Tm,Tm\rangle_{\calX}}-4e^{-\frac{1}{2}\langle (I+2TST)^{-1}Tm,Tm\rangle_{\calX}}\big)\\
		&\quad -8\det(\Sigma_{S})^{-\frac{1}{2}}e^{-\frac{1}{2}\langle (I+2TST)\Sigma_{S}^{-1}Tm,Tm\rangle_{\calX}},
	\end{align*}
	where $\Sigma_{S} = (I+TST)(I+3TST)$.
\end{theorem}

\subsection{Weak Convergence and MMD}\label{subsec:weak_convergence}
We know that MMD is a metric on $\calP$ when $k$ is characteristic, so it is natural to identify the topology it generates and in particular how it relates to the standard topology for elements of $\calP$, the weak topology. 

\begin{theorem}\label{thm:weak_convergence}
	Let $\calX$ be a Polish space, $k$ a bounded, continuous, characteristic kernel on $\calX\times\calX$ and $P\in\calP$ then $P_{n}\xrightarrow{w} P$ implies $\eMMD_{k}(P_{n},P)\rightarrow 0$ and if $\{P_{n}\}_{n=1}^{\infty}\subset\calP$ is tight then $\eMMD_{k}(P_{n},P)\rightarrow 0$ implies $P_{n}\xrightarrow{w} P$ where $\xrightarrow{w}$ denotes weak convergence.    
\end{theorem}

For a discussion on weak convergence and tightness see \citet{Billingsley1971}. The tightness is used to compensate for the lack of compactness of $\calX$ which is often required in analogous finite dimensional results. In particular, in \citet{Chevyrev2018} an example where $\MMD_{k}(P_{n},P)\rightarrow 0$ but $P_{n}$ but does converge to $P$ was given without the assumption of tightness. A precise characterisation of the relationship between MMD and weak convergence over a Polish space is an open problem.

\section{Practical Considerations for Kernel Selection} \label{sec:implementation}

We now present examples and techniques to choose kernels and construct maps $T$ that are admissible. Two main categories will be discussed, integral operators induced by kernels and situation specific kernels. 

For the first catergory assume $\calX = \calY = L^{2}(\calD)$ for some compact $\calD\subset\bbR^{d}$ and let $k_{0}$ be a measurable kernel over $\calD\times\calD$ and set $T = C_{k_{0}}$ where $C_{k_{0}}$ is the covariance operator associated with $k_{0}$, see Section \ref{sec:inf_measures}. We call $k_{0}$ an \emph{admissible kernel}  if $C_{k_{0}}$ is admissible. If $k_{0}$ is continuous then by Mercer's theorem $C_{k_{0}}x = \sum_{n=1}^{\infty}\lambda_{n}\langle x,e_{n}\rangle e_{n}$ for some positive sequence $\{\lambda_{n}\}_{n=1}^{\infty}$ and orthonormal set $\{e_{n}\}_{n=1}^{\infty}$ \citep[Chapter 4.5]{Steinwart2008}. To be admissible $C_{k_{0}}$ needs to be injective which is equivalent to $\{e_{n}\}_{n=1}^{\infty}$ forming a basis \citep[Proof of Theorem 3.1]{Steinwart2012}. Call $k_{0}$ \emph{integrally strictly positive definite} (ISPD) if $\sloppy{\int_{\calD}\int_{\calD} x(s)k_{0}(s,t)x(t)dsdt > 0}$ for all non-zero $x\in\calX$. Recall that if $k_{0}$ is translation invariant then by Theorem \ref{thm:bochner} there exists a measure $\mu_{k_{0}}$ such that $\hat{\mu}_{k_{0}}(s-t)= k(s,t)$.

\begin{proposition}\label{prop:admissible_kernel}
	Let $\calD\subset\bbR^{d}$ be compact and $k_{0}$ a continuous kernel on $\calD$, if $k_{0}$ is ISPD then $k_{0}$ is admissible. In particular, if $k_{0}$ is continuous and translation invariant and $\mu_{k_{0}}$ has full support on $\calD$ then $k_{0}$ is admissible. 
\end{proposition}

For multiple examples of ISPD kernels see \citet{Sriperumbudur2011} and of $\mu_{k_{0}}$ see \citet{Sriperumbudur2010}. Using the product to convolution property of the Fourier transform one can construct $k_{0}$ such that $\mu_{k_{0}}$ has full support relatively easily or  modify standard integral operators which aren't admissible. For example, for some $F\in\bbN$ consider the kernel $k_{\text{cos}(F)}(s,t) =\nobreak \sum_{n=0}^{F-1}\cos(2\pi n(s-t))$ on $[0,1]^{2}$ whose spectral measure if a sum of Dirac measures so does not have full support. If the Dirac measures are convolved with a Gaussian then they would be smoothed out and would result in a measure with full support. Since convolution in the frequency domain corresponds to a product in space domain the new kernel $k_{\text{c-exp}(F,l)}(s,t) = e^{-\frac{1}{2 l^{2}}(s-t)^{2}}k_{\text{cos}(F)}(s,t)$ satisfies the conditions of Proposition \ref{prop:admissible_kernel}. This technique of frequency modification has found success in modelling audio signals \citep[Section 3.4]{Wilkinson2019}. In general, any operator of the form $Tx = \sum_{n=1}^{\infty}\lambda_{n}\langle x,e_{n}\rangle_{\calX} e_{n}$ for positive, bounded $\{\lambda_{n}\}_{n=1}^{\infty}$ and an orthonormal basis $\{e_{n}\}_{n=1}^{\infty}$ is admissible even if it is not induced by a kernel, for example the functional Mahalanobis distance \citep{Berrendero2020}.

The second category is scenario specific choices. By this we mean kernels whose structure is specified to the testing problem at hand. For example, while the kernel two-sample test may be applied for distributions with arbitrary difference one may tailor it for a specific testing problem, such a difference of covariance operator. If one does only wish to test for difference in covariance operator of the two probability measures then an appropriate kernel would be $k_{\text{cov}}(x,y) = \langle x,y\rangle_{\calX}^{2}$ which is not characteristic but $\MMD_{k_{\text{cov}}}(P,Q) = 0$ if and only if $P,Q$ have the same covariance operator. However, a practitioner may want a kernel which emphasises difference in covariance operator, due to prior knowledge regarding the data, while still being able to detect arbitrary difference, in case the difference is more complicated than initially thought. We now present examples of $T$ which do this. To emphasise higher order moments, let $\calX\subset L^{4}(\calD)$ and $\calY$ the direct sum of $L^{2}(\calD)$ with itself equipped with the norm $\norm{(x,x')}_{\calY}^{2} = \norm{x}_{L^{2}(\calD)}^{2} + \norm{x'}_{L^{2}(\calD)}^{2}$ and $T(x) = (x,x^{2})$. This map captures second order differences and first order differences individually, as opposed to the polynomial map which combines them. Alternatively, one might be in a scenario where the difference in the distributions is presumed to be in the lower frequencies. In this case a map of the form $T(x) = \sum_{n=1}^{F}\lambda_{n}\langle x,e_{n}\rangle_{\calX} e_{n}$ could be used for decreasing, positive $\lambda_{n}$ and some orthogonal $e_{n}$. This will not be characteristic, since $T$ only acts on $F$ frequencies, however if $F$ is picked large enough then good performance could still be obtained in practice. For example, $\lambda_{n},e_{n}$ could be calculated empirically from the function samples using functional principal component analysis and $F$ could be picked so that the components explain a certain percentage of total variance. See \citet{horvath2012inference} for a deeper discussion on functional principal components and its central role in functional data analysis. 

All of the choices outlined above have associated hyperparameters, for example if $T = C_{k_{0}}$ then hyperparameters of $k_{0}$ are hyperparameters of $T$ such as the bandwidth. It is outside the scope of this paper to investigate new methods to choose these parameters but we do believe it is important future work. Multiple methods for finite dimensional data have been proposed using the surrogrates for test power outlined in Section \ref{sec:scaling} \citep{Sutherland2016,Gretton2012optimal,Liu2020} which could have potential for use in the infinite dimensional scenario.

\section{Numerical Simulations} \label{sec:numerics}

In this section we perform numerical simulations on real and synthetic data to reinforce the theoretical results. Code is available at \url{https://github.com/georgewynne/Kernel-Functional-Data}.

\subsection{Power Scaling of Functional Data}\label{subsec:num_power_scaling}
Verification of the power scaling when performing the mean shift two-sample test using functional data, discussed in Section \ref{sec:scaling}, is performed. Specifically we perform the two-sample test using the SE-$I$ kernel with $x\sim\calGP(0,k_{l})$ and $y\sim\calGP(m,k_{l})$ where $m(t) = 0.05$ for $t\in[0,1]$ and $k_{l}(s,t) = e^{-\frac{1}{2l^{2}}(s-t)^{2}}$ with $50$ samples from each distribution. This is repeated $500$ times to calculate power with $1000$ permutations used in the bootstrap to simulate the null. The observation points are a uniform grid on $[0,1]$ with $N$ points, meaning $N$ will be the dimension of the observed discretised function vectors. The parameter $l$ dictates the dependency of the fluctations. Small $l$ means less dependency between the random function values so the covariance matrix is closer to the identity. When the random functions are $m$ with $\calN(0,1)$ i.i.d. corruption the corresponding value of $l$ is zero which essentially means $k_{l}(x,y) = \delta_{xy}$. In this case the scaling of power is expected to follow \eqref{eq:scaling_id} and grow asymptotically as $\sqrt{N}$. On the other hand if $l > 0$ the fluctuations within each random function are dependent and we expect scaling as \eqref{eq:scaling_ratio} which does not grow asymptotically with $N$. 

Figure \ref{fig:gp_scaling} confirms this theory showing that power increases with a finer observation mesh only when there is no dependence in the random functions values. We see some increase of power as the mesh gets finer for the case of small dependency however the rate of increase is much smaller than the i.i.d. setting. 

\begin{figure}[ht]
\centering
\includegraphics[width=10cm]{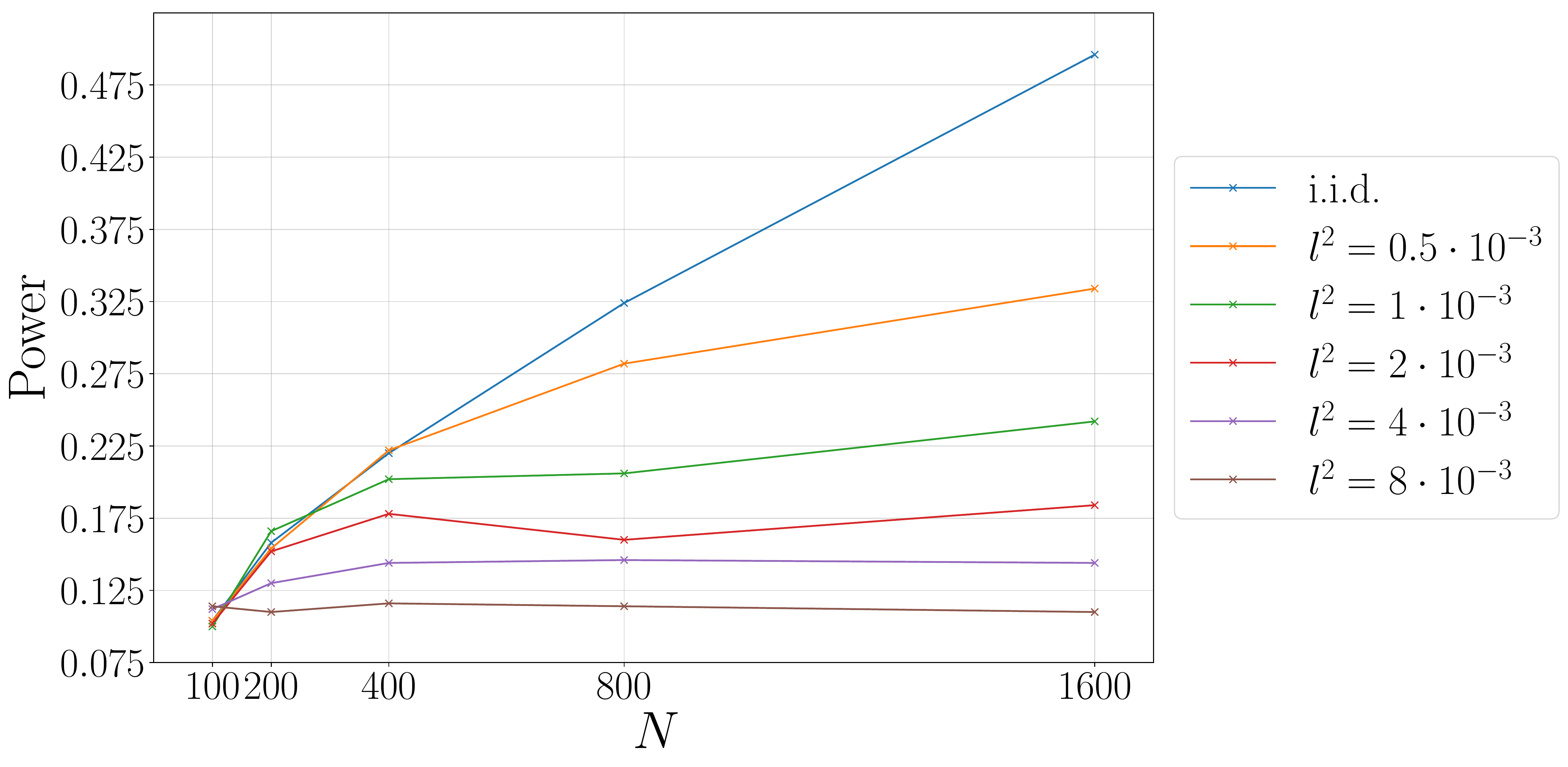}
\caption{Test power as mesh size decreases given different point dependency strengths}
\label{fig:gp_scaling}
\end{figure}

\subsection{Synthetic Data}

The tests are all performed using the $k_{\text{cov}}$ kernel from Section \ref{sec:implementation} and the SE-$T$ kernel for four different choices of $T$ and unless stated otherwise $\calY = L^{2}([0,1])$ and we use the short hand $L^{2}$ for $L^{2}([0,1])$ and $n_{x},n_{y}$ will denote the sample sizes of the two samples. To calculate power each test is repeated $500$ times and $1000$ permutations are used in the bootstrap to simulate the null distribution. 

COV will denote the $k_{\text{cov}}$ kernel, which can only detect difference of covariance operator. ID will denote $T = I$. CEXP will denote $T = C_{k_{0}}$ with $k_{0} = k_{\text{c-exp}(20,\sqrt{10})}$ the cosine exponential kernel. SQR will denote $T(x) = (x,x^{2})$ with $\calY$ the direct sum of $L^{2}([0,1])$ with itself as detailed in Section \ref{sec:implementation}. FPCA will denote $Tx = \sum_{n=1}^{F}\lambda_{n}^{1/2}\langle x,e_{n}\rangle e_{n}$ where $\lambda_{n},e_{n}$ are empirical functional principal components and principal values computed from the union of the two collections of samples with $F$ chosen such that $95\%$ of variance is explained. The abbreviations are summarised in Table \ref{fig:abbreviations} along with references to the other tests being compared against. 

For the four uses of the SE-$T$ kernel $\exp{(-\frac{1}{2\gamma^{2}}\norm{T(x)-T(y)}_{\calY}^{2})}$ we use, for all but SQR scenario, the median heuristic $\gamma^{2} = \text{Median}\big\{\norm{T(a)-T(b)}_{\calY}^{2}\colon a,b\in \{x_{i}\}_{i=1}^{n_{X}}\cup\{y_{i}\}_{i=1}^{n_{Y}}, a\neq b\big\}$. As the SQR scenario involves two norms in the exponent two calculations of median heuristic are needed so that the kernel used is $\exp(-\frac{1}{2\gamma_{1}^{2}}\norm{x-y}_{L^{2}}^{2} - \frac{1}{2\gamma_{2}^{2}}\norm{x^{2}-y^{2}}_{L^{2}}^{2})$ with $\gamma_{j}^{2} = \text{Median}\big\{\norm{a^{j}-b^{j}}_{L^{2}}^{2}\colon a,b\in \{x_{i}\}_{i=1}^{n_{X}}\cup\{y_{i}\}_{i=1}^{n_{Y}}, a\neq b\big\}$ for $j=1,2$.

\begin{table}
\begin{adjustwidth}{-1in}{-1in}
 \centering
  \begin{tabular}{ccc}
    \toprule
    Abbreviation & Description & Reference\\
    \midrule
    ID  & SE-$T$ kernel, $T = I$  & Section \ref{sec:kernel} \\
    FPCA  & SE-$T$ kernel, $T$ based on functional principle components & Section \ref{sec:implementation}  \\
    SQR  & SE-$T$ kernel, $T$ squaring feature expansion & Section \ref{sec:implementation}  \\
    CEXP  & SE-$T$ kernel, $T$ based on the cosine-exponential kernel  & Section \ref{sec:implementation}  \\
    COV  & Covariance kernel $k(x,y) = <x,y>_{\calX}^{2}$  & Section \ref{sec:implementation}  \\
    FAD  & Functional Anderson-Darling  & \citep{Pomann2016}  \\
    CVM  & Functional Cramer-von Mises  & \citep{Hall2007}  \\
    BOOT-HS  & Bootstrap Hilbert-Schmidt  & \citep{Paparoditis2016}  \\
    FPCA-$\chi$  & Functional Principal Component $\chi^{2}$  & \citep{FREMDT2012} \\ 
    \bottomrule
  \end{tabular}
  \end{adjustwidth}
  \caption{Summary of two-sample tests and kernels used in numerical experiments}
  \label{fig:abbreviations}
  
\end{table}

\subsubsection*{Difference of Mean}
We compare to the Functional Anderson-Darling (FAD) test in \citet{Pomann2016} which involves computing functional principal components and then doing multiple Anderson-Darling tests. Independent realisations $\{x_{i}\}_{i=1}^{n_{x}}$ and $\{y_{j}\}_{j=1}^{n_{y}}$ of the random functions $x,y$ over $[0,1]$ are observed on a grid of $100$ uniform points with $n_{x} = n_{y} = 100$ and observation noise $\calN(0,0.25)$. The two distributions are 
\begin{align*}
	x(t)& \sim t + \xi_{10}\sqrt{2}\sin(2\pi t) + \xi_{5}\sqrt{2}\cos(2\pi t), \\
	y(t)&\sim t+\delta t^{3} + \eta_{10}\sqrt{2}\sin(2\pi t) + \eta_{5}\sqrt{2}\cos(2\pi t),
\end{align*}
with $\xi_{5},\eta_{5}\stackrel{i.i.d}{\sim}\calN(0,5)$ and $\xi_{10},\eta_{10}\stackrel{i.i.d}{\sim}\calN(0,10)$. The $\delta$ parameter measures the deviation from the null hypothesis that $x,y$ have the same distribution. The range of the parameter is $\delta\in\{0,0.5,1,1.5,2\}$. 

Figure \ref{fig:mean_shift} shows CEXP performing best among all the choices which makes sense since this choice explicitly smooths the signal to make the mean more identifiable compared to the noise. We see that FPCA performs poorly because the principal components are deduced entirely from the covariance structure and do not represent the mean difference well. Likewise COV performs poorly since it can only detect difference in covariance, not mean. Except from FPCA and COV all choices of $T$ out perform the FAD method. This is most likely because the FAD method involves computing multiple principle components, an estimation which is inherently random, and computes multiple FAD tests with a Bonferroni correction which can cause too harsh a requirement for significance. There is a slight inflation of test size, meaning rejection is mildly larger than $5\%$ when the null hypothesis is true. 

Figure \ref{fig:ROC} shows an ROC curve. On the $x$-axis is $\alpha$ the false positive rate parameter in the test, see Section \ref{sec:RKHS}, and on the $y$-axis is the power of the test, meaning the true positive rate. The plot was obtained using $\delta = 1.25, n_{x} = 50, n_{y} = 50$ with the same observation locations and noise as described above. The dashed line is $y=x$ which corresponds to a test with trivial performance. We see that COV and FPCA performs trivially weakly implying the calculated principal components are uninformative for identifying the difference in mean. CEXP has the best curve and the other three choices of $T$ perform equally well.

\begin{figure}[ht]
\centering
\includegraphics[width=10cm]{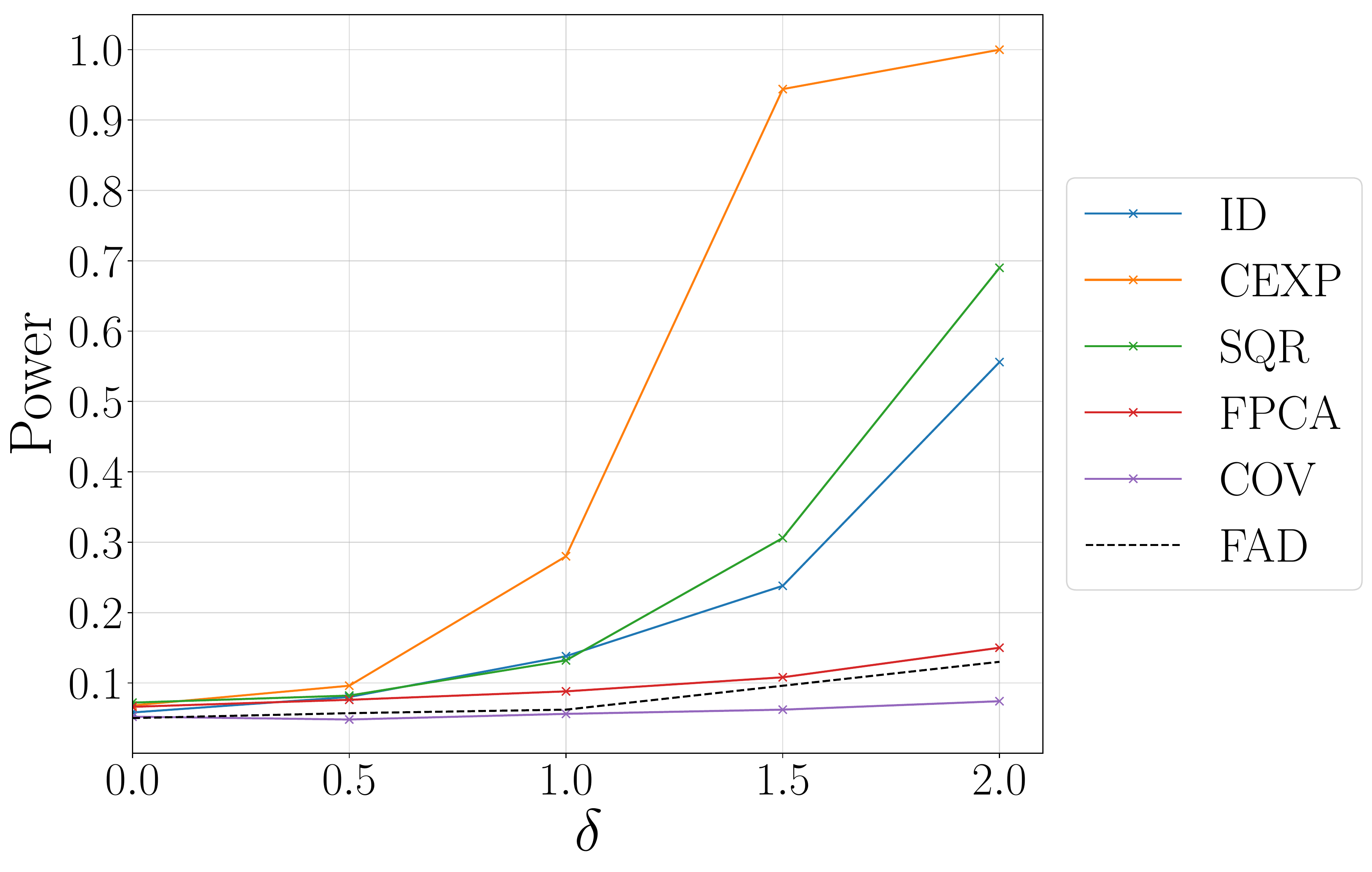}
\caption{Test power under mean difference for different kernels.}
\label{fig:mean_shift}
\end{figure}

\begin{figure}[ht]
\centering
\includegraphics[width=10cm]{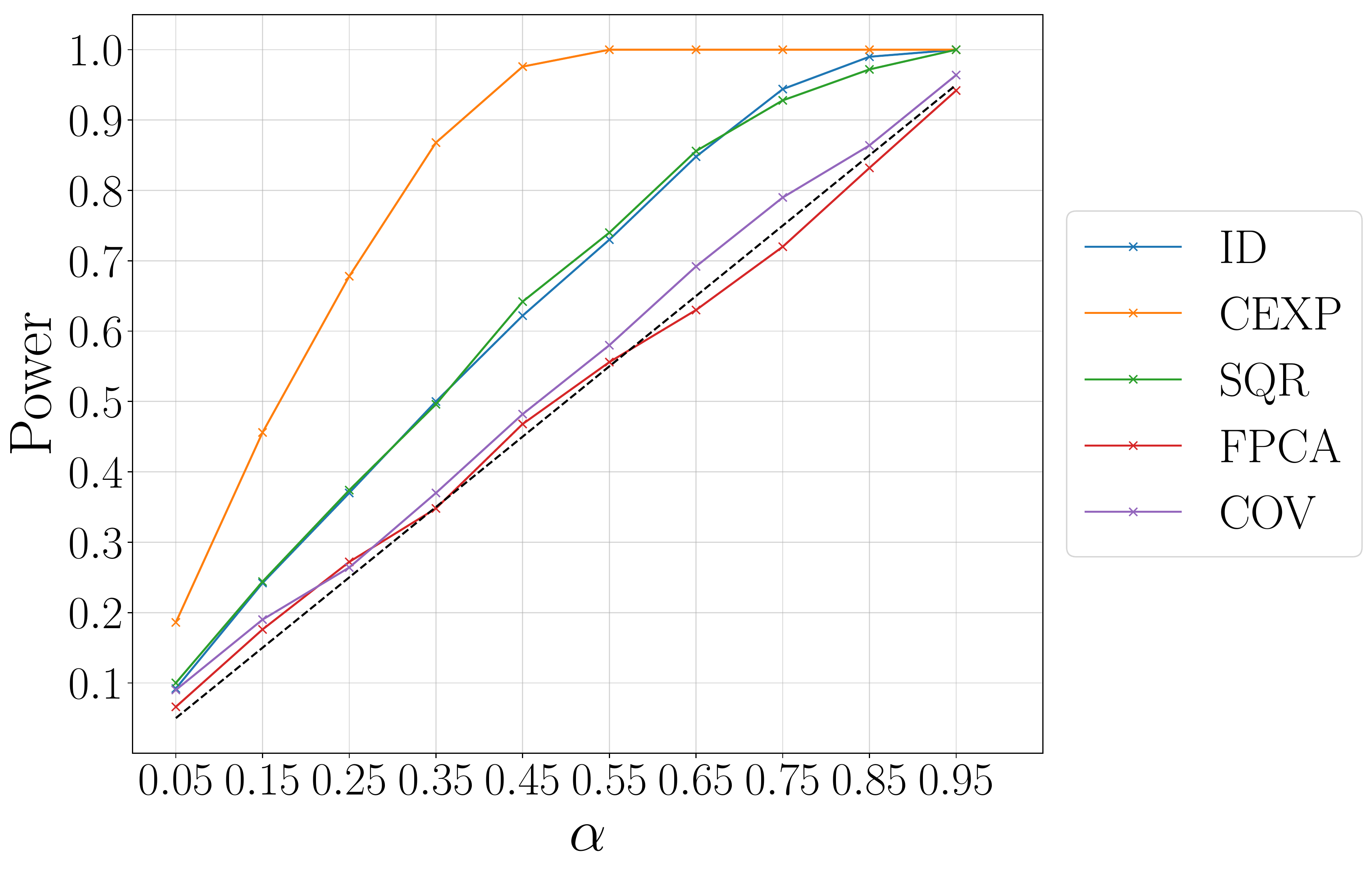}
\caption{ROC curve for different kernels.}
\label{fig:ROC}
\end{figure}

\subsubsection*{Difference of Variance}

We investigate two synthetic data sets, the first from \citet{Pomann2016} and the second from \citet{Paparoditis2016}. The first represents a difference in covariance in a specific frequency and the second a difference across all frequencies. 

In the first data set $n_{x} = n_{y} = 100$, observations are made on a uniform grid of $100$ points and the observation noise is $\calN(0,0.25)$. The two distributions are
\begin{align*}
	x(t)& \sim \xi_{10}\sqrt{2}\sin(2\pi t) + \xi_{5}\sqrt{2}\cos(2\pi t), \\
	y(t)&\sim \eta_{10+\delta}\sqrt{2}\sin(2\pi t) + \eta_{5}\sqrt{2}\cos(2\pi t),
\end{align*}
with $\xi_{5},\eta_{5}\stackrel{i.i.d}{\sim}\calN(0,5)$ and $\xi_{10}\sim\calN(0,10)$ and $\eta_{10+\delta}\sim\calN(0,10+\delta)$. Therefore the difference in covariance structure is manifested in the first frequency. The range of the parameter is $\delta\in\{0,5,10,15,20\}$ and we again compare against the FAD test.  

Figure \ref{fig:var_shift_1} shows that COV performs the best which is to be expected since it is specifically designed to only detect change in covariance. SQR and FPCA perform well since they are designed to capture covariance information too. CEXP performs almost identically to ID since it designed to improve performance on mean shift tests, not covariance shift. 

\begin{figure}[ht]
\centering
\includegraphics[width=10cm]{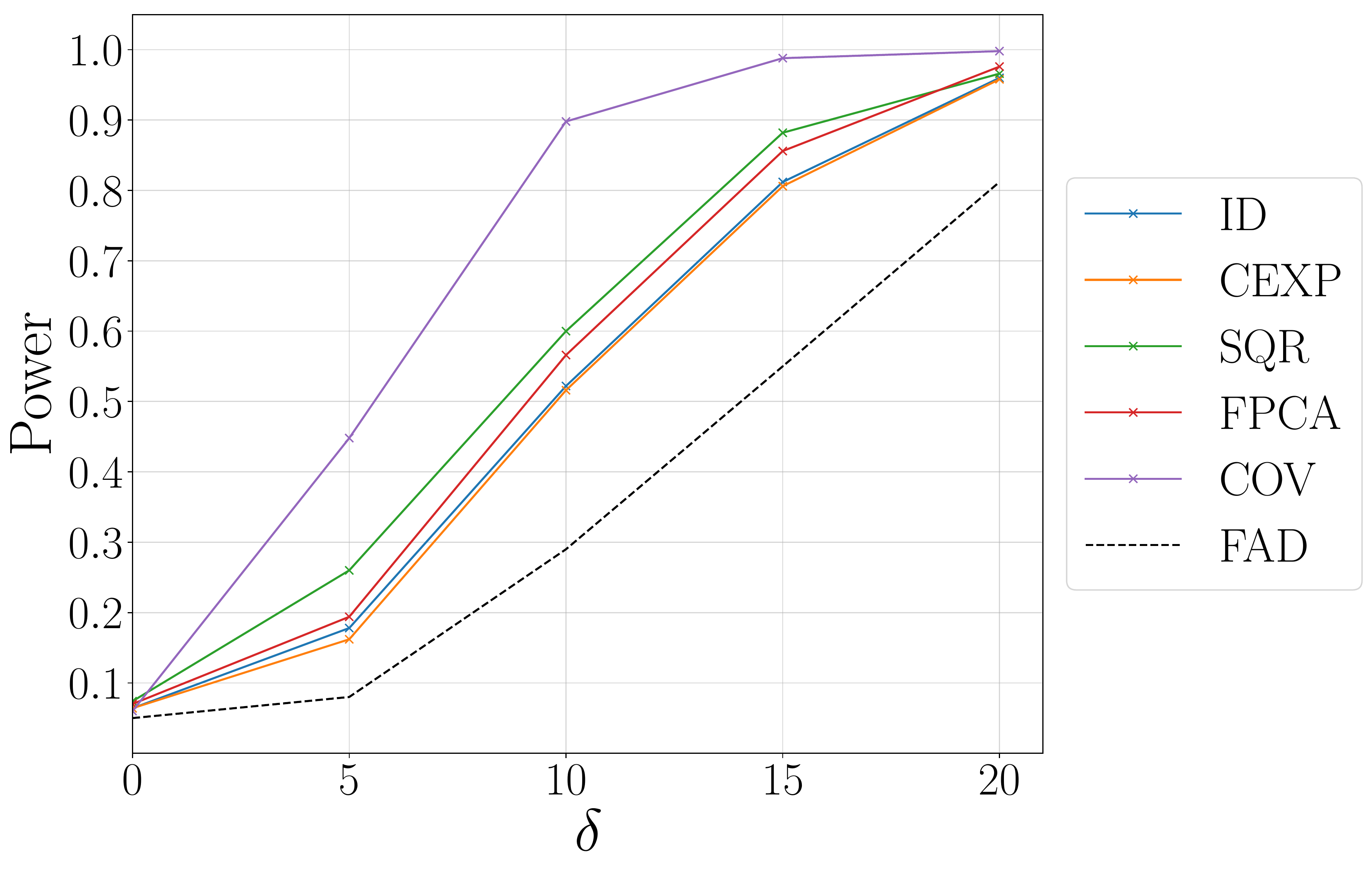}
\caption{Test power under variance difference in one frequency for different kernels.}
\label{fig:var_shift_1}
\end{figure}

The second dataset is from \citet{Paparoditis2016} and we compare against the data reported there of a bootstrap Hilbert-Schmidt norm (BOOT-HS) test \citep[Section 2.2]{Paparoditis2016} and a functional principal component chi-squared (FPCA-$\chi$) test \citep{FREMDT2012}, which is similar to the test in \citet{Panaretos2010}. The number of function samples is $n_{x} = n_{y} = 25$ and each sample is observed on a uniform grid over $[0,1]$ consisting of $500$ points. The first distribution is defined as
\begin{align*}
	x(t)\sim \sum_{n=1}^{10}\xi_{n}n^{-\frac{1}{2}}\sqrt{2}\sin(\pi n t) + \eta_{n}n^{-\frac{1}{2}}\sqrt{2}\cos(\pi n t),
\end{align*}
where $\xi_{n},\eta_{n}$ are i.i.d. Student's $t$-distribution random variables with $5$ degrees of freedom. For $\delta\in\bbR$ the other function distribution is $y\sim\delta x'$ where $x'$ is an i.i.d. copy of $x$. When $\delta = 1$ the two distributions are the same. The entire covariance structure of $Y$ is different from that of $X$ when $\delta\neq 1$ which is in contrast the previous numerical example where the covariance structure differed at only one frequency. The range of the deviation parameter is $\delta\in\{1,1.2,1.4,1.6,1.8,2\}$.

Figure \ref{fig:var_shift_2} shows again that COV and SQR performs the best. The BOOT-HS and FPCA-$\chi$ tests are both conservative, providing rejection rates below $5\%$ when the null is true as opposed to the kernel based tests which all lie at or very close to the $5\%$ level. 

\begin{figure}[ht]
\centering
\includegraphics[width=10cm]{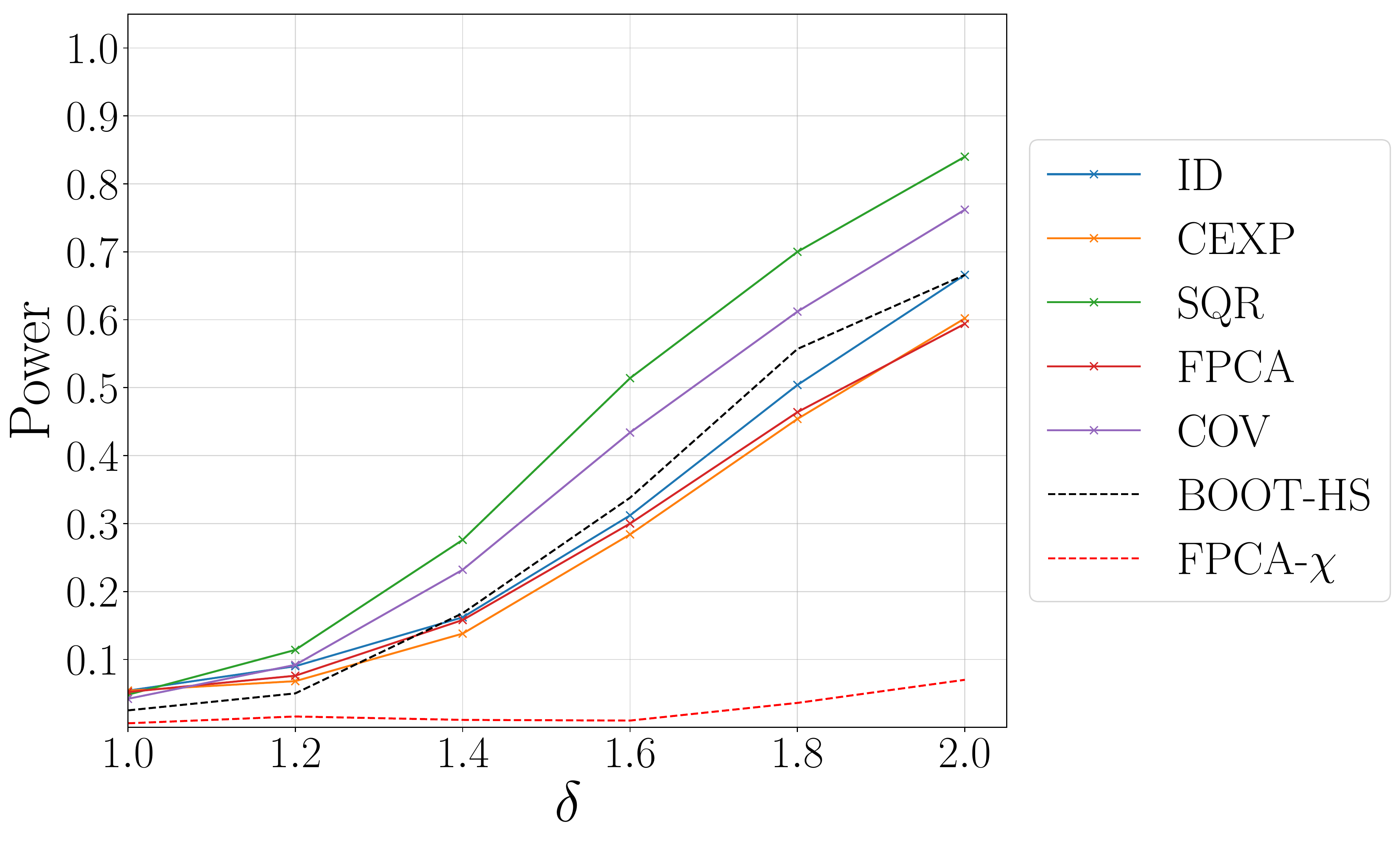}
\caption{Test power under variance difference across all frequencies for different kernels.}
\label{fig:var_shift_2}
\end{figure}

\subsubsection*{Difference of Higher Orders}

Data from \citet{Hall2007} is used when performing the test. The random functions $x,y$ are distributed as
\begin{align*}
	x(t) &\sim \sum_{n=1}^{15}e^{-\frac{n}{2}}\xi^{x}_{n}\psi_{n}(t),\\
	y(t) &\sim \sum_{n=1}^{15}e^{-\frac{n}{2}}\xi^{y}_{n,1}\psi_{n}(t) + \delta\sum_{n=1}^{15}n^{-2}\xi_{n,2}^{y}\psi_{n}^{*}(t),
\end{align*}
with $\xi^{x}_{n},\xi^{y}_{n,1},\xi^{y}_{n,2}\stackrel{i.i.d}{\sim}\calN(0,1)$, $\psi_{1}(t) = 1$, $\psi_{n}(t) = \sqrt{2}\sin((k-1)\pi t)$ for $n > 1$ and $\psi_{1}^{*}(t) = 1$, $\psi_{n}^{*}(t) = \sqrt{2}\cos((k-1)\pi(2t-1))$ if $n > 1$ is even, $\psi_{n}^{*}(t) = \sqrt{2}\sin((k-1)\pi(2t-1))$ if $n > 1$ is odd. The observation noise for $x$ is $\calN(0,0.01)$ and for y is $\calN(0,0.09)$. The range of the parameter is $\delta\in\{0,1,2,3,4\}$ and we compare against the FAD test and the Cramer-von Mises test in \citet{Hall2007}. The number of samples is $n_{x} = n_{y} = 15$ and for each random function $20$ observation locations are sampled randomly according to $p_{x}$ or $p_{y}$ with $p_{x}$ being the uniform distribution on $[0,1]$ and $p_{y}$ the distribution with density function $0.8 + 0.4t$ on $[0,1]$. 

Since the data is noisy and irregularly sampled, curves were fit to the data before the test was performed. The posterior mean of a Gaussian process with noise parameter $\sigma^{2} = 0.01$ was fit to each data sample using a Mat\'ern-$1.5$ kernel $k_{\text{Mat}}(s,t) = (1+\sqrt{3}(s-t))e^{-\sqrt{3}(s-t)}$.

Figure \ref{fig:hall} shows that the COV, SQR perform the best with other choices of $T$ performing equally. Good power is still obtained against the existing methods despite the function reconstructions, validating the theoretical results of Section \ref{sec:inf_dim_MMD}.

\begin{figure}[ht]
\centering
\includegraphics[width=10cm]{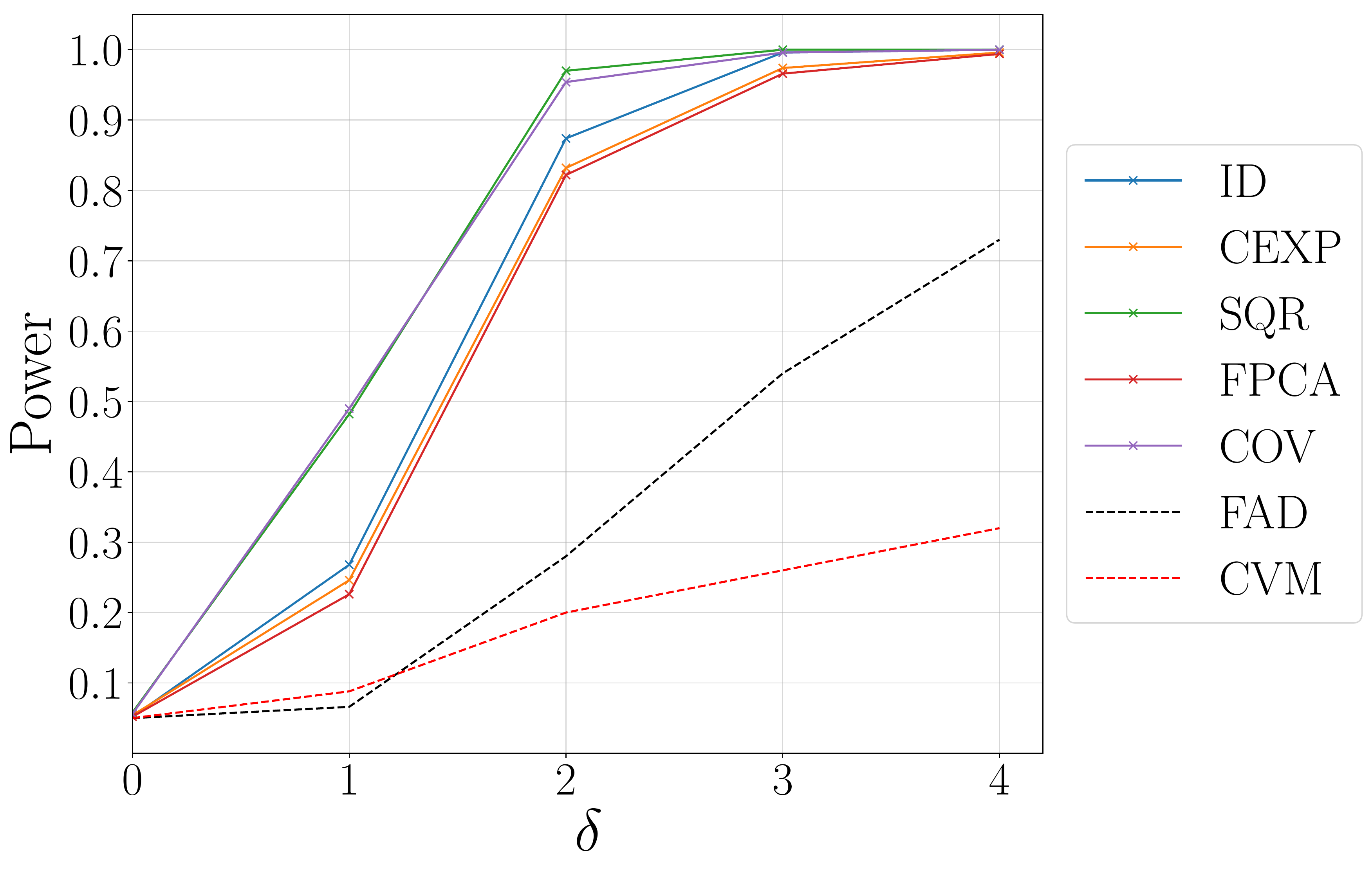}
\caption{Test power under difference of higher orders for different kernels.}
\label{fig:hall}
\end{figure}

\subsection{Real Data}

\subsubsection*{Berkeley Growth Data}
We now perform tests on the Berkeley growth dataset which contains the height of $39$ male and $54$ female children from age $1$ to $18$ and $31$ locations. The data can be found in the R package fda. We perform the two sample test on this data for the five different choices of $T$ with $\gamma$ chosen via the median heuristic outlined in the previous subsection. To identify the effect on test performance of sample size we perform random subsampling of the datasets and repeat the test to calculate test power. For each sample size $M\in\{5,15,25,35\}$ we sample $M$ functions from each data set and perform the test, this is repeated $500$ times to calculate test power. The results are plotted in Figure \ref{fig:berkeley}. Similarly, to investigate the size of the test we sample two disjoint subsets of size $M\in\{5,15,25\}$ from the female data set and perform the test and record whether the null was incorrectly rejected, this is repeated $500$ times to obtain a rate of incorrect rejection of the null, the results are reported in Table \ref{fig:berkeley_null}. 

Figure \ref{fig:berkeley} shows ID, SQR performing the best, COV performs weaker than CEXP suggesting that it is not just a difference of covariance operator that distinguishes the two samples. Table \ref{fig:berkeley_null} shows nearly all the tests have the correct empirical size.

\begin{figure}[ht]
\centering
\includegraphics[width=10cm]{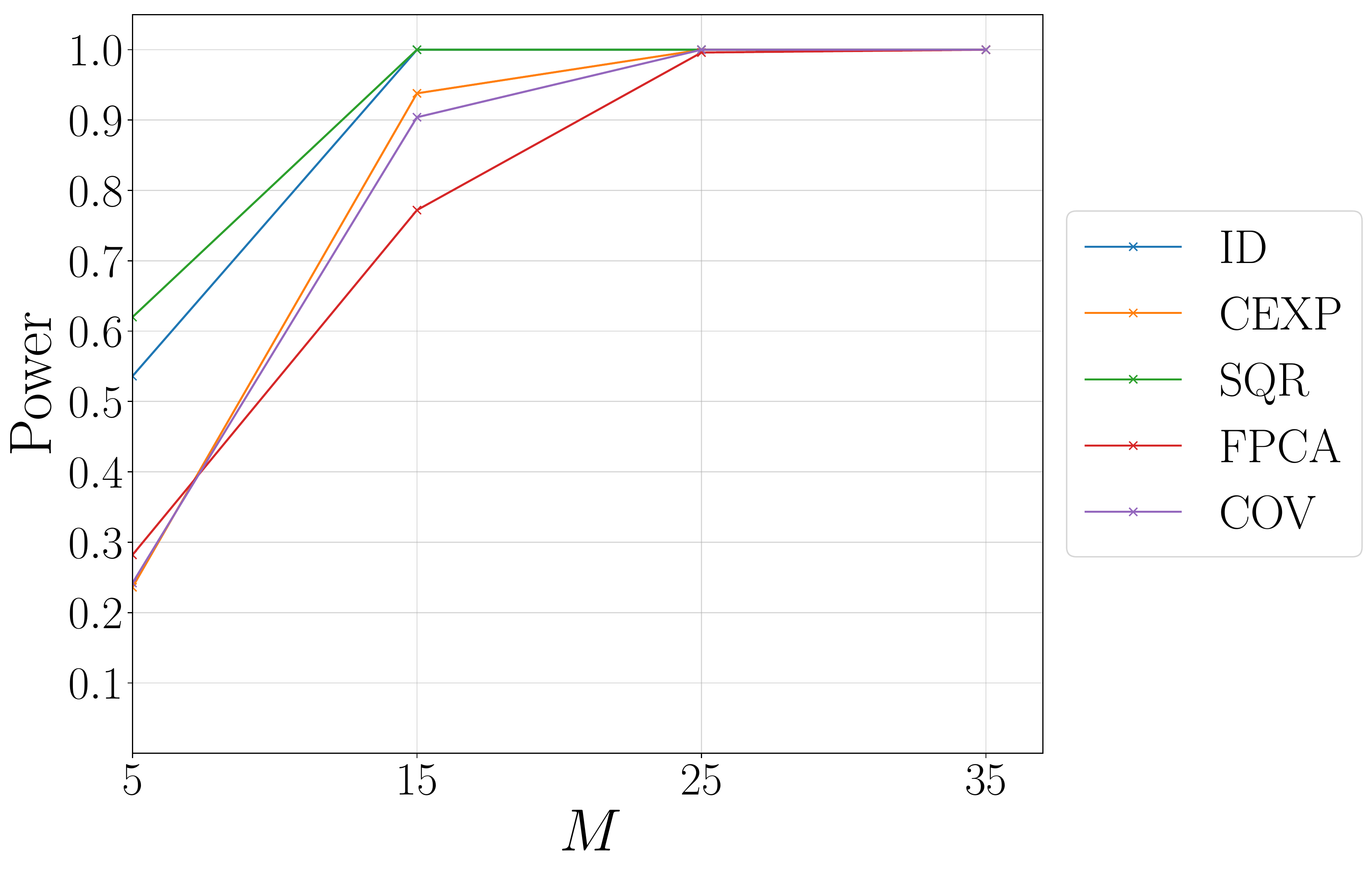}
\caption{Test power under subsamples of size $M$ using Berkeley growth data.}
\label{fig:berkeley}
\end{figure}

\begin{table}
\centering
  \begin{tabular}{l|ccccc}
    \toprule
    M & ID & CEXP & COV & SQR & FPCA \\
    \midrule
    5  & 5.0  & 4.8 &  4.2 &  5.6 & 5.0 \\
    15  & 4.4 & 4.6 &  4.4 &  5.2 &  4.6 \\
    25  & 5.0  & 4.6 &  5.4 &  5.8 &  5.6 \\
    \bottomrule
  \end{tabular}
  \caption{Empirical size, meaning the rate of rejection of the null when the null is true, of the two-sample test performed on the female Berkeley growth data for different sample sizes across different choices of $T$. The values are written as percentages, a value above $5$ shows too frequent rejection and below shows too conservative a test.}
  \label{fig:berkeley_null}
\end{table}

\subsubsection*{NEU Steel Data}
We perform the two-sample test on two classes from the North Eastern University steel defect dataset \citep{Song2013,He2020,Dong2019}. The dataset consists of $200\times 200$ pixel grey scale images of defects of steel surfaces with $6$ different classes of defects and $300$ images in each class. We perform the test on the two classes which are most visually similar, called rolled-in scale and crazing. See the URL \citep{SongURL} for further description of the dataset. For each sample size $M\in\{10,20,30,40\}$ we sample $M$ images from each class and perform the test, this is repeated $500$ times to calculate test power. Again we assess the empirical size by sampling two distinct subsets from one class, the rolled-in class, for sample sizes $M\in\{10,20,30,40\}$ and repeat this $500$ times and report the rate of incorrect null rejection. For CEXP we use the two dimensional tensor product kernel induced by the CEXP kernel with $20$ frequencies and $l = \sqrt{10}/200$ to normalise the size of the images.

Figure \ref{fig:steel} shows SQR having the best performance, CEXP performs well and so does ID. Table \ref{fig:steel_null} shows that the empirical size is inflated under some choices of $T$ especially CEXP. Once the test is performed with $40$ samples the sizes return to an acceptable level for SQR and FPCA. This inflation of empirical size should be taken into account when viewing the powers of the tests.   

\begin{figure}[ht]
\centering
\includegraphics[width=10cm]{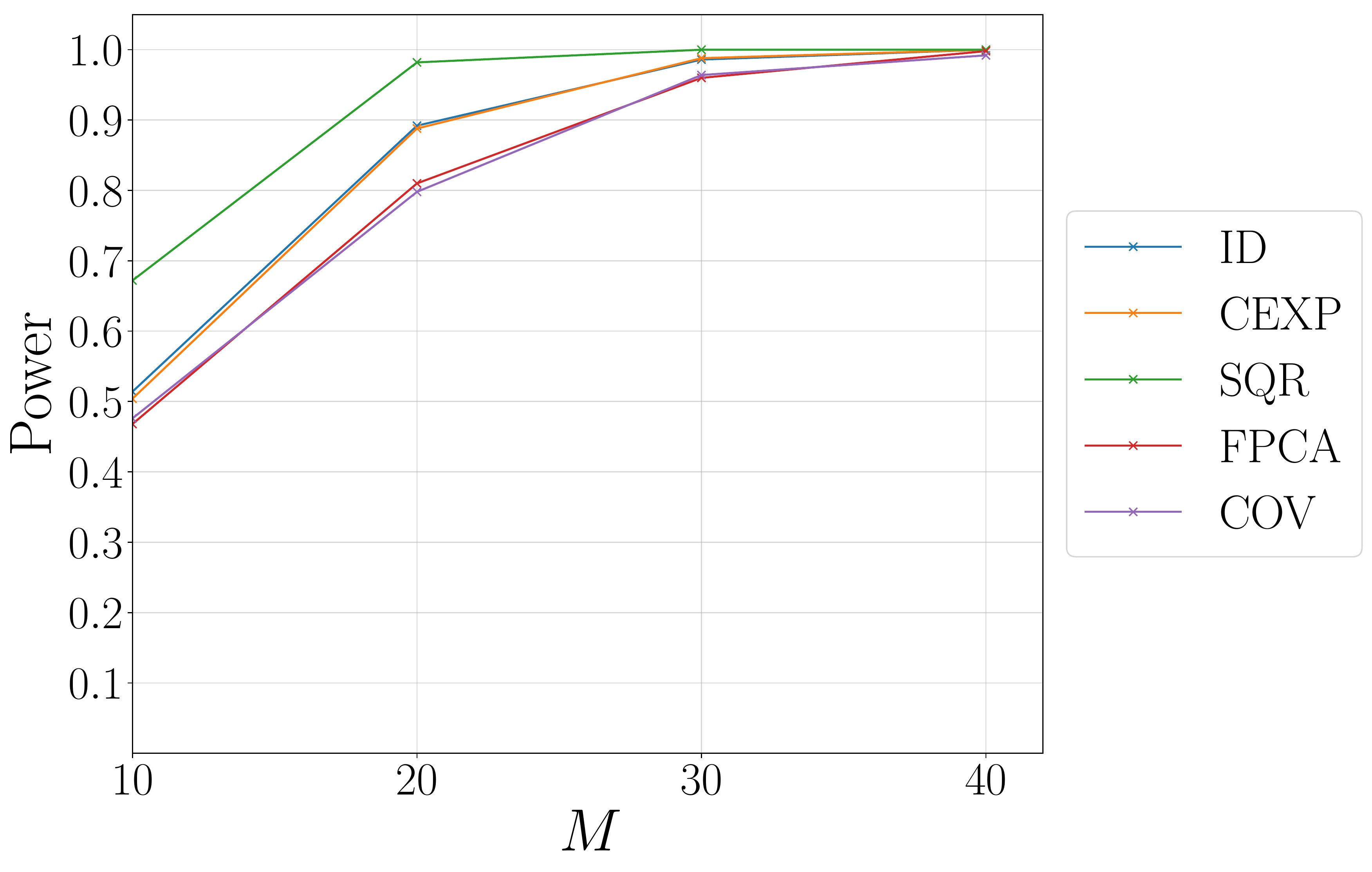}
\caption{Test power under subsamples of size $M$ using NEU steel data.}
\label{fig:steel}
\end{figure}

\begin{table}
\centering
  \begin{tabular}{l|ccccc}
     \toprule
    M & ID & CEXP & COV & SQR & FPCA \\
    \midrule
    10 & 4.2 & 4.2 & 5.0 & 5.8 & 5.8 \\
    20 & 4.8 & 4.8 & 6.6 & 5.8 & 6.2 \\
    30 & 5.2 & 6.4 & 6.0 & 6.2 & 6.6 \\
    40 & 6.4 & 7.0 & 5.2 & 4.4 & 4.8 \\
    \bottomrule
  \end{tabular}
  \caption{Empirical size, meaning the rate of rejection of the null when the null is true, of the two-sample test performed on the rolled-in scale class from the NEU steel defect data, for different sample sizes across different choices of $T$. The values are written as percentages, a value above $5$ shows too frequent rejection and below shows too conservative a test.}
  \label{fig:steel_null}
\end{table}

\section{Conclusion}
\label{sec:conclusion}

In this paper we studied properties of kernels on real, separable Hilbert spaces, and the associated Maximum Mean Discrepancy distance.   Based on this, we formulated a novel kernel-based two-sample testing procedure for functional data. The development of kernels on Hilbert spaces was motivated by the observation that certain scaling of kernel parameters in the finite dimensional regime can result in a kernel over a Hilbert space. Indeed, multiple theoretical properties emerged as natural infinite dimensional generalisations of the finite dimensional case. The development of kernels defined directly over Hilbert spaces facilitates the use of hyperparameters adapted for functional data, such as the choice of $T$ in the SE-$T$ kernel, which can result in greater test power.

While other nonparametric two-sample tests for functional data have been proposed recently, we believe that kernel-based approaches offer unique advantages.  In particular, the ability to choose the kernel to reflect \emph{a priori} knowledge about the data, such as any underlying  dependencies, or to emphasise at which spatial scales the comparison should be made between the samples can be of significant benefit to practitioners.   The construction of kernels which are tailor-made for two-sample testing of specific forms of functional data, for example time series and spatial data, is an interesting and open question, which we shall defer to future work.  

The theory of kernels on function spaces is of independent interest and our work highlights how existing results on probability measures on infinite dimensional spaces can be applied to kernel methods, for example the use of the Minlos-Sazanov theorem when proving characteristicness.  Recent theoretical developments relating to  kernels and MMD on general topological spaces in the absence of compactness or local compactness have revealed the challenges in establishing important properties of such metrics, for example determining weak convergence of sequences of probability measures \citep{Simon-Gabriel2018, simon2020metrizing, Chevyrev2018}, which have important implications for the development of effective MMD-based tests for functional data. 

 A further application of kernels on function spaces is statistical learning of maps between function spaces, in particular, the challenge of learning surrogate models for large-scale PDE systems which can be viewed as a nonlinear deterministic maps from an input function space, initial or boundary data, to an output function space, a system response.  Here there are fundamental challenges to be addressed relating to the universality properties of such kernels.   Preliminary work in \citet{nelsen2020random} indicates that this is a promising direction of research.

\subsection*{Acknowledgments}
GW was supported by an EPSRC Industrial CASE award [18000171] in partnership with Shell UK Ltd.  AD was supported by the Lloyds Register Foundation Programme on Data Centric Engineering and by The Alan Turing Institute under the EPSRC grant [EP/N510129/1]. We thank Sebastian Vollmer for helpful comments.

{\scriptsize
\setlength{\bibsep}{0.2pt}
\bibliographystyle{abbrvnat}
\bibliography{inf_dim_refs.bib}

}

\clearpage
\appendix
\section{Appendix} \label{sec:appendix}

\subsection{Bochner and Minlos-Sazanov Theorem}\label{subsec:minlos}
Bochner's theorem provides an exact relationship between continuous, translation invariant kernels on $\bbR^{d}$, meaning $k(x,y) = \phi(x-y)$ for some continuous $\phi$, and the Fourier transforms of finite Borel measures on $\bbR^{d}$. For a proof see \citet[Theorem 6.6]{Wendland2005}.

\begin{theorem}[\textbf{Bochner}]\label{thm:bochner}
	A continuous function $\phi\colon\bbR^{d}\rightarrow\bbC$ is positive definite if and only if it is the Fourier transform of a finite Borel measure $\mu_{\phi}$ on $\bbR^{d}$
	\begin{align*}
		\hat{\mu}_{\phi}(x)\coloneqq \int_{\bbR^{d}}e^{ix^{T}y}d\mu_{\phi}(y) = \phi(x).
	\end{align*}
\end{theorem}

Bochner's theorem does not continue to hold in infinite dimensions, for example the kernel $k(x,y) = e^{-\frac{1}{2}\norm{x-y}_{\calX}^{2}}$ when $\calX$ is an infinite dimensional, real, separable Hilbert space is not the Fourier transform of a finite Borel measure on $\calX$ \citep[Proposition 1.2.11]{Maniglia2004}. Instead, a stronger continuity property is required, this is the content of the Minlos-Sazanov theorem. For a proof see \citep[Theorem 1.1.5]{Maniglia2004} or \citet[Theorem \RomanNumeralCaps{6}.1.1]{Vakhania1987}.

\begin{theorem}[\textbf{Minlos-Sazanov}]\label{thm:minlos-sazanov}
	Let $\calX$ be a real, seperable Hilbert space and $\phi\colon\calX\rightarrow\bbC$ a positive definite function on $\calX$ then the following are equivalent
	\begin{enumerate}
		\item $\phi$ is the Fourier transform of a finite Borel measure on $\calX$
		\item There exists $C\in L^{+}_{1}(\calX)$ such that $\phi$ is continuous with respect to the norm induced by $C$ given by $\norm{x}_{C}^{2} = \langle Cx,x\rangle_{\calX}$.
	\end{enumerate}
\end{theorem}

The existence of such an operator is a much stronger continuity property than standard continuity on $\calX$ and will be crucial in the proof of Theorem \ref{thm:identity_char}, one of our main results. To see that continuity with respect to such a $C$ is stronger than usual continuity consider the following example. Fix any $\varepsilon > 0$ and assume we only know that $\phi\colon\calX\rightarrow\bbR$ is continuous and for simplicity assume that $\phi(0) = 1$, then we know there exists some $\delta > 0$ such that $\norm{x}_{\calX} < \delta$ implies $\lvert\phi(x) - 1\rvert < \varepsilon$ meaning we have control of $\phi(x)$ over the bounded set $\norm{x}_{\calX} <\delta$. On the other hand, if $\phi$ is continuous with respect to $\norm{\cdot}_{C}$ for some $C\in L^{+}_{1}(\calX)$ then we know there exists some $\delta' > 0$ such that $\norm{x}_{C} < \delta'$  implies $\lvert\phi(x) - 1\rvert < \varepsilon$ so we have control of $\phi(x)$ over the \emph{unbounded} set $\norm{x}_{C} < \delta'$. To see this set is unbounded let $\{\lambda_{n},e_{n}\}_{n=1}^{\infty}$ be the orthonormal eigensystem of $C$ and for $n\in\bbN$ let $y_{n} = \frac{\delta' e_{n}}{2\lambda_{n}}$ if $\lambda_{n} > 0$ otherwise $y_{n} = ne_{n}$, then since $C\in L^{+}_{1}(\calX)$ we know $\lambda_{n}\rightarrow 0$ so $\norm{y_{n}}_{\calX}\rightarrow\infty$. Since we used elements from the eigensystem it is clear that $\norm{y_{n}}_{C} \leq \delta'/2$. Therefore we have constructed a subset of the ball with respect to $\norm{\cdot}_{C}$ of radius $\delta'$ that has unbounded norm. 


\subsection{Proofs for Section \ref{sec:scaling}}\label{app:scaling}

\begin{proof}[\textbf{Proof of Proposition }\ref{prop:scaling}]

We begin by outlining the scaling that occurs for each of the two cases. For economy of notation we set $\gamma = \gamma_{N}$. If $k_{0}(s,t) = \delta_{st}$ then $\Tr(\Sigma^{i})= N$ for all $i\in\bbN$ and $\langle m_{N},\Sigma m_{N}\rangle = \norm{m_{N}}_{2}^{2} = \Theta(N)$ by the Riemann scaling. If $k_{0}$ is continuous and bounded then $\Tr(\Sigma^{i}) = \Theta(N^{i})$ since $N^{-i}\Tr(\Sigma^{i})\rightarrow \Tr(C_{k_{0}}^{i})$ and $\langle m_{N},\Sigma^{i}m_{N}\rangle = \Theta(N^{i+1})$ since $N^{-(i+1)}\langle m_{N},\Sigma^{i} m_{N}\rangle\rightarrow \int_{\calD}\int_{\calD}m(s)C_{k_{0}}^{i}m(t)dsdt$.

We shall outline the proof of the result in the second case, the proof for the first case is entirely analogous and is completed by substituting the above scaling results where appropriate. Throughout this proof we shall repeatedly use the following Taylor series approximation, for a positive definite matrix $\Sigma \in \mathbb{R}^{N\times N}$, based on the classical Girard-Waring formulas.
\begin{align*}
\det(I + \varepsilon\Sigma) & = 1 + \varepsilon \Tr(\Sigma) + \frac{\varepsilon^2}{2}R_1(s)\\
& = 1 + \varepsilon \Tr(\Sigma) + \frac{\varepsilon^2}{2}(\Tr(\Sigma)^2 - \Tr(\Sigma^2)) + \frac{\varepsilon^3}{6}R_2(s'),
\end{align*}
for some $0 < s, s'\leq \varepsilon$ where 
\begin{align*}
    R_1(s) = \left(\sum\frac{\lambda_{i}}{\left(1+s\lambda_{i}\right)}\right)^{2}-\det(I+s \Sigma)\sum\frac{\lambda_{i}^{2}}{(1+s\lambda_{i})^{2}}, 
\end{align*}
and
\begin{align*}
   R_2(s) =&-2(1 + \det(I+s \Sigma))\sum\frac{\lambda_{i}}{\left(1+s\lambda_{i}\right)}\sum\frac{\lambda_{i}^{2}}{(1+s\lambda_{i})^{2}} \\
   & +2\det(I+s \Sigma)\sum\frac{\lambda_{i}^{3}}{(1+s\lambda_{i})^{3}} ,
\end{align*}
where $\lbrace \lambda_1, \ldots, \lambda_N \rbrace$ are the eigenvalues of $\Sigma$. We may bound these remainders
\begin{align*}
    R_{1}(s) \leq \left(\sum\frac{\lambda_{i}}{\left(1+s\lambda_{i}\right)}\right)^{2} \leq \Tr(\Sigma)^{2} = O(N^{2}),
\end{align*}
and 
\begin{align*}
    R_{2}(s) \leq 2\det(I+s \Sigma)\sum\frac{\lambda_{i}^{3}}{(1+s\lambda_{i})^{3}} \leq 2e^{\varepsilon\Tr(\Sigma)}\Tr(\Sigma^{3}).
\end{align*}
In our scenario $\varepsilon = \gamma^{-2} = O(N^{-2\alpha}) = o(N^{-1})$ therefore $e^{\varepsilon\Tr(\Sigma)}$ is bounded since $\Tr(\Sigma) = O(N)$. Finally, $\Tr(\Sigma^{3}) = O(N^{3})$ therefore $R_{2}(s) = O(N^{3})$. Setting $T = \frac{1}{\gamma}I$ and $R = S=\Sigma$ in Corollary \ref{cor:MMD_two_GPs_commute} and applying the above Taylor expansion along with the expansions for $x^{-1/2}$ and $(I+\varepsilon\Sigma)^{-1}$ with Lagrange remainder term
\begin{align*}
    \MMD_{k}^2(P, Q) &= 2\det(I + 2\Sigma/\gamma^2)^{-\frac{1}{2}}\left(1 - e^{-\frac{1}{2\gamma^2}\langle (I + 2\Sigma/\gamma^2)^{-1}\mu, \mu \rangle }\right)\\
    &= \left(1 + \frac{1}{\gamma^2}r_1\right)\left(\frac{1}{\gamma^2}\left\langle (I + 2\Sigma/\gamma^2)^{-1}\mu, \mu \right\rangle + \frac{1}{\gamma^4}r_2\right)\\
    &= \left(1 + \frac{1}{\gamma^2}r_1\right)\left(\frac{1}{\gamma^2}\left\langle \mu, \mu \right\rangle + \frac{1}{\gamma^4}r_2 +  \frac{1}{\gamma^4}r_3\right),
\end{align*}
where $r_1,r_2,r_3$ are remainder terms.  \\

To leading order $r_1/\gamma^2=O(N^{1-2\alpha}), r_2/\gamma^4 =O(N^{2-4\alpha}),r_3/\gamma^4 =O( N^{2-4\alpha})$ so it follows that 
$\MMD_{k}^2(P, Q)   = \gamma^{-2}\lVert m_{N} \rVert_2 + r(N)$
where, to leading order, $r(N) =O( N^{2-4\alpha})$. We now consider the denominator. To this end, applying Theorem \ref{thm:variance_of_estimator_mean_shift} with $T = \gamma^{-1}I$ and $S = \Sigma$ we obtain
\begin{align*}
    \xi_2 &= 2\det(I+4\Sigma/\gamma^2)^{-\frac{1}{2}} \left( 1 + e^{-\frac{1}{\gamma^2}\langle \mu, (I+4\Sigma/\gamma^2)^{-1} \mu\rangle} \right) \\
    & -2\det(I + 2\Sigma/\gamma^2)^{-1}(1 + e^{-\frac{1}{\gamma^2}\langle \mu (I + 2\Sigma/\gamma^2)^{-1}\mu \rangle} - 4e^{-\frac{1}{2\gamma^2}\langle \mu (I + 2\Sigma/\gamma^2)^{-1}\mu \rangle}) \\
    & -8\det((I + \Sigma/\gamma^2)(I + 3\Sigma/\gamma^2))^{-1/2} e^{-\frac{1}{2\gamma^2}\langle \mu, (I + 2\Sigma/\gamma^2)(I+ 3\Sigma/\gamma^2)^{-1}(I+\Sigma/\gamma^2)^{-1}\mu\rangle}.
\end{align*}
We split this into two terms, $\xi_2 = A_1 + A_2$, where
\begin{align*}
    A_{1}  & =    4\det(I+4\Sigma/\gamma^2)^{-\frac{1}{2}} + 4\det(I + 2\Sigma/\gamma^2)^{-1}\\
    & -8\det((I + \Sigma/\gamma^2)(I + 3\Sigma/\gamma^2))^{-1/2},
\end{align*}
and 
\begin{align*}
    A_2 &= 2\det(I+4\Sigma/\gamma^2)^{-\frac{1}{2}} \left( -1 +  e^{-\frac{1}{\gamma^2}\langle \mu, (I+4\Sigma/\gamma^2)^{-1} \mu\rangle} \right) \\
    &-2\det(I + 2\Sigma/\gamma^2)^{-1}(3 + e^{-\frac{1}{\gamma^2}\langle \mu (I + 2\Sigma/\gamma^2)^{-1}\mu \rangle} - 4e^{-\frac{1}{2\gamma^2}\langle \mu (I + 2\Sigma/\gamma^2)^{-1}\mu \rangle}) \\
    & -8\det((I + \Sigma/\gamma^2)(I + 3\Sigma/\gamma^2))^{-1/2}\\
    &\cdot(-1 + e^{-\frac{1}{2\gamma^2}\langle \mu, (I + 2\Sigma/\gamma^2)(I+ 3\Sigma/\gamma^2)^{-1}(I+\Sigma/\gamma^2)^{-1}\mu\rangle} ).
\end{align*}
Applying Taylor's theorem for the determinant and exponential terms we obtain:
\begin{align*}
    A_1 &= 4\left[1 - \frac{2}{\gamma^2}\Tr(\Sigma) + \frac{2}{\gamma^4}(\Tr(\Sigma))^2 + \frac{4}{\gamma^4}\Tr(\Sigma^2)  + r_1\right]\\
    &+ 4\left[1 - \frac{2}{\gamma^2} \Tr(\Sigma) + \frac{2}{\gamma^4}(\Tr(\Sigma))^2 + \frac{2}{\gamma^4}\Tr(\Sigma^2) + r_2 \right] \\
    & - 8\left[1 - \frac{1}{2\gamma^2} \Tr(\Sigma) + \frac{1}{8\gamma^4}\Tr(\Sigma)^2 + \frac{1}{4\gamma^4}\Tr(\Sigma^2) + r_3 \right]\\ 
    & \times\left[1 - \frac{3}{2\gamma^2}\Tr(\Sigma) + \frac{9}{8\gamma^4}(\Tr(\Sigma))^2 + \frac{9}{4\gamma^4}\Tr(\Sigma^2)+ r_4\right],
\end{align*}
where to leading order $r_1, r_2, r_3, r_4 =O( {N^{3-6\alpha}} +N^{4-8\alpha})$.   After simplification we obtain $A_1 = \frac{4}{\gamma^4}\Tr(\Sigma^2) +r(N)$ where $r(N) =O(N^{3-6\alpha})$.  Similarly writing $A_2 = B_1 + B_2 + B_3$, we obtain:
\begin{align*}
B_1 &=  \frac{2}{\gamma^2}(1 - 2\Tr(\Sigma)/\gamma^2 + r_5)\left( -\langle \mu, \mu\rangle + \frac{4}{\gamma^2}\langle \mu, \Sigma \mu\rangle  + \frac{1}{2\gamma^2}\langle \mu,  \mu\rangle^2 + r_6\right),\\
 B_2  &= -\frac{2}{\gamma^2}(1 - 2\Tr(\Sigma)/\gamma^2 + r_7) \left(\langle \mu, \mu\rangle - \frac{2}{\gamma^2}\langle \mu, \Sigma \mu\rangle  +  r_8 \right),\\
 B_3 &= -\frac{8}{\gamma^2}\left(1 - 2\Tr(\Sigma)/\gamma^2 + r_9\right)\left(- \frac{1}{2}\langle \mu, \mu \rangle +\frac{1}{\gamma^2}\langle \mu, \Sigma \mu\rangle + \frac{1}{8\gamma^2}\langle \mu, \mu\rangle^2 + r_{10} \right),
\end{align*}

where $r_5, \ldots, r_{10}$ are remainder terms which satisfy $r_i =O(N^{3-6\alpha} + N^{4-8\alpha})$ for $i=5,\ldots, 10$.  Expanding $B_1 + B_2 + B_3$ and collecting into powers of $\gamma$, we see that the constant terms and the terms with denominator $\gamma^{2}$ cancel out. Collecting terms with $\gamma^{4}$ denominator gives $A_2  = \frac{4}{\gamma^4} \langle \mu, \Sigma \mu\rangle + r(N)$ where $r(d) =O(d^{3-6\alpha})$ is a remainder term containing the higher order terms.  Combining the leading order expressions for $A_1$ and $A_2$, and collecting together the remainder terms we obtain
$$
    \frac{\MMD^2_k(P,Q)}{\sqrt{\xi_{2}}} = \frac{\lVert \mu \rVert^2_2/\gamma^2 + r(N)}{\sqrt{4\Tr(\Sigma^2)/\gamma^4 + 4\langle \mu, \Sigma \mu \rangle/\gamma^4 + r'(N)}},
$$
where $r(N) =O(N^{2-4\alpha})$ and $r'(N)= O( N^{3-6\alpha})$
It follows that
\begin{equation}
\label{eq:mmd_xi_ratio_asymp}
\frac{\MMD^2_k(P,Q)}{\sqrt{\xi_{2}}}= \frac{\lVert \mu \rVert^2_2}{\sqrt{4\Tr(\Sigma^2) + 4\langle \mu, \Sigma \mu \rangle}}\left(\frac{1+ \frac{r(N)}{\lVert \mu \rVert^2_2/\gamma^2}}{\sqrt{1 + \frac{r'(N)}{4\Tr(\Sigma^2)/\gamma^4 + 4\langle \mu, \Sigma \mu \rangle/\gamma^4}}}\right).
\end{equation}
As discussed at the start of the proof $\lVert m_{N} \rVert^2_2=\Theta(N)$ and $ \Tr(\Sigma^{2}),\langle m_{N},\Sigma m_{N}\rangle = \Theta(N^{2})$ meaning that $\gamma^{2}r(N)/\norm{m_{N}}_{2}^{2} = O(N^{1-2\alpha})$ and so converges to zero. Likewise the fraction in the square root in the denominator converges to zero, meaning the term in the brackets of \eqref{eq:mmd_xi_ratio_asymp} converge to $1$ as $N\rightarrow \infty$, yielding the advertised result once numerator and denominator is divided by $N$ since
\begin{align*}
    N^{-2}\Tr(\Sigma^2) &\rightarrow \int_{\calD} \int_{\calD} k_{0}(s,t)dsdt = \lVert C_{k} \rVert^2_{HS},\\
N^{-2}\langle m_{N}, \Sigma m_{N}\rangle &\rightarrow \int_{\calD} \int_{\calD} m(s)k_{0}(s,t) m(t)dsdt = \lVert C_{k}^{1/2}m \rVert^2_{L^{2}(\calD)}.
\end{align*}

For the case of identity matrix substituting the corresponding scaling of $\Tr(\Sigma^{2}),\langle m_{N},\Sigma m_{N}\rangle$ in the relevant places will produce the result. 

\end{proof}

\begin{proof}[\textbf{Proof of Lemma \ref{lem:GP_median_heur}}]

We use  the standard fact that for a real valued random variables $\eta$ the following inequality holds $\lvert \bbE[\eta] - \text{Median}[\eta]\rvert\leq \text{Var}[\eta]^{\frac{1}{2}}$ and we will be using this inequality on $\norm{x-y}_{2}^{2}$ where $x\sim P, y\sim Q$.  For economy of notation let $\mu = \mu_{1}-\mu_{2}$ and $\Sigma = \Sigma_{1}+\Sigma_{2}$.
Using standard Gaussian integral identities we obtain
\begin{align*}
    \bbE[\norm{x-y}_{2}^{2}]^{2} & = (\Tr(\Sigma)+\norm{\mu}_{2}^{2})^{2},\\
    \bbE[\norm{x-y}_{2}^{4}] & = 2\Tr(\Sigma^{2}) + 4\langle\mu,\Sigma\mu\rangle + (\Tr(\Sigma)+\norm{\mu}_{2}^{2})^{2},
\end{align*}
therefore 
\begin{align*}
 \text{Var}[\norm{x-y}_{2}^{2}] & = \bbE[\norm{x-y}_{2}^{4}] - \bbE[\norm{x-y}_{2}^{2}]^{2} = 2\Tr(\Sigma)^{2} + 4\norm{\mu}_{2}^{2}\Tr(\Sigma).
\end{align*}  
Substituting into the inequality at the start of the proof
\begin{align*}
    \left\lvert\frac{\text{Median}[\norm{x-y}_{2}^{2}]}{\bbE[\norm{x-y}_{2}^{2}]}-1\right\rvert^{2} & \leq \frac{2\Tr(\Sigma)^{2} + 4\norm{\mu}_{2}^{2}\Tr(\Sigma)}{(\Tr(\Sigma)+\norm{\mu}_{2}^{2})^{2}} = 2\left(1-\frac{\norm{\mu}_{2}^{4}}{(\Tr(\Sigma)+\norm{\mu}_{2}^{2})^{2}}\right),
\end{align*}
which completes the first part of the proof. The special case follows from dividing the numerator and denominator of the fraction in the right hand side by $N^{2}$ and using the assumption that the kernels are continuous and the mesh satisfies the Riemann scaling property. Writing $m_{N}$ for the discretised version of $m_{1}-m_{2}$ gives  $N^{-2}\norm{m_{N}}_{2}^{4}\rightarrow \norm{m_{1}-m_{2}}_{L^{2}(\calD)}^{2}, N^{-1}\Tr(\Sigma)\rightarrow\Tr(C_{k_{1}}+C_{k_{2}}), N^{-1}\norm{ m_{N}}_{2}^{2}\rightarrow\norm{m_{1}-m_{2}}_{L^{2}(\calD)}^{2}$. 
\end{proof}

\subsection{Proofs for Section \ref{sec:kernel}}

\begin{lemma}\label{lem:pos_def}
	The SE-$T$ function is a kernel.
\end{lemma}
\begin{proof}[\textbf{Proof of Lemma \ref{lem:pos_def}}]
	Consider first the case $T=I$, it is shown in \citet[Theorem 3]{Schoenberg1938} that if $k(x,y) = \phi(\norm{x-y}_{\calY}^{2})$ for $\phi$ a completely monotone function then $k$ is positive definite on $\calY$ and it is well known that $e^{-ax}$ is such a function for $a >  0$ therefore $k_{I}$ is a kernel. Now take $k_{T}$ to be the SE-$T$ kernel then for any $N\in\bbN, \{a_{n}\}_{n=1}^{N}\subset\bbR, \{x_{n}\}_{n=1}^{N}\subset\calX$ we have $\sum_{n,m=1}^{N}a_{n}a_{m}k_{T}(x_{n},x_{m}) = \sum_{n,m=1}^{N}a_{n}a_{m}k_{I}(T(x_{n}),T(x_{m})) \geq 0$
	\end{proof}

\begin{proof}[\textbf{Proof of Theorem \ref{thm:gauss_rkhs}}]\label{proof:gauss_rkhs}
This proof uses the argument of \citet[Theorem 1]{Minh2009}. The plan is to first show the function space stated in the theorem is an RKHS and that the kernel is the SE-$T$ kernel so by uniqueness of kernel for RKHS we are done. This is done using the Aronszajn theorem \citep[Theorem 9]{Minh2009} which identifies the kernel as an infinite sum of basis functions, see also \citet[Theorem 2.4]{Paulsen2016}.

First we prove that $\calH_{k_{T}}(\calX)$ is a separable Hilbert space. That it is an inner product space is clear from the definition of the inner product and the assumption that $\lambda_{n} > 0$ for all $n\in\bbN$. The definition of the inner product means completeness of $\calH_{k_{T}}(\calX)$ equivalent to completeness of the weighted $l^{2}$ space given by 
$$l^{2}_{\lambda}(\Gamma) = \left\{(w_{\gamma})_{\gamma\in\Gamma}\colon \norm{(w_{\gamma})_{\gamma\in\Gamma}}_{l^{2}_{\lambda}(\Gamma)}^{2} \coloneqq \sum_{\gamma\in\Gamma}\frac{\gamma !}{\lambda^{\gamma}}w_{\gamma}^{2} < \infty\right\}
$$
which can be easily seen to be complete since $\Gamma$ is countable and $\gamma !/\lambda^{\gamma}$ is positive for all $\gamma\in\Gamma$. To see that $\calH_{k_{T}}(\calX)$ is separable observe that from the definition of the inner product, the countable set of functions $\phi_{\gamma}(x) = \sqrt{\lambda^{\gamma}/\gamma !}e^{-\frac{1}{2}\norm{Tx}_{\calX}^{2}}x^{\gamma}$
is orthonormal and spans $\calH_{k_{T}}(\calX)$ hence is an orthonormal basis. 

Next we prove that $H_{k_{T}}(\calX)$ is an RKHS. Expanding the kernel through the exponential function gives
\begin{align*}
	k_{T}(x,y) & = e^{-\frac{1}{2}\norm{Tx}_{\calX}^{2}}e^{-\frac{1}{2}\norm{Tx}_{\calX}^{2}}e^{\langle Tx,Ty\rangle_{\calX}} = e^{-\frac{1}{2}\norm{Tx}_{\calX}^{2}}e^{-\frac{1}{2}\norm{Tx}_{\calX}^{2}}\sum_{n=0}^{\infty}\frac{\langle Tx,Ty\rangle_{\calX}^{n}}{n!},
\end{align*} 
and by the assumption on $T$ we know $\langle Tx,Ty\rangle_{\calX}^{n} = \big(\sum_{m=1}^{\infty}\lambda_{m}x_{m}y_{m}\big)^{n} = \sum_{\lvert\gamma\rvert = n}\frac{n!}{\gamma !}\lambda^{\gamma}x^{\gamma}y^{\gamma}$ where $x_{m} = \langle x,e_{m}\rangle_{\calX}$, similarly for $y_{m}$, therefore
\begin{align*}
	k_{T}(x,y) = e^{-\frac{1}{2}\norm{Tx}_{\calX}^{2}}e^{-\frac{1}{2}\norm{Tx}_{\calX}^{2}}\sum_{\lvert\gamma\rvert \geq 0}\frac{\lambda^{\gamma}}{\gamma !}x^{\gamma}y^{\gamma} = \sum_{\lvert\gamma\rvert \geq 0}\phi_{\gamma}(x)\phi_{\gamma}(y).
\end{align*}
So for any $F\in \calH_{k_{T}}(\calX)$ we have $\langle F,k_{T}(\cdot,x)\rangle_{\calH_{k_{T}}(\calX)} = \sum_{\gamma\in\Gamma}\langle F,\phi_{\gamma}\rangle_{\calH_{k_{T}}(\calX)}\phi_{\gamma}(x) = F(x)$ so $k_{T}$ is a reproducing kernel of $\calH_{k_{T}}(\calX)$ so by uniqueness of reproducing kernels we have that $\calH_{k_{T}}(\calX)$ is the RKHS of $k_{T}$.
\end{proof} 


\begin{proof}[\textbf{Proof of Theorem \ref{thm:trace_class_char}}]\label{proof:trace_class_char} 
If  $k_{T}(x,y) = \hat{\mu}(x-y)$ for some Borel measure on $\calX$ then \eqref{eq:MMD_integral_version} lets us write 
\begin{align}
	&\MMD_{k_{T}}(P,Q)^{2}\nonumber \\
	& = \int_{\calX}\int_{\calX}k_{T}(x,y)d(P-Q)(x)d(P-Q)(y) \nonumber\\
	 & =\int_{\calX}\int_{\calX}\int_{\calX}e^{i\langle h,x-y\rangle_{\calX}}d\mu(h)d(P-Q)(x)d(P-Q)\nonumber \\
	 & = \int_{\calX}\left(\int_{\calX}e^{i\langle h,x\rangle_{\calX}}d(P-Q)(x)\right)\left(\int_{\calX}e^{-i\langle h,y\rangle_{\calX}}d(P-Q)(y)\right)d\mu(h) \label{eq:fubini}\\
	 & = \int_{\calX}\left\lvert \widehat{P}(h)-\widehat{Q}(h)\right\rvert^{2}d\mu(h),\nonumber
\end{align}
where \eqref{eq:fubini} is obtained by using Fubini's theorem to swap the integrals which is permitted since $\lvert e^{i\langle h,x-y\rangle_{\calX}}\lvert = 1$ and is therefore integrable with respect to $P$ and $Q$. 

Fourier transforms of finite Borel measures on $\calX$ are uniformly continuous \citep[Proposition 2.21]{Albeverio2015} therefore if $\mu$ has full support, meaning that $\mu(U) > 0$ for every open $U\subset\calX$, then we may immediately conclude that $\widehat{P} = \widehat{Q}$ and since the Fourier transform of finite Borel measures on $\calX$ is injective \citep[Proposition 1.7]{DaPrato2006} we may conclude that $P = Q$ meaning that $\Phi_{k_{T}}$ is injective. Assume $T$ is injective. If $k_{T}$ is the SE-$T$ kernel then \citet[Proposition 1.25]{DaPrato2006} shows that $\mu = N_{T^{2}}$ has full support since $T^{2}$ is also injective. Therefore $k_{T}$ is characteristic. For the converse direction we use the contrapositive and assume that $T$ is not injective meaning there exists $x^{*}\in\calX$ with $x^{*}\neq 0$ and $Tx = 0$. Set $P = \delta_{0}$ and $Q = \delta_{x^{*}}$ the Dirac measures on $0,x^{*}$ then $\Phi_{k_{T}}(P) = k_{T}(0,\cdot) = k_{T}(x^{*},\cdot) = \Phi_{k_{T}}(Q)$ therefore $\Phi_{k_{T}}$ is not injective so $k_{T}$ is not characteristic. 
\end{proof}


\begin{proof}[\textbf{Proof of Theorem \ref{thm:admissible_char}}]\label{proof:identity_char}
The result is a corollary of the next theorem which is where Theorem \ref{thm:minlos-sazanov} is employed. 

\begin{theorem}\label{thm:identity_char}
	Let $\calX$ be a real, separable Hilbert space and  $T=I$ then the SE-$T$ kernel is characteristic.
\end{theorem}

\begin{proof}[\textbf{Proof of Theorem \ref{thm:identity_char}}]
The idea of this proof is to use the contrapositive and assume $P\neq Q$ and conclude that $\MMD_{k_{I}}(P,Q) > 0$. The main tool shall be Theorem \ref{thm:minlos-sazanov} since $P\neq Q$ implies $\widehat{P}\neq \widehat{Q}$ and Theorem \ref{thm:minlos-sazanov} implies that these Fourier transforms vary slowly in some sense so there will be a set of big enough measure such that $\widehat{P}(x)\neq \widehat{Q}(x)$ for $x$ in this set, which will allow us to deduce $\MMD_{k_{I}}(P,Q) > 0$.

Suppose $P,Q$ are Borel probability measures on $\calX$ with $P\neq Q$ then $\widehat{P}\neq\widehat{Q}$ \citep[Proposition 1.7]{DaPrato2006} so there exists $x^{*}\in\calX, \varepsilon > 0$ such that $\lvert\widehat{P}(x^{*}) -\widehat{Q}(x^{*})\rvert > \varepsilon$. By Theorem \ref{thm:minlos-sazanov} there exists $S,R\in L^{+}_{1}(\calX)$ such that $\widehat{P}$ (respectively $\widehat{Q}$) is continuous with respect to the norm induced by $S$ (repectively $R$). For $r > 0$ let $B_{S}(x^{*},r) = \{x\in\calX\colon \langle S(x-x^{*}),x-x^{*}\rangle_{\calX} < r^{2}\}$ be the ball based at $x^{*}$ of radius $r$ with respect to the norm induced by $S$, and $B_{R}(x^{*},r)$ will denote the analogous ball with respect to the norm induced by $R$. By the continuity properties of $\widehat{P},\widehat{Q}$ there exists $r > 0$ such that
\begin{align*}
x\in B_{S}(x^{*},r) & \implies \lvert\widehat{P}(x)-\widehat{P}(x^{*})\rvert < \frac{\varepsilon}{4}\\
x\in B_{R}(x^{*},r) & \implies \lvert\widehat{Q}(x)-\widehat{Q}(x^{*})\rvert < \frac{\varepsilon}{4}.
\end{align*}
The set $B_{S}(x^{*},r)\cap B_{R}(x^{*},r)$ is non-empty since it contains $x^{*}$ and if $x\in B_{S}(x^{*},r)\cap B_{R}(x^{*},r)$ then by reverse triangle inequality
\begin{align}
	\lvert \widehat{P}(x)-\widehat{Q}(x)\rvert & = \lvert \widehat{P}(x)-\widehat{P}(x^{*}) + \widehat{P}(x^{*}) - \widehat{Q}(x^{*})+\widehat{Q}(x^{*}) - \widehat{Q}(x)\rvert \nonumber\\
	& \geq \lvert \widehat{P}(x^{*})-\widehat{Q}(x^{*})\rvert - \lvert\widehat{P}(x)-\widehat{P}(x^{*})\rvert - \lvert\widehat{Q}(x)-\widehat{Q}(x^{*})\rvert\nonumber\\
	& > \varepsilon - \frac{\varepsilon}{4} - \frac{\varepsilon}{4} = \frac{\varepsilon}{2}. \label{eq:varepsilon_lower_bound}
\end{align}

Now define the operator $U = S + R$ which is positive, symmetric and trace class since both $S$ and $R$ have these properties. Note that $B_{U}(x^{*},r)\subset B_{S}(x^{*},r)\cap B_{R}(x^{*},r)$ because
\begin{align*}
	\norm{x-x^{*}}_{U}^{2} &= \langle U(x-x^{*}),x-x^{*}\rangle_{\calX} \\
	&= \langle S(x-x^{*}),x-x^{*}\rangle_{\calX} + \langle R(x-x^{*}),x-x^{*}\rangle_{\calX} \\
	& = \norm{x-x^{*}}_{S}^{2} + \norm{x-x^{*}}_{R}^{2}.
\end{align*}

Since $U$ is a positive, compact, symmetric operator there exists a decomposition into its eigenvalues $\{\lambda_{n}\}_{n=1}^{\infty}$, a non-negative sequence converging to zero, and eigenfunctions $\{e_{n}\}_{n=1}^{\infty}$ which form an orthonormal basis of $\calX$. We will later need to associate a non-degenerate Gaussian measure with $U$. To this end define $V$ to be the positive, symmetric, trace class operator with eigenvalues $\{\rho_{n}\}_{n=1}^{\infty}$ where $\rho_{n} = \lambda_{n}$ if $\lambda_{n} > 0$ otherwise $\rho_{n} = n^{-2}$, and eigenfunctions $\{e_{n}\}_{n=1}^{\infty}$ inherited from $U$. Then by construction $V$ is injective, positive, symmetric and trace class. The $V$ induced norm dominates the $U$ induced norm therefore $B_{V}(x^{*},r)\subset B_{U}(x^{*},r)$ so for $x\in B_{V}(x^{*},r)$ we have $\lvert\widehat{P}(x)-\widehat{Q}(x)\rvert > \varepsilon/2$.

Now we construct an operator which will approximate $I$, define the operator $I_{m}x = \sum_{n=1}^{\infty}\omega_{n}^{(m)}\langle x,e_{n}\rangle_{\calX} e_{n}$ where $\omega_{n}^{(m)} = 1$ for $n\leq m$ and $\omega_{n}^{(m)} = n^{-2}$ for $n > m$ and $\{e_{n}\}_{n=1}^{\infty}$ is the eigenbasis of $V$, then $I_{m}\in L^{+}_{1}(\calX)$ for every $m\in\bbN$. It is easy to see $k_{I_{m}^{1/2}}$ converges pointwise to $k_{I}$ as $m\rightarrow\infty$ since $e^{-\frac{1}{2}x}$ is a continuous function on $\bbR$ and $\norm{x}_{I_{m}^{1/2}}^{2}\rightarrow\norm{x}^{2}_{I} = \norm{x}_{\calX}^{2}$ however clearly $I_{m}^{1/2}$ does not converge in operator norm to $I$ since $I$ is not a compact operator. Since $k_{I_{m}} \leq 1$ for all $m$ we may use the bounded convergence theorem to obtain
\begin{align}
	\MMD_{k_{I}}(P,Q)^{2} & = \int_{\calX}\int_{\calX}k_{I}(x,y)d(P-Q)(x)d(P-Q)(y) \nonumber\\
	& = \lim_{m\rightarrow\infty}\int_{\calX}\int_{\calX}k_{I_{m}^{1/2}}(x,y)d(P-Q)(x)d(P-Q)(y) \nonumber\\
	& = \lim_{m\rightarrow\infty}\int_{\calX}\lvert\widehat{P}(x)-\widehat{Q}(x)\rvert^{2}dN_{I_{m}}(x),\label{eq:MMD_spec}
\end{align}
where \eqref{eq:MMD_spec} is by the same reasoning as in the proof of Theorem \ref{thm:trace_class_char}. In light of the lower bound we derived earlier over $B_{V}(x^{*},r)$ of the integrand in \eqref{eq:varepsilon_lower_bound} 
\begin{align*}
	\MMD_{k_{I}}(P,Q)^{2} & = \lim_{m\rightarrow\infty}\int_{\calX}\lvert\widehat{P}(x)-\widehat{Q}(x)\rvert^{2}dN_{I_{m}}(x) \geq \lim_{m\rightarrow\infty}\int_{B_{V}(x^{*},r)}\frac{\varepsilon^{2}}{4}dN_{I_{m}}(x),
\end{align*}
so if we can lower bound $N_{I_{m}}(B_{V}(x^{*},r))$  by a positive constant indepdendent of $m$ then we are done. This set is the ball with respect to $V\in L^{+}_{1}(\calX)$ which is a somehow large set, see the discussion after Theorem \ref{thm:minlos-sazanov}, and we will use a push-forward of measure argument.

Define $T(x) = x-x^{*}$ then $N_{I_{m}}(B_{V}(x^{*},r)) = T_{\#}N_{I_{m}}(B_{V}(0,r))$ and \citep[Proposition 1.17]{DaPrato2006} tells us $T_{\#}N_{I_{m}}(B_{V}(0,r)) = N_{-x^{*},I_{m}}(B_{V}(0,r))$. Next we note that $N_{-x^{*},I_{m}}(B_{V}(0,r)) = V^{\frac{1}{2}}_{\#}N_{-x^{*},I_{m}}(B(0,r))$ and \citep[Proposition 1.18]{DaPrato2006} tells us that 
\begin{align*}
 V^{\frac{1}{2}}_{\#}N_{-x^{*},I_{m}}(B(0,r)) = N_{-V^{\frac{1}{2}}x^{*},V^{\frac{1}{2}}I_{m}V^{\frac{1}{2}}}(B(0,r)).
\end{align*}

For ease of notation let $y^{*} = -V^{\frac{1}{2}}x^{*}$ and since we choose to construct $I_{m}$ from the eigenbasis of $V$ we have $V_{m}x\coloneqq V^{\frac{1}{2}}I_{m}V^{\frac{1}{2}}x = \sum_{n\in\bbN}\rho_{n}^{(m)}\langle x,e_{n}\rangle_{\calX} e_{n}$ where $\rho_{n}^{(m)} = \rho_{n}$ for $n\leq m$ and $\rho_{n}^{(m)} = \rho_{n}n^{-2}$ for $n > m$ so $V_{m}\in L^{+}_{1}(\calX)$ and is injective for every $m\in\bbN$. We follow the proof of \citep[Proposition 1.25]{DaPrato2006} and define the sets 
\begin{align*}
	A_{l} = \left\{x\in\calX\colon\sum_{n=1}^{l}\langle x,e_{n}\rangle_{\calX}^{2} \leq \frac{r^{2}}{2}\right\} \quad
	B_{l} = \left\{x\in\calX\colon\sum_{n=l+1}^{\infty}\langle x,e_{n}\rangle_{\calX}^{2} < \frac{r^{2}}{2}\right\}.
\end{align*}
Since $V_{m}$ is non-degenerate for every $m\in\bbN$ we have that for every $l\in\bbN$ the events $A_{l}, B_{l}$ are independent under $N_{y^{*},V_{m}}$ \citep[Example 1.22]{DaPrato2006} meaning $\forall m,l\in\bbN$ we have
\begin{align*}
	N_{y^{*},V_{m}}(B(0,r)) \geq N_{y^{*},V_{m}}(A_{l}\cap B_{l}) = N_{y^{*},V_{m}}(A_{l})N_{y^{*},V_{m}}(B_{l}),
\end{align*}
and by the measure theoretic Chebyshev inequality, for every $l\in\bbN$
\begin{align*}
	N_{y^{*},V_{m}}(B_{l})  \geq 1 - N_{y^{*},V_{m}}(B^{c}_{l}) & \geq 1 - \frac{2}{r^{2}}\sum_{n=l+1}^{\infty}\int_{\calX}\langle x,e_{n}\rangle_{\calX}^{2} dN_{y^{*},V_{m}} \\
	& = 1 - \frac{2}{r^{2}}\left(\sum_{n=l+1}\rho^{(m)}_{n} + \langle y^{*},e_{n}\rangle_{\calX}^{2}\right) \\
	& \geq 1 - \frac{2}{r^{2}}\left(\sum_{n=l+1}^{\infty}\rho_{n} + \langle y^{*},e_{n}\rangle_{\calX}^{2}\right).
\end{align*}

As the final line involves the tail of a finite sum with no dependcy on $m$ there exists an $L\in\bbN$ such that $N_{y^{*},V_{m}}(B_{L}) > \frac{1}{2}$ for every $m\in\bbN$ and $l\geq L$. Note that for $m > L$ we have $N_{y^{*},V_{m}}(A_{L}) = N_{y^{*},V}(A_{L})$ since $A_{L}$ depends only on the first $L$ coordinates and $\rho_{n}^{(m)} = \rho_{n}$ for $n\leq L$ if $m > L$. So for $m > L$
\begin{align*}
	N_{y^{*},V_{m}}(A_{L}) = N_{y^{*},V}(A_{L}) \geq N_{y^{*},V}\left(B\left(0,\frac{r}{\sqrt{2}}\right)\right) > c > 0,
\end{align*}
for some $c$ since $V$ is non-degenerate \citep[Proposition 1.25]{DaPrato2006}. 

Overall we have shown that there exists an $L\in\bbN$ such that for $m > L$ we have $N_{I_{m}}(B_{V}(x^{*},r)) = N_{y^{*},V_{m}}(B(0,r)) > \frac{c}{2}$. Therefore, by substituting back into the lower bound for $\MMD_{k_{I}}(P,Q)^{2}$
\begin{align*}
	\MMD_{k_{I}}(P,Q)^{2} & = \lim_{m\rightarrow\infty}\int_{\calX}\lvert\widehat{P}(x)-\widehat{Q}(x)\rvert^{2}dN_{I_{m}}(x) \\
	& \geq \lim_{m\rightarrow\infty}\int_{B_{V}(x^{*},r)}\frac{\varepsilon^{2}}{4}dN_{I_{m}}(x)  > \frac{\varepsilon^{2}c}{8} > 0,
\end{align*}
which implies by contrapositive that $k_{I}$ is characteristic. 
\end{proof}

With Theorem \ref{thm:identity_char} proved we proceed to prove Theorem \ref{thm:admissible_char}. By \eqref{eq:MMD_integral_version}
\begin{align}
	\text{MMD}_{k_{T}}& (P,Q)^{2}  \nonumber\\
	& = \int \int k_{T}(x,x')dP(x)dP(x') + \int\int k_{T}(y,y')dQ(y)dQ(y')\nonumber \\
	&\qquad-2\int\int k_{T}(x,y)dQ(x)dP(y) \nonumber \\
	& = \int \int k_{I}(x,x')dT_{\#}P(x)dT_{\#}P(x') + \int\int k_{I}(y,y')dT_{\#}Q(y)dT_{\#}Q(y')\nonumber \\
	&\qquad-2\int\int k_{I}(x,y)dT{\#}P(x)dT_{\#}Q(y) \\
	& = \MMD_{k_{I}}(T_{\#}P,T_{\#}Q)^{2}.\nonumber
\end{align}
Using Theorem \ref{thm:identity_char} we know that if $\text{MMD}_{k_{T}}(P,Q) = 0$ then $T_{\#}P = T_{\#}Q$ and all that is left to show is that the assumption on $T$ implies $P = Q$. By the definition of push-forward measure we know that for every $B\in\calB(\calY)$ we have the equality $P(T^{-1}B) = Q(T^{-1}B)$. By the assumptions on $\calX,\calY, T$ we know that $T(A)\in\calB(\calY)$ for every $A\in\calB(\calX)$ \citep[Theorem 15.1]{Kechris1995}. Hence for any $A\in\calB(\calX)$ take $B = T(A)$ then $P(A) = P(T^{-1}B) = Q(T^{-1}B) = Q(A)$, which shows $P=Q$. 
\end{proof}

\begin{proof}[\textbf{Proof of Proposition \ref{prop:integral_kernel}}]
For any $N\in\bbN$ take any $\{a_{n}\}_{n=1}^{N}\subset\bbR,\{x_{n}\}_{n=1}^{N}\subset\calX$ then
\begin{align*}
	\sum_{n,m=1}^{N}k_{C,k_{0}}(x_{n},x_{m}) = \int_{\calX}\sum_{n,m=1}^{N}k_{0}(\langle x_{n},h\rangle_{\calX},\langle x_{m},h\rangle_{\calX})dN_{C}(h) \geq 0,
\end{align*}
since this is the integral of a non-negative quantity as $k_{0}$ is a kernel. Symmetry of $k_{C,k_{0}}$ follows since $k_{0}$ is symmetric meaning $k_{C,k_{0}}$ is a kernel. Expanding $k_{0}$ using its spectral measure we have
\begin{align*}
	k_{C,k_{0}}(x,y) = \int_{\calX}\int_{\bbR}e^{i\langle x-y,rh\rangle_{\calX}}d\mu(r)dN_{C}(h) = \hat{\nu}(x-y),
\end{align*}
where $\nu(A) = \int_{\calX}\int_{\bbR}\mathds{1}_{A}(rh)d\mu(r)dN_{C}(h)$ for all $A\in\calB(\calX)$. This is the law of the $\calX$ valued random variable $\xi X$ where $\xi\sim \mu$ and $X\sim N_{C}$ are independent. We will show that $\nu$ has full support on $\calX$ from which is follows that $k_{C,k_{0}}$ is characteristic by following the proof of Theorem \ref{thm:trace_class_char}.

Fix any open ball $B = B(h,r)\subset\calX$ then given the assumption on $\mu$ by intersecting with $(0,\infty)$ or $(-\infty,0)$ we may assume that $a,b$ have the same sign. Assume that $a,b > 0$, the proof for when $a,b < 0$ is analogous. We first treat the case $h\neq 0$. Set $\delta = \min(\frac{1}{2}(\frac{b}{a}-1),\frac{r}{2\norm{h}_{\calX}})$ so that $(a,a(1+\delta))\subset(a,b)$. Now consider the ball $B' = B(\frac{h}{a(1+\delta)},\frac{r}{4a(1+\delta)})$, take any $c\in(a,a(1+\delta))$ and any $x\in B'$ then
\begin{align*}
	\norm{cx-h}_{\calX}& \leq \norm[\bigg]{\xi x - \frac{\xi h}{a(1+\delta)}}_{\calX} + \norm[\bigg]{\frac{\xi h}{a(1+\delta)} - h}_{\calX} \\
	&\leq \frac{cr}{4a(1+\delta)} + \norm{h}_{\calX}\bigg(1-\frac{c}{a(1+\delta)}\bigg)\\
	& < \frac{r}{4} + \norm{h}_{\calX}\bigg(1-\frac{1}{1+\delta}\bigg) < \frac{r}{4}+\frac{r}{2} < r.
\end{align*}
Therefore for any $c\in(a,a(1+\delta))$ we have $cB'\subset B$. Hence $\bbP(\xi X\in B) = \nu(B)\geq \mu\big((a,a(1+\delta))\big)N_{C}(B') > 0$ by the assumptions on $\mu$ and the way $N_{C}$ is non-degenerate. 

The case $h = 0$ is analogous, take $B' = B(0,\frac{r}{2b})$ then for every $c\in(a,b)$ we have $cB'\subset B$ and we again conclude $\nu(B) > 0$. 
\end{proof}

\begin{proof}[\textbf{Proof of Corollary \ref{cor:IMQ_characteristic}}] The idea of the proof is to represent the IMQ-$T$ kernel as an integral of the SE-$T$ kernel then use the same limit argument as in the proof of Theorem \ref{thm:identity_char} and push-forward argument of Theorem \ref{thm:admissible_char}.

Throughout this proof $k^{\text{IMQ}}_{T}$ and $k^{\text{SE}}_{T}$ will denote the IMQ-$T$ and SE-$T$ kernels respectively. By the same proof technique as Theorem \ref{thm:admissible_char} is suffices to prove that $k^{\text{IMQ}}_{I}$ is characteristic. Let $\{e_{n}\}_{n=1}^{\infty}$ be any orthonormal basis of $\calX$ and let $I_{m}x = \sum_{n=1}^{m}\langle x,e_{n}\rangle_{\calX} e_{n}$ so that $I_{m}$ converges to $I$ pointwise. Then by the same limiting argument in the proof of Theorem \ref{thm:identity_char} using bounded convergence theorem, letting $N_{1}$ be the $\calN(0,1)$ measure on $\bbR$ we have
\begin{align}
	\MMD_{k^{\text{IMQ}}_{I}}(P,Q)^{2} & = \lim_{m\rightarrow\infty}\int_{\calX}\int_{\calX}k_{I_{m}}^{\text{IMQ}}(x,y)d(P-Q)(x)d(P-Q)(y)\nonumber \\
	& = \lim_{m\rightarrow\infty}\int_{\calX}\int_{\calX}\int_{\bbR}k_{I_{m}}^{\text{SE}}(zx,zy)dN_{1}(z)d(P-Q)(x)d(P-Q)(y) \label{swap_IMQ} \\ 
	& = \int_{\bbR}\MMD_{k_{I}^{\text{SE}}}(z_{\#}P,z_{\#}Q)^{2}dN_{1}(z), \label{fubini_bounded}
\end{align}
where \eqref{swap_IMQ} is from the integral representation of $k_{I_{m}}^{\text{IMQ}}$ in \eqref{eq:IMQ_kernel} which can be used since $I_{m}\in L^{+}_{1}(\calX)$, \eqref{fubini_bounded} is obtained by using Fubini's theorem and bounded convergence theorem and $z_{\#}P$ denotes the push-forward of the measure under the linear map from $\calX$ to $\calX$ defined as multiplication by the scalar $z$. The integrand in \eqref{fubini_bounded} can be rewritten as
\begin{align*}
	\MMD_{k_{I}^{\text{SE}}}(z_{\#}P,z_{\#}Q)^{2} &  = \int_{\calX}\int_{\calX}k_{I}^{\text{SE}}(x,y)d(z_{\#}P-z_{\#}Q)(x)d(z_{\#}P-z_{\#}Q)(y) \\
& = \int_{\calX}\int_{\calX}k_{I}^{\text{SE}}(zx,zy)d(P-Q)(x)d(P-Q)(y) \\
& = \int_{\calX}\int_{\calX}e^{-\frac{z^{2}}{2}\norm{x-y}_{\calX}^{2}}d(P-Q)(x)d(P-Q)(y) \\
& = \MMD_{k_{z^{2}I}^{\text{SE}}}(P,Q)^{2},
\end{align*}
which makes it clear that $\MMD_{k_{I}^{\text{SE}}}(z_{\#}P,z_{\#}Q)^{2}$ is a continuous, non-negative function of $z$ and equals $0$ if and only if $z = 0$. Using this we deduce
\begin{align*}
	\MMD_{k_{I}^{\text{IMQ}}}(P,Q)^{2} & = \int_{\bbR}\MMD_{k_{I}^{\text{SE}}}(z_{\#}P,z_{\#}Q)^{2}dN_{\sigma^{2}}(z) \\
	& = \int_{\bbR}\MMD_{k_{z^{2}I}^{\text{SE}}}(P,Q)^{2}dN_{\sigma^{2}}(z) > 0,
\end{align*}
since $N_{\sigma^{2}}$ has strictly positive density. This proves that the IMQ-$I$ kernel is characteristic and by the same push-forward argument as in the proof of Theorem \ref{thm:admissible_char} we may conclude that the IMQ-$T$ kernel is characteristic too.
\end{proof}

\subsection{Proofs for Section \ref{sec:inf_dim_MMD}}
\begin{proof}[\textbf{Proof of Proposition \ref{prop:approx_MMD}}]\label{proof:approx_MMD}
For any $i,j$ using the assumption on $k$ we have
\begin{align}
\lvert k(x_{i},y_{j}) - k(\calR\calI x_{i},\calR\calI y_{j})\rvert 
&\leq L\norm{x_{i}-y_{j}-\calR\calI x_{i}+\calR\calI y_{j}}_{\calX}\label{eq:reverse_TI}\\
&\leq L\big(\norm{\calR\calI x_{i}-x_{i}}_{\calX} + \norm{\calR \calI y_{j}-y_{j}}_{\calX}\big),\label{eq:lip_bound}
\end{align}
where \eqref{eq:reverse_TI} is by assumption and \eqref{eq:lip_bound} uses the triangle inequality. Using \eqref{eq:MMD_h_representation} gives
\begin{align}
	&\bigg\lvert\widehat{\MMD}_{k}(X_{n},Y_{n})^{2} - \widehat{\MMD}_{k}(\calR\calI{X}_{n},\calR\calI{Y}_{n})^{2}\bigg\rvert \nonumber\\
	&  \leq \frac{1}{n(n-1)}\sum_{i\neq j}^{n}\lvert h(z_{i},z_{j})-h(\calR\calI {z}_{i},\calR\calI {z}_{j})\rvert \nonumber\\
	& \leq \frac{2L}{n(n-1)}\sum_{i\neq j}^{n}\norm{\calR\calI x_{i}-x_{i}}_{\calX} + \norm{\calR\calI x_{j}-x_{j}}_{\calX}\nonumber\\
	& \qquad + \norm{\calR\calI y_{i}-y_{i}}_{\calX} + \norm{\calR\calI y_{j}-y_{j}}_{\calX}\label{eq:expand_h}\\
	& =\frac{4L}{n}\sum_{i=1}^{n}\norm{\calR\calI x_{i}-x_{i}}_{\calX} + \norm{\calR\calI y_{i}-y_{i}}_{\calX},\label{eq:collect_terms}
\end{align}
where \eqref{eq:expand_h} follows from expanding using the definition of $h$ in Section \ref{sec:RKHS} and using the triangle inequality and \eqref{eq:collect_terms} follows from counting the number of pairs of indices in the sum. 
\end{proof}

\begin{proof}[\textbf{Proof of Corollary \ref{cor:approx_T}}]
    This can be deduced by the Lipschitz constants of $e^{-x^{2}/2}$ and $(x^{2}+1)^{-1/2}$. Then the proof of Proposition \ref{prop:approx_MMD} may be continued in the same manner. 
\end{proof}

\begin{proof}[\textbf{Proof of Theorem \ref{thm:recon_normal}}]
Define the two random variables
\begin{align*}
A_{n} & = n^{\frac{1}{2}}\big(\widehat{\MMD}_{k}(\calR\calI{X}_{n},\calR\calI{Y}_{n})^{2} - \widehat{\MMD}_{k}(X_{n},Y_{n})^{2}\big) \\
B_{n} &= n^{\frac{1}{2}}\big(\widehat{\MMD}_{k}(X_{n},Y_{n})^{2}-\MMD_{k}(P,Q)^{2}\big).
\end{align*} 
It is known $B_{n}\xrightarrow[]{d}\calN(0,\xi)$ \citep[Corollay 16]{Gretton2012} so the proof is complete by Slutsky's theorem if $A_{n}\xrightarrow[]{\bbP}0$. Fix any $\varepsilon > 0$ then by Proposition \ref{prop:scaling}
\begin{align}
	\bbP(\abs{A_{n}}&> \varepsilon)  \leq \bbP\bigg(\frac{4L}{n^{\frac{1}{2}}}\sum_{i=1}^{n}\norm{\calR\calI{x}_{i}-x_{i}}_{\calX} + \norm{\calR\calI{y}_{i}-y_{i}}_{\calX} > \varepsilon\bigg)\nonumber\\
	& \leq \frac{4L}{\varepsilon n^{\frac{1}{2}}}\bbE\bigg[\sum_{i=1}^{n}\norm{\calR\calI{x}_{i}-x_{i}}_{\calX} + \norm{\calR\calI{y}_{i}-y_{i}}_{\calX}\bigg]\label{eq:markov}\\
	& = \frac{4Ln^{\frac{1}{2}}}{\varepsilon}\bbE[\norm{\calR\calI{x}-x}_{\calX} + \norm{\calR\calI{y}-y}_{\calX}]\label{eq:iid}\\
	& \rightarrow 0, \label{eq:0}
\end{align}
where \eqref{eq:markov} is by Markov's inequality, \eqref{eq:iid} is by the assumption that the samples from $P,Q$ and disctretisation $U,V$ are i.i.d. across samples and \eqref{eq:0} is by assumption. 
\end{proof}



\begin{proof}[\textbf{Proof of Theorem \ref{thm:embedding_of_GP}}]\label{proof:embedding_GP} Note that $\Phi_{k_{T}}(N_{a,S})(x) = \Phi_{k_{I}}(N_{Ta,TST})(Tx)$ \citep[Proposition 1.18]{DaPrato2006}. The proof simply uses \citep[Proposition 1.2.8]{DaPrato2002} to calculate the Gaussian integrals. 
\begin{align}
	\Phi_{k_{I}}&(N_{a,S})(x) = \int_{\calX}e^{-\frac{1}{2}\langle x-y,x-y\rangle_{\calX}}dN_{a,S}(y) \nonumber\\
	& = e^{-\frac{1}{2}\langle a-x,a-x\rangle_{\calX}}\int_{\calX}e^{-\frac{1}{2}\langle y,y\rangle_{\calX}} e^{-\langle y,a-x\rangle_{\calX}}dN_{S}(y)\nonumber\\
	& = \det(I+S)^{-\frac{1}{2}}e^{-\frac{1}{2}\langle a-x,a-x\rangle_{\calX}}e^{\frac{1}{2}\norm{(I+S)^{-\frac{1}{2}}S^{\frac{1}{2}}(a-x)}_{\calX}^{2}} \label{eq:prop1.2.8}\\
	& = \det(I+S)^{-\frac{1}{2}}e^{-\frac{1}{2}\langle a-x,a-x\rangle_{\calX}}e^{\frac{1}{2}\langle S^{\frac{1}{2}}(I+S)^{-1}S^{\frac{1}{2}}(a-x),a-x\rangle_{\calX}} \nonumber\\
	& = \det(I + S)^{-\frac{1}{2}}e^{-\frac{1}{2}\langle (I-S^{\frac{1}{2}}(I+S)^{-1}S^{\frac{1}{2}})(a-x),(a-x)\rangle_{\calX}}\nonumber \\
	& = \det(I+S)^{-\frac{1}{2}}e^{-\frac{1}{2}\langle (I+S)^{-1}(x-a),x-a\rangle_{\calX}},
\end{align}
where \eqref{eq:prop1.2.8} is due to \citet[Proposition 1.2.8]{DaPrato2002}. The last equality is due to the Sherman-Morrison-Woodbury identity for operators \citep[Theorem 3.5.6]{Hsing2015}. Substituting in $Ta$ for $a$, $STS$ for $S$ and $Tx$ for $x$ gives the desired expression, as discussed at the start of the proof. 
\end{proof}

\begin{proof}[\textbf{Proof of Theorem \ref{thm:MMD_two_GPs}}]\label{proof:MMD_two_GPs} The idea of the proof is that $\MMD_{k_{T}}(P,Q)^{2}$ is simply double integrals of the kernel with respect to Gaussian measures. One integral was completed in Theorem \ref{thm:embedding_of_GP} and we apply \citet[Proposition 1.2.8]{DaPrato2002} again. Note $\MMD_{k_{T}}(N_{a,S},N_{b,R})^{2} = \MMD_{k_{I}}(N_{Ta,TST},N_{Tb,TRT})^{2}$ so it suffices to do the calculations for $k_{I}$ and substitute the other values in. Also since $k_{T}$ is translation invariant we may without loss of generality assume $a = 0$ and replace $b$ with $a-b$ at the end.
\begin{align}
	\int_{\calX}&\int_{\calX}k_{I}(x,y)dN_{S}(x)dN_{b,R}(y)\nonumber \\
	& = \det(I+S)^{-\frac{1}{2}}\int_{\calX}e^{-\frac{1}{2}\langle(I+S)^{-1}(y-b),y-b\rangle_{\calX}}dN_{R}(y)\label{substitution}\\
	& = \det(I+S)^{-\frac{1}{2}}e^{-\frac{1}{2}\langle(I+S)^{-1}b,b\rangle_{\calX}}\int_{\calX}e^{-\frac{1}{2}\langle(I+S)^{-1}y,y\rangle_{\calX}}e^{\langle y,(I+S)^{-1}b\rangle_{\calX}}dN_{R}(y)\nonumber\\
	& = \det(I+S)^{-\frac{1}{2}}\det\big(I+R^{\frac{1}{2}}(I+S)^{-1}R^{\frac{1}{2}}\big)^{-\frac{1}{2}}\label{use_da_prato}\\
	&\hspace{0.7cm}\times e^{-\frac{1}{2}\langle(I+S)^{-1}b,b\rangle_{\calX}}e^{-\frac{1}{2}\langle (I+S)^{-1}R^{\frac{1}{2}}(I+R^{\frac{1}{2}}(I+S)^{-1}R^{\frac{1}{2}})^{-1}R^{\frac{1}{2}}(I+S)^{-1}b,b\rangle_{\calX}}\nonumber\\
	& = \det(I+S)^{-\frac{1}{2}}\det\big(I+R^{\frac{1}{2}}(I+S)^{-1}R^{\frac{1}{2}}\big)^{-\frac{1}{2}}\label{rearrange}\\
	& \hspace{0.7cm}\times e^{-\frac{1}{2}(I-(I+S)^{-1}R^{\frac{1}{2}}(I+R^{\frac{1}{2}}(I+S)^{-1}R^{\frac{1}{2}})^{-1}R^{\frac{1}{2}})(I+S)^{-1}b,b\rangle_{\calX}}\nonumber\\
	& = \det(I+S)^{-\frac{1}{2}}\det\big(I+R^{\frac{1}{2}}(I+S)^{-1}R^{\frac{1}{2}}\big)^{-\frac{1}{2}}e^{-\frac{1}{2}\langle(I+S+R)^{-1}b,b\rangle_{\calX}},\label{woodbury}
\end{align}
where \eqref{substitution} is obtained by substituting the result of Theorem \ref{thm:embedding_of_GP}, \eqref{use_da_prato} is applying \citet[Proposition 1.2.8]{DaPrato2002}, \eqref{rearrange} is just rearranging terms and \eqref{woodbury} is using the Sherman-Morrison-Woodbury identity for operators \citep[Theorem 3.5.6]{Hsing2015}. The proof is completed by using the expression of $\MMD$ in terms of three double integrals and substituting in the appropriate values of $S,R,b$ inline with the description at the start of the proof. In particular when $b = 0$ and $S = R$ 
\begin{align*}
	\det(I+S)\det\big(I + S^{\frac{1}{2}}(I+S)^{-1}S^{\frac{1}{2}}\big)& =  \det\big((I+S)(I+(I+S)^{-1}S)\big) \\
	& = \det(I+2S),
\end{align*}
by the Sherman-Morrison-Woodbury identity for operators. 
\end{proof}


\begin{proof}[\textbf{Proof of Theorem \ref{thm:variance_of_estimator_mean_shift}}]

A more general result for which Theorem \ref{thm:variance_of_estimator_mean_shift} is a specific case shall be proved. 
\begin{theorem}\label{thm:variance_of_estimator}
	Let $P = N_{a,S}, Q = N_{b,R}$ be two non-degenerate Gaussian measures on $\calX$, $C\in L^{+}(\calX)$ and assume $C,S,R$ all commute then when using the SE-$T$ kernel
	\begin{align*}
	\xi_{1} & = \alpha(T,S,R,a,b) + \alpha(T,R,S,a,b) \\
	\xi_{2} & = \beta(T,S,R,a,b) + \beta(T,R,S,a,b),
	\end{align*}
	where
	\begin{align*}
	\alpha (T,S,R,a,b) & =\det((I+TST)(I+3TST))^{-\frac{1}{2}} - \det\big(I+2TST\big)^{-1} \\
& + \det\big((I+TRT)(I+T(2S+R)T)\big)^{-\frac{1}{2}}e^{-\langle (I+T(2S+R)T)^{-1}T(a-b),T(a-b)\rangle_{\calX}} \\
& -\det\big(I+T(S+R)T\big)^{-1}e^{-\langle (I+T(S+R)T)^{-1}T(a-b),T(a-b)\rangle_{\calX}} \\
& -2\det(\Sigma_{S})^{-\frac{1}{2}}e^{-\frac{1}{2}\langle (I+2TST)\Sigma_{S}^{-1}T(a-b),T(a-b)\rangle_{\calX}} \\
& +2\det\big((I+2TST)(I+T(S+R)T)\big)^{-\frac{1}{2}}e^{-\frac{1}{2}\langle (I+T(S+R)T)^{-1}T(a-b),T(a-b)\rangle_{\calX}}, \\
\vspace{0.5cm}\\
 \beta(T,S,R,a,b) & = \det(I+4TST)^{-\frac{1}{2}} - \det(I+2TST)^{-1} \\
& + \det\big(I + 2T(S +R)T\big)^{-\frac{1}{2}}e^{-\langle (I+2T(S+R)T)^{-1}(a-b),a-b\rangle_{\calX}} \\
& -\det\big(I+T(S+R)T\big)^{-1}e^{-\langle (I+T(S+R)T)^{-1}(a-b),a-b\rangle_{\calX}} \\
& + 4\det\big((I+T(S+R)T)(I+2TST)\big)^{-\frac{1}{2}}e^{-\frac{1}{2}\langle (I+T(S+R)T)^{-1}(a-b),a-b\rangle_{\calX}} \\
& -  4\det(\Sigma_{S})^{-\frac{1}{2}}e^{-\frac{1}{2}\langle (I+2TST)\Sigma_{S}^{-1}(a-b),a-b\rangle_{\calX}},
\end{align*}
 and $\Sigma_{X}  = (I+TST)(I+TRT) + TXT(2I+T(S+R)T)$ for $X\in\{S,R\}$.
\end{theorem}

\begin{proof}[\textbf{Proof of Theorem \ref{thm:variance_of_estimator}}]
	As in the proof of Theorem \ref{thm:MMD_two_GPs} it suffices to consider $T = I$ and $a = 0$. Set $k = k_{I}$ and $ \langle\cdot,\cdot\rangle = \langle\cdot,\cdot\rangle_{\calX}$ for ease of notation. The expression for $\xi_{1}$ is derived first. The simplifications in \citet{Sutherland2019} reveal
	\begin{align*}
		\xi_{1}  & =  \bbE_{x}[\bbE_{x'}[k(x,x')]^{2}] - \bbE_{x,x'}[k(x,x')]^{2} \\
		& +  \bbE_{y}[\bbE_{y'}[k(y,y')]^{2}] - \bbE_{y,y'}[k(y,y')]^{2} \\
		& + \bbE_{x}[\bbE_{y}[k(x,y)]^{2}] - \bbE_{x,y}[k(x,y)]^{2}\\
		& + \bbE_{y}[\bbE_{x}[k(y,x)]^{2}] - \bbE_{y,x}[k(y,x)]^{2}\\
		& - 2\bbE_{x}[\bbE_{x'}[k(x,x')]\bbE_{y}[k(x,y)]] + 2\bbE_{x,x'}[k(x,x')]\bbE_{x,y}[k(x,y)]\\
		& - 2\bbE_{y}[\bbE_{y'}[k(y,y')]\bbE_{x}[k(y,x)]] + 2\bbE_{y,y'}[k(y,y')]\bbE_{x,y}[k(x,y)].
	\end{align*}
	
	To calculate this only three of the terms need to be calculated then the rest are deduced by substituting in different values. For example $\bbE_{x,x'}[k(x,x')]$ can be deduced from the formula for $\bbE_{x,y}[k(x,y)]$ by setting $b = 0$ and $S = R$ since this would make $y$ acts as an independent copy of $x$ in the expectation. The three terms needed are 
	\begin{align}
		&\bbE_{x,y}[k(x,y)] \label{eq:calc_1}\\
		&\bbE_{x}[\bbE_{y}[k(x,y)]^{2}] \label{eq:calc_2}\\
		&\bbE_{x}[\bbE_{x'}[k(x,x')]\bbE_{y}[k(x,y)]].\label{eq:calc_3}
	\end{align}
	
	Expression \eqref{eq:calc_1} was derived in the proof of Theorem \ref{thm:MMD_two_GPs} as
	\begin{align*}
		\bbE_{x,y}[k(x,y)] = \det(I+S+R)^{-\frac{1}{2}}e^{-\frac{1}{2}\langle (I+S+R)^{-1}b,b\rangle}.
	\end{align*}
	
	Next a formula for \eqref{eq:calc_2} is derived. First note $\bbE_{y}[k(x,y)]$ is the content of Theorem \ref{thm:embedding_of_GP}. The rest follows by using \citep[Proposition 1.2.8]{DaPrato2002} and rearranging terms. 
	\begin{align*}
		\bbE_{x}[\bbE_{y}[k(x,y)]^{2}] &= \det(I+R)^{-1}\int_{\calX}e^{-\langle (I+R)^{-1}(x-b),x-b\rangle}dN_{S}(x)\\
		& = \det\big((I+R)(I+R+2S)\big)^{-\frac{1}{2}}e^{-\langle(I+R)^{-1}b,b\rangle}\\
		&\hspace{1cm}\times e^{\langle 2S(I+R)^{-1}(I+2S(I+R)^{-1})^{-1}(I+R)^{-1}b,b\rangle}\\
		& = \det\big((I+R)(I+R+2S)\big)^{-\frac{1}{2}}e^{-\langle (I+R+2S)^{-1}b,b\rangle}.
	\end{align*}
	
	Finally \eqref{eq:calc_3} is derived which involves the longest calculations. The terms in the first expectation are the content of Theorem \ref{thm:embedding_of_GP}. 
	\begin{align*}
		& \bbE_{x}[\bbE_{x'}[k(x,x')]\bbE_{y}[k(x,y)]]  \\
		& = \det\big((I+S)(I+R)\big)^{-\frac{1}{2}}\int_{\calX}e^{-\frac{1}{2}\langle (I+S)^{-1}x,x\rangle}e^{-\frac{1}{2}\langle(I+R)^{-1}(x-b),x-b\rangle}dN_{S}(x)\\
		& = \det\big((I+S)(I+R)\big)^{-\frac{1}{2}} e^{-\frac{1}{2}\langle (I+R)^{-1}b,b\rangle}\\
		&\hspace{2cm}\times \int_{\calX}e^{-\frac{1}{2}\langle ((I+S)^{-1}+(I+R)^{-1})x,x\rangle}e^{-\frac{1}{2}\langle (I+R)^{-1}b,x\rangle}dN_{S}(x)\\
		& = \det\big((I+S)(I+R)\big)^{-\frac{1}{2}}\det\big(I+S((I+S)^{-1}+(I+R)^{-1})\big)^{-\frac{1}{2}}\\
		&\hspace{2cm}\times e^{\frac{1}{2}\langle\big[\big(I+S((I+S)^{-1}+(I+R)^{-1})\big)^{-1}S(I+R)^{-1} - I\big]b,(I+R)^{-1}b\rangle}\\
		& = \det(\Sigma_{S}\big)^{-\frac{1}{2}} e^{-\frac{1}{2}\langle (I+2S)\Sigma_{S}^{-1}b,b\rangle},
	\end{align*}
	where $\Sigma_{S} = (I+S)(I+R) + S(2I+S+R)$. The last equality is obtained by rearranging the terms in the exponent and determinant. Substituting into the formula for $\xi_{1}$ the derivations for \eqref{eq:calc_1}, \eqref{eq:calc_2} and \eqref{eq:calc_3} completes the derivation for $\xi_{1}$. The simplification of $\xi_{2}$ in \citep{Sutherland2019} is
	\begin{align*}
		\xi_{2} & = \bbE_{x,x'}[k(x,x')^{2}] - \bbE_{x,x'}[k(x,x')]^{2} + \bbE_{y,y'}[k(y,y')^{2}] - \bbE_{y,y'}[k(y,y')]^{2} \\
		& + 2\bbE_{x,y}[k(x,y)^{2}] - 2\bbE_{x,y}[k(x,y)]^{2}\\
		& -4\bbE_{x}[\bbE_{x'}[k(x,x')]\bbE_{y}[k(x,y)]] + 4\bbE_{x,x'}[k(x,x')]\bbE_{x,y}[k(x,y)]\\
		& -4\bbE_{y}[\bbE_{y'}[k(y,y')]\bbE_{x}[k(y,x)]] + 4\bbE_{y,y'}[k(y,y')]\bbE_{y,x}[k(y,x)],
	\end{align*}
	only the terms involving $k^{2}$ need to be calculated. Note that $k_{I}^{2} = k_{\sqrt{2}I}$ meaning if $S,R,b$ are replaced by $2S,2R,\sqrt{2}b$ then the formula for \eqref{eq:calc_1} immediately gives a formula for the terms involving $k^{2}$. Combining these derived formulas gives the desired expression for $\xi_{2}$.
	\end{proof}

	Theorem \ref{thm:variance_of_estimator_mean_shift} is recovered by substituting $S = R, a = 0, b = m$ into Theorem \ref{thm:variance_of_estimator}.
	\end{proof}

\begin{proof}[\textbf{Proof of Theorem \ref{thm:weak_convergence}}]
Suppose $P_{n}\xrightarrow{w}P$ then by \citet[Lemma 10]{Simon-Gabriel2018}, which holds in our case since the key intermediate result \citet[Theorem 3.3]{Berg1984} only requires $\calX$ to be a Hausdorff space, we have $\MMD(P_{n},P)\rightarrow 0$. 

Suppose $\MMD(P_{n},P)\rightarrow 0$, by Prokhorov's theorem \citep[Section 5]{Billingsley1971} we know that $\{P_{n}\}_{n=1}^{\infty}$ is relatively compact. Since $k$ is characteristic we know that $\calH_{k}$ is a separating set in the sense of \citet[Chapter 4]{Ethier1986} and $\MMD(P_{n},P)\rightarrow 0$ implies that for every $F\in\calH_{k}$ that $\lim_{n\rightarrow\infty}\int FdP_{n} = \int FdP$ therefore \citet[Lemma 3.4.3]{Ethier1986} applies and we may conclude that $P_{n}\xrightarrow{w}P$.
\end{proof}

\subsection{Proof for Section \ref{sec:implementation}}

\begin{proof}[\textbf{Proof of Proposition \ref{prop:admissible_kernel}}]\label{proof:admissible_kernel}
Suppose $k_{0}$ is ISPD and $C_{k_{0}}$ isn't injective. Then there exists non-zero $x\in L^{2}(\calD)$ such that $C_{k_{0}}x = 0$ so $\int_{\calD}\int_{\calD}x(s)k_{0}(s,t)x(t)dsdt = \langle x,C_{k_{0}}x\rangle_{L^{2}(\calD)} = \langle x,0\rangle_{L^{2}(\calD)} = 0$ contradicting $k_{0}$ being ISPD. Combining \citet[Proposition 5]{Sriperumbudur2011} and \citet[Theorem 9]{Sriperumbudur2010} shows that if $\mu_{k_{0}}$ has full support then $k_{0}$ is ISPD. 
\end{proof}



\end{document}